\documentclass[11pt,reqno]{amsart}
\usepackage{amsmath}
\usepackage{amssymb, latexsym}
\usepackage{hyperref}
\usepackage[margin=1in]{geometry}
\makeatletter
\def\section{\@startsection{section}{1}%
	\z@{.7\linespacing\@plus\linespacing}{.5\linespacing}%
	{\bfseries
		\centering
}}
\def\@secnumfont{\bfseries}
\makeatother
\usepackage{amsmath,amssymb,amsthm,graphicx,amsxtra, setspace}
\usepackage[utf8]{inputenc}
\usepackage{charter}
\usepackage{mathrsfs}
\usepackage{alltt}
\usepackage{relsize}
\usepackage{hyperref}
\usepackage{aliascnt}
\usepackage{tikz}
\usepackage{mathtools}
\usepackage{multicol}
\usepackage{upgreek}
\usepackage{graphicx,type1cm,eso-pic,color}
\allowdisplaybreaks
\usepackage{scalerel,stackengine}
\stackMath
\newcommand\reallywidehat[1]{%
	\savestack{\tmpbox}{\stretchto{%
			\scaleto{%
				\scalerel*[\widthof{\ensuremath{#1}}]{\kern-.6pt\bigwedge\kern-.6pt}%
				{\rule[-\textheight/2]{1ex}{\textheight}}
			}{\textheight}%
		}{0.5ex}}%
	\stackon[1pt]{#1}{\tmpbox}%
}
\newtheorem{theorem}{Theorem}[section]

\newaliascnt{lemma}{theorem}
\newtheorem{lemma}[lemma]{Lemma}
\aliascntresetthe{lemma} 

\newaliascnt{proposition}{theorem}

\aliascntresetthe{proposition}

\newaliascnt{assumption}{theorem}
\newtheorem{assumption}[assumption]{Assumption}
\aliascntresetthe{assumption}

\newaliascnt{corollary}{theorem}

\aliascntresetthe{corollary}

\newaliascnt{definition}{theorem}
\newtheorem{definition}[definition]{Definition}
\aliascntresetthe{definition}

\newaliascnt{example}{theorem}
\newtheorem{example}[example]{Example}
\aliascntresetthe{example}

\newaliascnt{remark}{theorem}
\newtheorem{remark}[remark]{Remark}
\aliascntresetthe{remark}

\newaliascnt{hypothesis}{theorem}

\aliascntresetthe{hypothesis}

\newaliascnt{property}{theorem}

\aliascntresetthe{property}

\let\originalleft\left
\let\originalright\right
\renewcommand{\left}{\mathopen{}\mathclose\bgroup\originalleft}
\renewcommand{\right}{\aftergroup\egroup\originalright}

\makeatletter
\newcommand{\doublewidetilde}[1]{{%
		\mathpalette\double@widetilde{#1}%
}}
\newcommand{\double@widetilde}[2]{%
	\sbox\z@{$\m@th#1\widetilde{#2}$}%
	\ht\z@=.9\ht\z@
	\widetilde{\box\z@}%
}
\makeatother

\renewcommand{\d}{\/\mathrm{d}\/}

\def\w{\textbf{W}^{\varepsilon}_{{\theta}^{\varepsilon}}}

\def\e{\varepsilon}

\def\me{\mathbf{e}}
\def\L{\mathrm{L}}
\def\A{\mathrm{A}}
\def\I{\mathrm{I}}
\def\F{\mathrm{F}}
\def\C{\mathrm{C}}

\def\h{\mathbf{h}}
\def\J{\mathrm{J}}
\def\B{\mathrm{B}}
\def\D{\mathrm{D}}
\def\y{\mathbf{y}}

\def\X{\mathbb{X}}

\def\z{\mathbf{z}}
\def\v{\mathbf{v}}
\def\V{\mathbb{V}}
\def\w{\mathbf{w}}
\def\W{\mathrm{W}}
\def\G{\mathbb{G}}

\def\no{\nonumber}
\def\V{\mathbb{V}}
\def\wi{\widetilde}

\def\U{\mathrm{U}}
\def\P{\mathbb{P}}
\def\u{\mathbf{u}}
\def\H{\mathbb{H}}

\def\p{\mathbf{p}}

\newcommand{\R}{\mathbb{R}}

\renewcommand{\d}{\/\mathrm{d}\/}

\newcommand{\Addresses}{{
		\footnote{

			\noindent \textsuperscript{1}School of Mathematics,
			Indian Institute of Science Education and Research, Trivandrum (IISER-TVM),
			Maruthamala PO, Vithura, Thiruvananthapuram, Kerala, 695 551, INDIA.  \par\nopagebreak \noindent
			\textit{e-mail:} \texttt{tania9114@iisertvm.ac.in}

			\noindent \textsuperscript{2}School of Mathematics, Indian Institute of Science Education and Research, Trivandrum (IISER-TVM),
			Maruthamala PO, Vithura, Thiruvananthapuram, Kerala, 695 551, INDIA.  \par\nopagebreak \noindent
			\textit{e-mail:} \texttt{sheetal@iisertvm.ac.in}

			\noindent \textsuperscript{3}Statistics and Mathematics Unit,
			Indian Statistical Institute(ISI), Bangalore Centre, 8th Mile, Mysore Road, RVCE Post, Bangalore-560 059, INDIA.\par\nopagebreak
			\noindent  \textit{e-mail:} \texttt{maniltmohan@gmail.com, manil$\_$vs@isibang.ac.in}
			
			\noindent \textsuperscript{*}Corresponding author.

			\medskip\noindent
			{\bf Acknowledgments:}  Tania Biswas  would like to thank the  Indian Institute of Science Education and Research, Thiruvananthapuram, for providing financial support and stimulating environment for the research. M. T. Mohan would like to thank the Indian Statistical Institute (ISI) Bangalore Centre, for providing stimulating scientific environment and resources.
			
}}}

\begin{document}

	\title[Optimal Control Problems for Cahn-Hilliard-Navier-Stokes System]{Maximum Principle for Some Optimal Control Problems Governed by 2D Nonlocal Cahn-Hillard-Navier-Stokes Equations	\Addresses	}

	\author[ T. Biswas, S. Dharmatti and M. T. Mohan ]
	{Tania Biswas\textsuperscript{1}, Sheetal Dharmatti\textsuperscript{2*}  and Manil T. Mohan\textsuperscript{3}}

	\maketitle

	\begin{abstract}
		This work is concerned about some optimal control problems associated to the evolution of two isothermal, incompressible, immiscible fluids in a two-dimensional bounded domain. The Cahn-Hilliard-Navier-Stokes model consists of a Navier–Stokes equation governing the fluid velocity field coupled with a convective Cahn–Hilliard equation for the relative concentration of one of the fluids. A distributed optimal control problem is formulated as the minimization of a cost functional subject to the controlled nonlocal Cahn-Hilliard-Navier-Stokes equations. We establish the first-order necessary conditions of optimality by proving Pontryagin maximum principle for optimal control of such system via the seminal Ekeland variational principle. The optimal control is characterized using the adjoint variable. We also study another control problem which is similar to that of data assimilation problems in meteorology of obtaining unknown initial data using optimal control techniques when the underlying system is same as above. 
	\end{abstract}

	
	

	\section{Introduction}\label{sec1}\setcounter{equation}{0}
	This work concerns with control problems related to the evolution of two isothermal, incompressible, immiscible fluids, assuming that the temperature variations are negligible. A general model for such a system is known as \emph{nonlocal Cahn-Hilliard-Navier-Stokes system}   and is given by
	\begin{equation}\label{1.1}
	\left\{
	\begin{aligned}
	\varphi_t + \u\cdot \nabla \varphi &=\text{ div} (m(\varphi) \nabla \mu),\\
	\mu &= a \varphi - \J\ast \varphi + \F'(\varphi)  \ \text{ in } \  \Omega \times (0,T),\\
	\u_t - 2 \text{div } ( \nu(\varphi) \mathrm{D}\u ) + (\u\cdot \nabla )\u + \nabla \uppi &= \mu \nabla \varphi + \mathbf{h}  \ \text{ in } \ \Omega \times (0,T), \\
	\text{div }\u&= 0  \ \text{ in } \  \Omega \times (0,T), \\
	\frac{\partial \mu}{\partial \mathbf{n}} &= \mathbf{0}, \ \u=\mathbf{0} \ \text{ on } \ \partial \Omega \times (0,T),\\
	\u(0) &= \u_0, \  \ \varphi(0) = \varphi _0 \ \text{ in } \ \Omega, 
	\end{aligned}   
	\right.
	\end{equation}
	where  $\u(x,t)$ denotes  the  average velocity of the fluid and  $\varphi(x,t)$ denotes the relative concentration(difference of the concentrations of two fluids) of the fluids. The evolution happens in  $\Omega\times(0,T)$, where $\Omega $ is a  subset of $\R^2$ or $\R^3$. In (\ref{1.1}), $m$ is the mobility parameter, $\mu$ is the \emph{chemical potential}, $\uppi$ is the \emph{pressure}, $\J$ is the \emph{spatial-dependent internal kernel}, $\J \ast \varphi$ denotes the spatial convolution over $\Omega$, $a$ is defined by $a(x) := \int _\Omega \J(x-y) \d y$, $\F$ is a \emph{double-well potential}, $\nu$ is the coefficient of \emph{kinematic viscosity} and $\mathbf{h}$ is the external forcing term acting in the mixture.  (see \cite{weak,unique} for more details). The symmetric part of the gradient of the flow velocity vector is denoted by $\mathrm{D}\u$,  that is, $\mathrm{D}\u$ is the strain tensor $\frac{1}{2}\left(\nabla\u+(\nabla\u)^{\top}\right)$. The density is supposed to be constant and is scaled to one (i.e., matched densities). The system (\ref{1.1}) is  called nonlocal because of the term $\J$, which is averaged over the spatial domain. Ignoring the $\J$ term and replacing  $\mu $ equation in \eqref{1.1}  by $ \mu = \Delta \varphi + \F' (\varphi), $ one obtains the model which is a local version of Cahn-Hilliard-Navier-Stokes' system. 
	The nonlocal version is mathematically challenging, moreover, it is physically more relevant.

	For the past several years, many mathematicians and physicists have studied various models including some simplified versions of the general model mentioned above.  One of the simplified models is to assume the mobility parameter $m$ and the viscosity coefficient $\nu$ to be constants. The system \eqref{1.1} is difficult to tackle because of the non linear coupling between $\mu$ and $\phi$ namely the term $\mu \nabla \varphi$, which gives rise to \emph{Korteweg force}  acting on the fluid.  Even in two dimensions, this term can be less regular than the convective term $(\u\cdot\nabla)\u$ (see \cite{weak}). 
	Recent years have seen a lot of works towards the establishment of the existence and uniqueness  of simplified Cahn-Hilliard-Navier-Stokes systems as well as of the system (\ref{1.1}). In the literature various authors have proved the existence of weak solutions, strong solutions results for these equations (see \cite{exist local chns,weak,unique,strong}) in dimensions 2 and 3. The uniqueness for the system \eqref{1.1} is established in \cite{unique} for dimension 2. However, in dimension 3,  similar to  the case of Navier-Stokes equations, the uniqueness of the weak solution  for the nonlocal Cahn-Hilliard-Navier-Stokes system remains open though the existence is obtained in \cite{weak}. 
	
	Control of partial differential equations in general and fluid flow problems in particular have  several engineering  applications (see for example \cite{fursikov,gunzburger,sritharan}). Controlling fluid flow and turbulence inside a flow in a given physical domain, with various means, for example, body forces, boundary values, temperature (cf. \cite{FART,sritharan} etc), is an interesting problem in fluid mechanics.  Another interesting control problem is to find an optimal controlled initial data with a given external forcing such that a suitable cost functional is minimized (see \cite{sritharan}). Such problems are well studied in meteorology and are an inevitable part of data assimilation problems. The mathematical developments in infinite dimensional nonlinear system theory and partial differential equations in the past several decades, have helped to resolve many control problems for fluid flow equations. Such problems are extensively addressed in \cite{FART,lions,fursikov,gunzburger,sritharan} etc,  to name a few. All these problems mostly deal with the celebrated Navier-Stokes' equations. 
	
	The optimal control problems related to Cahn-Hilliard system  are studied by several mathematicians and some of the results can be found in   \cite{boundary,optimal control for ch eq, optimal control for ch eq with state constraints}. Turning to the optimal control problem for the Cahn-Hilliard-Navier-Stokes' type equations, some works are available in the literature. Optimal control problem with state constraint and robust control for local Cahn-Hilliard-Navier-Stokes' system are investigated in \cite{robust, state const}, respectively. In \cite{BDM}, authors have studied a distributed control problem for Nonlocal Cahn-Hilliard-Navier-Stokes system and established the Pontryagin’s maximum principle. They have characterized the optimal control using the adjoint variable. A similar kind of problem is examined in \cite{ControlCHNS}, where the authors prove the optimality condition under some restrictive condition on the spatial operator $\J$ (see (H4) in \cite{ControlCHNS}). An optimal distributed control problem for two-dimensional nonlocal Cahn-Hilliard-Navier-Stokes systems with degenerate mobility and singular potential is described in \cite{SFMG}. An optimal distributed control of a diffuse interface model of tumor growth is considered in \cite{tumor}. The model studied in \cite{tumor} is  a kind of local Cahn-Hilliard-Navier-Stokes type system with some additional conditions on $\F$.
	
	In  the current work our aim is to study two optimal control problems related to the nonlocal system (\ref{nonlin phi})-(\ref{initial conditions}) (see below) assuming that $m$ and $\nu$ are independent of $\varphi$ and are assumed to be constants (see section \ref{sec2}). As it is necessary to know the solvability and uniqueness of solution to discuss the optimality principle, we restrict ourselves to dimension 2,  where the existence of weak and strong solutions along with uniqueness is known. The two main problems we are considering in this work are:
	\begin{itemize}
		\item [(i)] An enstrophy minimization problem.
		\item [(ii)] A data assimilation problem. 
	\end{itemize} 
	
	The first problem that we consider is a distributed optimal control problem, which aims to minimize the cost functional which involves enstrophy of the velocity field, $\L^2$-energy of the relative concentration and total effort by the control.  As compared to  optimal control problems studied in the literature, here we consider more difficult cost functional which involves enstrophy minimization. This changes the adjoint equations considerably and existence of solution for the linearized adjoint system needs to be established. We use Ekeland's variational principle to prove existence of approximate optimal trajectories and controls. In the seminal paper \cite{EKE}, Ekeland established a variational principle  (see Theorem 1.1, \cite{EKE}) and used it to obtain a necessary condition for the optimal control of a system of ordinary differential equations (see Theorem 7.1, \cite{EKE}). Later, several mathematicians used this principle to establish the first order necessary conditions of optimality for optimal control problems involving infinite dimensional systems as constraints (see Chapter 4, \cite{LIYO} and references therein).  The works \cite{GW,GWLW} made use of this variational principle to establish the Pontryagin maximum principle of optimal control governed by fluid dynamic systems, namely 3D incompressible Navier-Stokes equations. Recently, in the paper \cite{SDMTS}, the authors established the first order necessary conditions of optimality for the optimal control of linearized compressible Navier-Stokes equations using this variational principle.
	
	The concept of \emph{data assimilation} can be defined as the set of statistical techniques that enable us to improve knowledge of the past, present or future system states, jointly using experimental data and the theoretical (a-priori) knowledge on the system. It is also a notion incorporating any method for combining observations of states such as temperature and atmospheric pressure into models used in numerical weather prediction. It is well known that the atmosphere is a fluid and the concept of numerical weather prediction is to sample the state of the fluid at a given time, and use  fluid dynamics equations and thermodynamics equations to estimate the state of the fluid at some time in the future. The process of entering observational data into the model to generate initial conditions is called \emph{initialization}. The second problem which we study is a data assimilation type problem (optimization of the initial velocity field), where the cost functional is $\L^2$-energy  of the velocity, concentration and total efforts by the control in appropriate spaces. A data assimilation problem for the case of 2D incompressible Navier-Stokes equations is described in  \cite{sritharan}.  We establish the  first order optimality condition namely, Pontryagin maximum principle, for the corresponding  cost functional using variational technique and characterize optimal control in terms of adjoint variable. 
In both these control problems,  we need a higher order  regularity of the solution to obtain these results. The unique global strong solution of the system established in \cite{strong}  helps us to achieve this goal.

	The novelty of this work lies in studying optimal control  problem for turbulence minimization of Cahn-Hilliard-Navier-Stokes' system, which exhibits strong nonlinear coupling between velocity field and relative concentration of the fluid. Pontryagin maximum principle is proved via Ekeland's variational principle which first gives the $\varepsilon$ optimal solution and the existence of a minimizer then guarantees the optimal solution. Moreover, to the best of authors knowledge, the initial value estimation problem (data assimilation)  studied in the later part of the paper is completely new and such a problem has not been studied elsewhere for Cahn-Hilliard-Navier-Stokes' system.

	The organization of the paper is as follows. In the next section, we  explain the  two dimensional Cahn-Hilliard-Navier-Stokes system and discuss the necessary functional settings to obtain the existence and uniqueness of (weak and strong) solution for such systems. We also state a few of the well known estimates in the form of lemmas and relevant existence uniqueness theorems available in the literature.  A distributed optimal control problem is formulated in section \ref{se4} and the optimal control is characterized using the adjoint variable. The unique solvability results for the  linearized system and its adjoint is also discussed in this section. We establish the first-order necessary conditions of optimality by proving Pontryagin maximum principle for optimal control of such systems via the Ekeland variational principle (see Theorem \ref{3.10} and \ref{optimal}). In section \ref{se5}, we have incorporated a data assimilation type of problem. We prove the existence of an optimal control and its characterization via adjoint method, where control is assumed to be acting as the initial data (see Theorem \ref{optimal1} and \ref{data}).

	\section{Mathematical Formulation}\label{sec2}\setcounter{equation}{0} In this section, we mathematically formulate the two dimensional Cahn-Hilliard-Navier-Stokes system and discuss the necessary function spaces required to obtain the global solvability results for such systems. We mainly follow the papers \cite{weak,unique} for the mathematical formulation and functional setting. 
	
	A well known model which describes the evolution of an incompressible isothermal mixture of two immiscible fluids is governed by  \emph{Cahn-Hilliard-Navier-Stokes system} (see \cite{weak}). We consider the Cahn-Hilliard-Navier-Stokes equations  which consist of the incompressible Navier–Stokes equations governing the fluid velocity field coupled with a convective Cahn–Hilliard equation for the relative concentration of one of the fluids. Let the average velocity of the fluid  is denoted by $\u(x,t)$ and the relative concentration of the fluid is denoted by $\varphi(x,t),$ for $(x,t)\in\Omega\times(0,T)$, where $\Omega \subset \mathbb{R}^2$ is a bounded domain with sufficiently smooth boundary. Let us denote $\mathbf{n}$ as the unit outward normal to the boundary $\partial\Omega$. 
	Since the aim of this paper is to study some optimal control problems, we first consider the following controlled Cahn-Hilliard-Navier-Stokes system: 
	\begin{subequations}
		\begin{align}
		\varphi_t + \u\cdot \nabla \varphi &= \Delta \mu, \label{nonlin phi}\ \text{ in }\ \Omega\times(0,T),\\
		\mu &= a \varphi - \J\ast \varphi + \F'(\varphi), \label{mu}\ \text{ in }\ \Omega\times(0,T),\\
		\u_t - \nu \Delta \u + (\u\cdot \nabla )\u + \nabla \uppi &= \mu \nabla \varphi + \mathbf{h} +\U, \label{nonlin u}\ \text{ in }\ \Omega\times(0,T),\\
		\text{div }\u&= 0, \ \text{ in }\ \Omega\times(0,T), \label{div zero}\\
		\frac{\partial \mu}{\partial \mathbf{n}} &= \mathbf{0}  , \ \u=\mathbf{0} \ \text{ on } \ \partial \Omega \times (0,T),\label{boundary conditions}\\
		\u(0) &= \u_0, \  \ \varphi(0) = \varphi _0 \ \text{ in } \ \Omega, \label{initial conditions}
		\end{align}   
	\end{subequations}
	where   $\U$ is the \emph{distributed control} acting in the system and the coefficient of kinematic viscosity $\nu$ is a constant. Note that $\nu$ and $\mu$ are independent of $\varphi$ and are assumed to be constants. 
	
	\subsection{Functional setting}
	Let us introduce the following function spaces and operators required for getting the unique global solvability results of the system (\ref{nonlin phi})-(\ref{initial conditions}). Note that on the boundary, we are not prescribing any condition (Dirichlet, Neumann, etc) on the value of  $\varphi$, the relative concentration of one of the  fluid. Instead we impose a Neumann boundary condition for the chemical potential $\mu$ on the boundary. Let us define
	\begin{align*}\G_{\text{div}} &:= \Big\{ \u \in \mathrm{L}^2(\Omega;\R^2) : \text{div }\u=0, \u\cdot \mathbf{n}\big|_{\partial\Omega}=0\Big\}, \\
	\V_{\text{div}} &:= \Big\{\u \in \mathrm{H}^1_0(\Omega;\R^2): \text{div }\u=0\Big\},\\ \mathrm{H}&:=\mathrm{L}^2(\Omega;\R),\ \mathrm{V}:=\mathrm{H}^1(\Omega;\R). \end{align*}
	Let us denote $\| \cdot \|$ and $(\cdot, \cdot),$ the norm and the scalar product, respectively, on both $\mathrm{H}$ and $\G_{\text{div}}$. The norms on $\G_{\text{div}}$ and $\mathrm{H}$ are given by $\|\u\|^2:=\int_{\Omega}|\u(x)|^2\d x$ and $\|\varphi\|^2:=\int_{\Omega}|\varphi(x)|^2\d x$, respectively. The duality between $\V_{\text{div }}$ and its topological dual $\V_{\text{div}}'$ is denoted by $\langle\cdot,\cdot\rangle$. We know that $\V_{\text{div}}$ is endowed with the scalar product 
	$$(\u,\v)_{\V_{\text{div }}}= (\nabla \u, \nabla \v)=2(\mathrm{D}\u,\mathrm{D}\v)\ \text{ for all }\ \u,\v\in\V_{\text{div}}.$$ The norm on $\V_{\text{div }}$ is given by $\|\u\|_{\V_{\text{div }}}^2:=\int_{\Omega}|\nabla\u(x)|^2\d x=\|\nabla\u\|^2$. In the sequel, we use the notations $\mathbb{H}^2(\Omega):=\mathrm{H}^2(\Omega;\mathbb{R}^2)$ and  $\mathrm{H}^2(\Omega):=\mathrm{H}^2(\Omega;\mathbb{R})$ for second order Sobolev spaces.

	\subsection{Linear and nonlinear operators}
	Let us define the \emph{Stokes operator} $\A : \D(\A)\cap  \G_{\text{div}} \to \G_{\text{div}}$ by 
	\begin{align}
	\label{stokes}
	\A:=-\mathrm{P}_{\G}\Delta,\ \D(\A)=\mathbb{H}^2(\Omega) \cap \V_{\text{div}},\end{align} where $\mathrm{P}_{\G} : \mathbb{L}^2(\Omega) \to \G_{\text{div}}$ is the \emph{Helmholtz-Hodge orthogonal projection}. We also have
	$$\langle\A\u, \v\rangle = (\u, \v)_{\V_{\text{div}}} = (\nabla\u, \nabla\v) \  \text{ for all } \ \u \in\D(\A), \v \in \V_{\text{div}}.$$
	It should also be noted that  $\A^{-1} : \G_{\text{div}} \to \G_{\text{div }}$ is a self-adjoint compact operator on $\G_{\text{div}}$ and by
	the classical spectral theorem, there exists a sequence $\lambda_j$ with $0<\lambda_1\leq \lambda_2\leq \lambda_j\leq\cdots\to+\infty$
	and a family of $\mathbf{e}_j \in \D(\A)$ which is orthonormal in $\G_\text{div}$ and such that $\A\mathbf{e}_j =\lambda_j\mathbf{e}_j$. We know that $\u$ can be expressed as $\u=\sum\limits_{j=1}^{\infty}\langle \u,\mathbf{e}_j\rangle \mathbf{e}_j,$ so that $\A\u=\sum\limits_{j=1}^{\infty}\lambda_j\langle \u,\mathbf{e}_j\rangle \mathbf{e}_j$. Thus, it is immediate that 
	\begin{align}
	\|\nabla\u\|^2=\langle \A\u,\u\rangle =\sum_{j=1}^{\infty}\lambda_j|\langle \u,\mathbf{e}_j\rangle|^2\geq \lambda_1\sum_{j=1}^{\infty}|\langle \u,\mathbf{e}_j\rangle|^2=\lambda_1\|\u\|^2,
	\end{align} 
	which is the \emph{Poincar\'e inequality}.

	For $\u,\v,\w \in \V_{\text{div}},$ we define the trilinear operator $b(\cdot,\cdot,\cdot)$ as
	$$b(\u,\v,\w) = \int_\Omega (\u(x) \cdot \nabla)\v(x) \cdot \w(x)\d x=\sum_{i,j=1}^2\int_{\Omega}u_i(x)\frac{\partial v_j(x)}{\partial x_i}w_j(x)\d x,$$
	and the bilinear operator $\B:\V_{\text{div}} \times \V_{\text{div}} \to\V_{\text{div}}'$ defined by,
	$$ \langle \B(\u,\v),\w  \rangle = b(\u,\v,\w) \  \text{ for all } \ \u,\v,\w \in \V_\text{{div}}.$$
	An integration by parts yields, 
	\begin{equation}\label{best}
	\left\{
	\begin{aligned}
	b(\u,\v,\v) &= 0, \ \text{ for all } \ \u,\v \in\V_\text{{div}},\\
	b(\u,\v,\w) &=  -b(\u,\w,\v), \ \text{ for all } \ \u,\v,\w\in \V_\text{{div}}.
	\end{aligned}
	\right.\end{equation}
	Using (\ref{best}), H\"older and Ladyzhenskaya inequalities (see Lemma \ref{lady} below), for every $\u,\v,\w \in \V_{\text{div}},$ we have the following estimate:
	\begin{align}
	|b(\u,\v,\w)|&= |b(\u,\w,\v)|\leq \|\u\|_{\mathbb{L}^4}\|\nabla\w\|\|\v\|_{\mathbb{L}^4}\leq \sqrt{2}\|\u\|^{1/2}\| \nabla \u\|^{1/2}\|\v\|^{1/2}\| \nabla \v\|^{1/2}\| \w\|_{\V_{\text{div }}}.
	\end{align}
	Thus for all $\u\in\V_{\text{div}},$ we have
	\begin{align}
	\label{be}
	\|\B(\u,\u)\|_{\V_{\text{div}}'}\leq \sqrt{2}\|\u\|\|\nabla\u\|\leq \sqrt{\frac{2}{\lambda_1}}\|\u\|_{\V_{\text{div}}}^2 ,
	\end{align}
	by using the Poincar\'e inequality. 
	For more details about the linear and nonlinear operators defined above, we refer the readers to \cite{Te}.

	For every $f \in \mathrm{V}',$ we denote $\overline{f}$ the average of $f$ over $\Omega$, i.e., $\overline{f} := |\Omega|^{-1}\langle f, 1 \rangle$, where $|\Omega|$ is the Lebesgue measure of $\Omega$. 
	Let us also introduce the spaces (see \cite{unique})
	\begin{align*}\mathrm{V}_0 &= \{ v \in \mathrm{V} \ : \ \overline{v} = 0 \},\\
	\mathrm{V}_0' &= \{ f \in \mathrm{V}' \ : \ \overline{f} = 0 \},\end{align*}
	and the operator $\mathcal{A} : \mathrm{V} \rightarrow \mathrm{V}'$ is defined by
	\begin{align*}\langle \mathcal{A} u ,v \rangle := \int_\Omega \nabla u(x) \cdot \nabla v(x) \d x \  \text{for all } \ u,v \in \mathrm{V}.\end{align*}
	Clearly $\mathcal{A}$ is linear and it maps $\mathrm{V}$ into $\mathrm{V}_0'$ and its restriction $\mathcal{B}$ to $\mathrm{V}_0$ onto $\mathrm{V}_0'$ is an isomorphism.   
	We know that for every $f \in \mathrm{V}_0'$, $\mathcal{B}^{-1}f$ is the unique solution with zero mean value of the \emph{Neumann problem}:
	$$
	\left\{
	\begin{array}{ll}
	- \Delta u = f, \  \mbox{ in } \ \Omega, \\
	\frac{\partial u}{\partial\mathbf{n}} = 0, \ \mbox{ on } \  \partial \Omega.
	\end{array}
	\right.
	$$
	In addition, we have
	\begin{align} \langle \mathcal{A}u , \mathcal{B}^{-1}f \rangle &= \langle f ,u \rangle, \ \text{ for all } \ u\in \mathrm{V},  \ f \in \mathrm{V}_0' , \label{bes}\\
	\langle f , \mathcal{B}^{-1}g \rangle &= \langle g ,\mathcal{B}^{-1}f \rangle = \int_\Omega \nabla(\mathcal{B}^{-1}f)\cdot \nabla(\mathcal{B}^{-1}g)\d x, \ \text{for all } \ f,g \in \mathrm{V}_0'.\label{bes1}
	\end{align}
	Note that $\mathcal{B}$ can be also viewed as an unbounded linear operator on $\mathrm{H}$ with
	domain $$\D(\mathcal{B}) = \left\{v \in \mathrm{H}^2 : \frac{\partial v}{\partial\mathbf{n}}= 0\text{ on }\partial\Omega \right\}.$$

	\subsection{Some useful inequalities} Let us now list some useful  interpolation inequalities, which we use frequently in the sequel. Let us first give a version of the Gagliardo-Nirenberg inequality which holds true for all $\u\in\mathbb{W}^{1,p}_0(\Omega;\R^n),p\geq 1$. 
	\begin{lemma}[Gagliardo-Nirenberg inequality, Theorem 2.1, \cite{ED}] \label{gn}
		Let $\Omega\subset\R^n$ and $\u\in\mathbb{W}^{1,p}_0(\Omega;\R^n),p\geq 1$. Then for any fixed number $1\leq q,r\leq \infty$, there exists a constant $C>0$ depending only on $n,p,q$ such that 
		\begin{align}\label{gn0}
		\|\u\|_{\mathbb{L}^r}\leq C\|\nabla\u\|_{\mathbb{L}^p}^{\theta}\|\u\|_{\mathbb{L}^q}^{1-\theta},\ \theta\in[0,1],
		\end{align}
		where the numbers $p, q, r$ and $\theta$ satisfy the relation
		$$\theta=\left(\frac{1}{q}-\frac{1}{r}\right)\left(\frac{1}{n}-\frac{1}{p}+\frac{1}{q}\right)^{-1}.$$
	\end{lemma}
	The particular cases of Lemma \ref{gn} are well known inequalities, due to Ladyzhenskaya (see Lemma 1 and 2, Chapter 1, \cite{OAL}), which is given below.
	\begin{lemma}[Ladyzhenskaya inequality]\label{lady}
		For $\u\in\ \C_0^{\infty}(\Omega;\R^n), n = 2, 3$, there exists a constant $C$ such that
		\begin{align}
		\|\u\|_{\mathbb{L}^4}\leq C^{1/4}\|\u\|^{1-\frac{n}{4}}\|\nabla\u\|^{\frac{n}{4}},\text{ for } n=2,3,
		\end{align}
		where $C=2,4$ for $n=2,3$ respectively. 
	\end{lemma}

	We also use the following general version of the Gagliardo-Nirenberg interpolation inequality and Agmon's inequality for higher order estimates. For functions $\u: \Omega\to\R$ defined on a bounded Lipschitz domain $\Omega\subset\R^n$, the Gagliardo-Nirenberg interpolation inequality is given by: 
	\begin{lemma}[Gagliardo-Nirenberg interpolation inequality, Theorem 1, \cite{LN}]\label{GNI} Let $\Omega\subset\R^n$, $\u\in\mathbb{W}^{m,p}(\Omega;\R^n), p\geq 1$ and fix $1 \leq q, r \leq \infty$ and a natural number $m$. Suppose also that a real number $\theta$ and a natural number $j$ are such that
		\begin{align}
		\label{theta}
		\theta=\left(\frac{j}{n}+\frac{1}{q}-\frac{1}{r}\right)\left(\frac{m}{n}-\frac{1}{p}+\frac{1}{q}\right)^{-1}\end{align}
		and
		$\frac{j}{m} \leq \theta \leq 1.$ Then for any $\u\in\mathbb{W}^{m,p}(\Omega;\R^n),$ we have 
		\begin{align}\label{gn1}
		\|\nabla^j\u\|_{\mathbb{L}^r}\leq C\left(\|\nabla^m\u\|_{\mathbb{L}^p}^{\theta}\|\u\|_{\mathbb{L}^q}^{1-\theta}+\|\u\|_{\mathbb{L}^s}\right),
		\end{align}
		where $s > 0$ is arbitrary and the constant $C$ depends upon the domain $\Omega,m,n$. 
	\end{lemma}
	Note that  for $\u\in\mathbb{W}_0^{1,p}(\Omega;\R^n)$, Lemma \ref{gn} is a special case of the above inequality, since for $j=0$, $m=1$ and $\frac{1}{s}=\frac{\theta}{p}+\frac{1-\theta}{q}$ in \eqref{gn1}, and application of the Poincar\'e inequality yields (\ref{gn0}). It should also be noted that \eqref{gn1} can also be written as 
	\begin{align}\label{gn2}
	\|\nabla^j\u\|_{\mathbb{L}^r}\leq C\|\u\|_{\mathbb{W}^{m,p}}^{\theta}\|\u\|_{\mathbb{L}^q}^{1-\theta}.
	\end{align}
	By taking $j=1$, $r=4$, $n=m=p=q=s=2$ in (\ref{theta}), we get $\theta=\frac{3}{4},$ and 
	\begin{align}\label{gu}
	\|\nabla\u\|_{\mathbb{L}^4}\leq C\left(\|\Delta\u\|^{3/4}\|\u\|^{1/4}+\|\u\|\right).
	\end{align} 
	Using Young's inequality, we obtain 
	\begin{align}\label{gu1}
	\|\nabla\u\|_{\mathbb{L}^4}^2\leq  C(\|\Delta\u\|^{3/2}\|\u\|^{1/2}+\|\u\|^2)\leq C(\|\Delta\u\|^{2}+\|\u\|^2)\leq C\|\u\|_{\mathbb{H}^2}^2.
	\end{align}
	\begin{lemma}[Agmon's inequality, Lemma 13.2, \cite{SA}]
		For any 
		$\u\in \H^{s_2}(\Omega;\R^n),$ choose  $s_{1}$ and $s_{2}$ such that 
		$ s_1< \frac{n}{2} < s_2$. Then, if 
		$0< \alpha < 1$ and 
		$\frac{n}{2} = \alpha s_1 + (1-\alpha)s_2$, the following inequality holds 
		$$ \|\u\|_{\mathbb{L}^\infty}\leq C \|\u\|_{\H^{s_1}}^{\alpha} \|\u\|_{\H^{s_2}}^{1-\alpha}.$$
	\end{lemma}
	For $\u\in\H^2(\Omega)\cap\H_0^1(\Omega)$, the \emph{Agmon's inequality} in 2D states that there exists a constant 
	$C>0$ such that
	\begin{align}\label{agm}
	\|\u\|_{\mathbb{L}^{\infty}}\leq C\|\u\|^{1/2}\|\u\|_{\H^2}^{1/2}\leq C\|\u\|_{\H^2}.
	\end{align}
	The inequality (\ref{agm}) can also be obtained from (\ref{gn1}), by taking $j=0$, $r=\infty$, $m=2=p=q=2$, so that we have $\theta=\frac{1}{2}$.

	\subsection{Weak and strong solution of the system \eqref{nonlin phi}-\eqref{initial conditions}}
	Now we state the results regarding the existence theorem and uniqueness of  weak and strong solution for the uncontrolled nonlocal Cahn-Hilliard-Navier-Stokes system  given by \eqref{nonlin phi}-\eqref{initial conditions} with $\U=0$.  
	Let us first make the following assumptions:
	\begin{assumption}\label{prop of F and J} Let
		$\mathrm{J}$ and $\mathrm{F}$ satisfy:
		\begin{enumerate}
			\item [(1)] $ \J \in \W^{1,1}(\mathbb{R}^2;\R), \  \J(x)= \J(-x) \; \text {and} \ a(x) = \int\limits_\Omega \J(x-y)\d y \geq 0,$ a.e., in $\Omega$.
			\item [(2)] $\F \in \C^{2}(\mathbb{R})$ and there exists $C_0 >0$ such that $\F''(s)+ a(x) \geq C_0$, for all $s \in \mathbb{R}$, a.e., $x \in \Omega$.
			\item [(3)] Moreover, there exist $C_1 >0$, $C_2 > 0$ and $q>0$ such that $\F''(s)+ a(x) \geq C_1|s|^{2q} - C_2$, for all $s \in \mathbb{R}$, a.e., $x \in \Omega$.
			\item [(4)] There exist $C_3 >0$, $C_4 \geq 0$ and $r \in (1,2]$ such that $|\F'(s)|^r \leq C_3|\F(s)| + C_4,$ for all $s \in \mathbb{R}$.
		\end{enumerate}
	\end{assumption}
	\begin{remark}\label{remark J}
		Assumption $\J \in \W^{1,1}(\mathbb{R}^2;\R)$ can be weakened. Indeed, it can be replaced by $\J \in \W^{1,1}(\B_\delta;\R)$, where $\B_\delta := \{z \in \mathbb{R}^2 : |z| < \delta \}$ with $\delta := \text{diam}(\Omega)$, or also by
		\begin{eqnarray} \label{Estimate J}
		\sup_{x\in \Omega} \int_\Omega \left( |\J(x-y)| + |\nabla \J(x-y)| \right) dy < +\infty .
		\end{eqnarray}
	\end{remark}
	
	\begin{remark}\label{remark F}
		Since $\F(\cdot)$ is bounded from below, it is easy to see that Assumption \ref{prop of F and J} (4) implies that $\F(\cdot)$ has a polynomial growth of order $r'$, where $r' \in [2,\infty)$ is the conjugate index to $r$. Namely, there exist $C_5$ and $C_6 \geq 0$ such that 
		\begin{eqnarray}\label{2.9}
		|\F(s)| \leq C_5|s|^{r'} + C_6, \text{ for all } s \in \R.
		\end{eqnarray}
		Observe that Assumption \ref{prop of F and J} (4) is fulfilled by a potential of arbitrary polynomial growth. For example, (2)-(4) are satisfied for the case of the well-known double-well potential $\F = (s^2 - 1)^2$.
	\end{remark}
	Now, we examine the global solvability results for uncontrolled system  \eqref{nonlin phi}-\eqref{initial conditions} i.e. with $\U=0$ under the Assumption \ref{prop of F and J}. The existence of a weak solution for such a system in two and three dimensional bounded domains is established in Theorem 1, Corollaries 1 and 2, \cite{weak}, and the uniqueness for two dimensional case  is  obtained in Theorem 2, \cite{unique}. Under some extra assumptions on $\F$ and $\J$, and enough regularity on the initial data and external forcing,  a unique global strong solution is established in Theorem 2, \cite{strong}.
	\begin{definition}[weak solution]
		Let $\u_0\in\G_{\text{div}}$, $\varphi_0\in\H$ with $\F(\varphi_0)\in\mathrm{L}^1(\Omega)$ and $0<T<\infty$ be given. Then $(\u,\varphi)$ is a \emph{weak solution} to the uncontrolled system ($\U=0$) \eqref{nonlin phi}-\eqref{initial conditions} on $[0,T]$ corresponding to initial conditions $\u_0$ and $\varphi_0$ if 
		\begin{itemize}
			\item [(i)] $\u,\varphi$ and $\mu$ satisfy 
			\begin{equation}\label{sol}
			\left\{
			\begin{aligned}
			&	\u \in \mathrm{L}^{\infty}(0,T;\G_{\text{div}}) \cap \mathrm{L}^2(0,T;\V_{\text{div}}),  \\ 
			&	\u_t \in \mathrm{L}^{2-\gamma}(0,T;\V_{\text{div}}'), \ \text{ for all  }\gamma \in (0,1),  \\
			&	\varphi \in \mathrm{L}^{\infty}(0,T;\mathrm{H}) \cap \mathrm{L}^2(0,T;\mathrm{V}),   \\
			&	\varphi_t \in \mathrm{L}^{2-\delta}(0,T;\mathrm{V}'),\ \text{ for all  }\delta \in (0,1), \\
			&	\mu \in \mathrm{L}^2(0,T;\mathrm{V}),
			\end{aligned}
			\right.
			\end{equation}
			\item [(ii)]  for every $\psi\in\mathrm{V}$, every $\v \in \V_{\text{div}}$, if we define 		             $\rho$ by 
			\begin{align}
			\rho(x,\varphi):=a(x)\varphi+\F'(\varphi),
			\end{align} and for almost any $t\in(0,T)$, we have
			\begin{align}
			\langle \varphi_t,\psi\rangle +(\nabla\rho,\nabla\psi)&=\int_{\Omega}(\u\cdot\nabla\psi)\varphi\d x+\int_{\Omega}(\nabla\J*\varphi)\cdot\nabla\psi\d x,\\
			\langle \u_t,\v\rangle +\nu(\nabla\u,\nabla\v)+b(\u,\v,\w)&=-\int_{\Omega}(\v\cdot\nabla\mu)\varphi\ d x+\langle\h,\v\rangle.
			\end{align}
			\item [(iii)] Moreover, the following initial conditions hold in the weak sense 
			\begin{align}
			\u(0)=\u_0,\ \varphi(0)=\varphi_{0},
			\end{align}
			i.e., for every $\v \in \V_{\text{div}}$, we have $(\u(t),\v) \to (\u_0,\v)$ as $t\to 0$, and for every $\chi \in \mathrm{V}$, we have $(\varphi(t),\chi) \to (\varphi_0,\chi)$ as $t\to  0$.
		\end{itemize}
	\end{definition}
	
	\begin{theorem}[Existence, Theorem 1, Corollaries 1 and 2, \cite{weak}]\label{exist}
		Let the Assumption \ref{prop of F and J} be satisfied. Let $\u_0 \in \G_{\text{div}}$, $\varphi_0 \in \mathrm{H}$ be such that $\F(\varphi_0) \in \mathrm{L}^1(\Omega)$ and $\mathbf{h} \in \mathrm{L}^2_{\text{loc}}([0,\infty), \V_{\text{div}}')$ are given. Then, for every given $T>0$, there exists a weak solution $(u,\varphi)$ to the uncontrolled system  \eqref{nonlin phi}-\eqref{initial conditions} such that \eqref{sol} is satisfied.
		Furthermore, setting
		\begin{align}\mathscr{E}(\u(t),\varphi(t)) = \frac{1}{2} \|\u(t)\|^2 + \frac{1}{4} \int_\Omega \int_\Omega \J(x-y) (\varphi(x,t) - \varphi(y,t))^2 \d x \d y + \int_\Omega \F(\varphi(x,t))\d x,\end{align}
		the following energy estimate holds for almost any $t>0$:
		\begin{align}\label{energy}
		\mathscr{E}(\u(t),\varphi(t)) + \int_0^t \left(\nu \| \nabla \u(s)\|^2 + \| \nabla\mu(s) \|^2 \right)\d s \leq \mathscr{E}(\u_0,\varphi_0) + \int_0^t \langle \mathbf{h}(s), \u(s) \rangle\d s,\end{align}
		and the weak solution $(\u,\varphi)$ satisfies the following energy identity,
		$$\frac{\d}{\d t}\mathscr{E}(\u(t),\varphi(t)) + \nu \|\nabla \u(t) \|^2+ \| \nabla \mu(t) \|^2 = \langle \mathbf{h}(t) , \u(t) \rangle.$$
	\end{theorem} 
	\begin{remark}
		It can also be proved that $\u\in\C([0,T];\G_{\text{div}})$ and $\varphi\in \C([0,T];\mathrm{H})$., for more details refer Remark 7 of \cite{weak}.
	\end{remark}
	
	\begin{remark}\label{rem2.5}
		We denote by $\mathbb{Q}$ a continuous monotone increasing function with respect to each of its arguments. As a consequence of energy inequality (\ref{energy}), we have the following bound:
		\begin{align}
		&	\|\u\|_{\mathrm{L}^{\infty}(0,T;\G_{\text{div}}) \cap \mathrm{L}^2(0,T;\V_{\text{div}})} + \|\varphi \|_{\mathrm{L}^{\infty}(0,T;\mathrm{H}) \cap \mathrm{L}^2(0,T;\mathrm{V})} + \|\F(\varphi)\|_{\mathrm{L}^{\infty}(0,T;\mathrm{H})} \nonumber\\&\qquad\leq \mathbb{Q}\left(\mathscr{E}(\u_0,\varphi_0),\|\mathbf{h}\|_{\mathrm{L}^2(0,T;\V_{\text{div}}')}\right),
		\end{align}
		where $\mathbb{Q}$ also depends on $\F$, $\J$, $\nu$ and $\Omega$.
	\end{remark}
	
	\begin{theorem}[Uniqueness, Theorem 2, \cite{unique}]\label{unique}
		Suppose that the Assumption \ref{prop of F and J} is satisfied. Let $\u_0 \in \G_{\text{div}}$, $\varphi_0 \in \mathrm{H}$ with $\F(\varphi_0) \in \mathrm{L}^1(\Omega)$ and $\mathbf{h} \in \mathrm{L}^2_{\text{loc}}([0,\infty);\V_{\text{div}}')$ be given. Then, the weak solution $(\u,\varphi)$ corresponding to $(\u_0,\varphi_0)$ and given by Theorem \ref{exist} is unique. 
		
		Furthermore, for $i=1,2$, let $\z_i :=  (\u_i,\varphi_i)$ be two weak solutions corresponding to two initial data $\z_{0i} :=  (\u_{0i},\varphi_{0i})$ and external forces $\h_i$, with $\u_{0i} \in \G_{\text{div}}$, $\varphi_{0i} \in \mathrm{H}$ with $\F(\varphi_{0i}) \in \mathrm{L}^1(\Omega)$ and $\h_i \in \mathrm{L}^2_{\text{loc}}([0,\infty);\V_{\text{div}}')$. Then the following continuous dependence estimate holds:
		\begin{align*}
		&\| \u_2(t) - \u_1(t) \|^2 + \| \varphi_2(t) - \varphi_1(t) \|^2_{\mathrm{V}'} + \int_0^t \left( \frac{C_0}{2} \| \varphi_2(s) - \varphi_1(s) \|^2 + \frac{\nu}{4} \|\nabla( \u_2(s) - \u_1(s)) \|^2 \right) \d s \\
		&\leq \left( \|\u_2(0) - \u_1(0) \|^2 + \| \varphi_2(0) - \varphi_1(0) \|^2_{\mathrm{V}'} \right) \Lambda_0(t) \\
		&\quad + \| \overline{\varphi}_2(0) - \overline{\varphi_1}(0)\| \mathbb{Q} \left( \mathcal{E}(z_{01}),\mathcal{E}(z_{02}),\| \h_1 \|_{\mathrm{L}^2(0,t;\V_{div}')},\| \h_2 \|_{\mathrm{L}^2(0,t;\V_{div}')} \right)  \Lambda_1(t) \\&\quad+ \| \h_2 - \h_1 \|^2_{\mathrm{L}^2(0,T;\V_{\text{div}}')} \Lambda_2(t) ,
		\end{align*} 
		for all $t \in [0,T]$, where $\Lambda_0(t)$, $\Lambda_1(t)$ and $\Lambda_2(t)$ are continuous functions which depend on the norms of the two solutions. The functions $\mathbb{Q}$ and $\Lambda_i(t)$ also depend on $\F$, $\J$ and $\Omega$.
	\end{theorem}
	
	For the further analysis of this paper, we also need the existence of a unique strong solution to the uncontrolled system  \eqref{nonlin phi}-\eqref{initial conditions}. The following theorem established in Theorem 2, \cite{strong} gives the existence and uniqueness of the strong solution   the uncontrolled system  \eqref{nonlin phi}-\eqref{initial conditions}. 
	\begin{definition}[Strong solution]
		Let $\u_0\in\V_{\text{div}}$, $\varphi_0\in\mathrm{V} \cap \mathrm{L}^{\infty}(\Omega)$, $\h,\in \mathrm{L}^2_{\text{loc}}(0,\infty;\G_{\text{div}})$ and $0<T<\infty$ be given. Then $(\u,\varphi)$ is a \emph{strong solution} to the uncontrolled system \eqref{nonlin phi}-\eqref{initial conditions} on $[0,T]$ corresponding to initial conditions $\u_0$ and $\varphi_0$ if 
	$\u,\varphi$ and $\mu$ satisfy 
			\begin{equation}\label{sol1}
			\left\{
			\begin{aligned}
			&	\u \in \mathrm{L}^{\infty}(0,T;\V_{\text{div}})\cap \mathrm{L}^2(0,T;\H^2),  \\ 
			&	\u_t \in \mathrm{L}^2(0,T;\G_{\text{div}}),  \\
			&	\varphi \in \mathrm{L}^{\infty}((0,T) \times\Omega)\cap \mathrm{L}^{\infty}(0,T;\mathrm{V}),   \\
			&	\varphi_t \in \mathrm{L}^2(0,T;\mathrm{H}), \\
			&	\mu \in \mathrm{L}^2(0,T;\mathrm{V}).
			\end{aligned}
			\right.
			\end{equation}
			\end{definition}
	\begin{theorem}[Global Strong Solution, Theorem 2, \cite{strong}]\label{strongsol}
		Let $\h,\U\in \mathrm{L}^2(0,T;\G_{\text{div}})$, $\u_0\in\V_{\text{div}}$, $\varphi_0\in\mathrm{V}\cap\mathrm{L}^{\infty}(\Omega)$ be given and the Assumption \ref{prop of F and J} be satisfied. Then, for a given $T > 0$, there exists \emph{a unique weak solution} $(\u,\varphi)$ to the system \eqref{nonlin phi}-\eqref{initial conditions} such that (\ref{sol1}) is satisfied. \\
		Furthermore, suppose in addition that $\F \in\C^3(\mathbb{R}),\ a \in\mathrm{H}^2(\Omega)$ and that $\varphi_0\in \mathrm{H}^2(\Omega)$. Then, the uncontrolled system \eqref{nonlin phi}-\eqref{initial conditions} admits \emph{a unique strong solution} on $[0, T ]$ satisfying (\ref{sol1}) and also
		\begin{equation}
		\label{0.23}\left\{
		\begin{aligned}
		&\varphi \in\mathrm{L}^{\infty}(0,T;\mathrm{W}^{1,p}),\ 2\leq p<\infty,\\ &\varphi_t\in \mathrm{L}^{\infty}(0,T;\mathrm{H})\cap\mathrm{L}^2(0,T;\mathrm{V}).
		\end{aligned}
		\right.
		\end{equation}
		If $\J \in\mathrm{W}^{2,1}(\mathbb{R}^2;\R)$, we have in addition
		\begin{equation}\label{0.24}\varphi\in \mathrm{L}^{\infty}(0, T ; \mathrm{H}^2). \end{equation}
	\end{theorem}
	
	\begin{remark}\label{rem2.12}
		The regularity properties given in (\ref{sol1})-(\ref{0.24}) imply that
		\begin{align}\label{ues}
		\u\in\C  ([0, T] ; \V_{\text{div}} )), \ \varphi \in \C ( [0, T ]; \mathrm{V} )  \cap \C_w ( [0, T ]; \mathrm{H}^ 2).
		\end{align}
			\end{remark}
	
	\section{Optimal Control Problem}\label{se4}\setcounter{equation}{0}
	In this section, we formulate a distributed optimal control problem  as the minimization of a suitable cost functional subject to the controlled nonlocal Cahn-Hilliard-Navier-Stokes equations. We consider a problem of controlling relative concentration, fluid flow and turbulence inside a flow in a bounded domain $\Omega$. We characterize the optimal control in terms of the adjoint variable. A similar kind of optimal control problem is considered in \cite{BDM} with a different cost functional and the techniques used to establish the first order necessary conditions for optimality in both papers are also entirely different. 
	\subsection{The optimal control problem formulation}
	Let $\u_d\in\mathrm{L}^2(0,T;\V_{\text{div }})$, $\varphi_d\in\mathrm{L}^2(0,T;\mathrm{H}),$ $\u_f\in\G_{\text{div }}$ and $\varphi_f\in\mathrm{H}$ be some given target functions. 
	The main goal  is to establish the existence of an optimal control that minimizes the cost functional 
	\begin{equation}\label{cost1}
	\begin{aligned}
	\mathcal{J}(\u,\varphi,\U) &:= \frac{1}{2} \int_0^T \|\nabla\times(\u(t)-\u_d(t))\|^2 \d t+ \frac{1}{2} \int_0^T \|\varphi(t) -\varphi_d(t)\|^2 \d t\\&\quad+\frac{1}{2}\|\u(T)-\u_f\|^2+\frac{1}{2}\|\varphi(T)-\varphi_f\|^2+ \frac{1}{2} \int_0^T\|\U(t)\|^2\d t,\end{aligned}
	\end{equation}
	subject to the constraint  \eqref{nonlin phi}-\eqref{initial conditions} and prove the first order necessary condition of optimality using the Ekeland variational principle. Note that the cost functional given in (\ref{cost1}) is the sum of total enstrophy of the average velocity field, $\mathbb{L}^2$-energy of the relative concentration and total effort by the distributed controls. An easy calculation shows that 
	\begin{align*}
	\|\mathrm{A}^{1/2}\u\|=\|\nabla\times\u\|=\|\nabla\u\|,
	\end{align*}
	where $\mathrm{A}$ is the Stokes operator defined in \eqref{stokes}. Thus, the associated cost functional given in (\ref{cost1}) becomes 
	\begin{equation}\label{cost}
	\begin{aligned}
	\mathcal{J}(\u,\varphi,\U) &:= \frac{1}{2} \int_0^T \|\nabla(\u(t)-\u_d(t))\|^2 \d t+ \frac{1}{2} \int_0^T \|\varphi(t) -\varphi_d(t)\|^2 \d t\\&\quad+\frac{1}{2}\|\u(T)-\u_f\|^2+\frac{1}{2}\|\varphi(T)-\varphi_f\|^2+ \frac{1}{2} \int_0^T\|\U(t)\|^2\d t,\end{aligned}
	\end{equation}
	where $\u_d(\cdot)$, $\u_f$, $\varphi_d(\cdot)$ and $\varphi_f$ are the desired states. For the rest of the paper, we consider the formulation given by \eqref{cost}
	which would intern minimize the enstrophy. We consider set of admissible controls to be the space consisting of controls $\U \in\mathrm{L}^{2}(0,T;\G_{\text{div}})$, and denote it by $\mathscr{U}_{\text{ad}}$ .
	
\begin{example} \label{example}
 As an another example one can take $\mathscr{U}_{\text{ad}}$ to be $\mathrm{L}^{2}(0,T;B(0,R))$ where $B(0,R)$ is a ball of radius $R$ and centre at $0$ in $\G_{\text{div}}$ . 
\end{example}

	For our further discussion we take $\mathscr{U}_{\text{ad}}$  to be whole space $ \mathrm{L}^{2}(0,T;\G_{\text{div}})$.

	\begin{definition}[Admissible Class]\label{definition 1}
		The \emph{admissible class} $\mathscr{A}_{\text{ad}}$ of triples $(\u,\varphi,\U)$ is defined as the set of states $(\u,\varphi)$ with initial data $\u_0 \in \V_{\text{div}}$, $\varphi_0 \in \mathrm{V} \cap \mathrm{L}^{\infty}$, $a\in\mathrm{H}^2(\Omega)$ and $\F\in\C^3(\R)$, solving the system \eqref{nonlin phi}-\eqref{initial conditions} with control $\U \in \mathscr{U}_{ad}$.  That is,
		\begin{align*}
		\mathscr{A}_{\text{ad}}:=\Big\{(\u,\varphi,\U) :(\u,\varphi)\text{ is \emph{a unique strong solution} of }\eqref{nonlin phi}\text{-}\eqref{initial conditions}  \text{ with control }\U\Big\}.
		\end{align*}
	\end{definition}
	In view of the above definition, the optimal control problem we are considering can be formulated as
	\begin{align}\label{control problem}\tag{OCP}
	\min_{ (\u,\varphi,\U) \in \mathscr{A}_{\text{ad}}}  \mathcal{J}(\u,\varphi,\U).
	\end{align}
	\begin{definition}[Optimal Solution]
		A solution to the Problem \eqref{control problem} is called an \emph{optimal solution} and the optimal triplet is denoted by $(\u^* ,\varphi^*, \U^*)$. The control $\U^*$ is called an \emph{optimal control}.
	\end{definition}
	
	Our main aim in this work is to establish the existence of an optimal control to the problem \eqref{control problem} and establish the first order necessary conditions of optimality. From the well known theory for the optimal control problems  governed by ordinary and partial differential equations, we know that the optimal control is derived in terms of the adjoint variable which satisfies a linear system. As a first step to fulfill this aim, we linearize the nonlinear system and obtain the existence and uniqueness of weak solution of the linearized system using a standard Galerkin approximation technique (see \cite{BDM} also).

	\subsection{The linearized system}
	Let us linearize the equations \eqref{nonlin phi}-\eqref{initial conditions}  around $(\widehat{\u}, \widehat{\varphi})$, which is the \emph{unique strong solution} of the system (\ref{nonlin phi})-(\ref{initial conditions}) with control term $\U=0$ (uncontrolled system) and external forcing $\widehat{\h}$ such that
	$$\widehat{\h}\in \mathrm{L}^2(0,T;\G_{\text{div}}),\ \widehat{\u}_0\in\V_{\text{div}},\ \widehat{\varphi}_0\in\mathrm{V}\cap\mathrm{L}^{\infty}(\Omega).$$ 
	Let us first  rewrite the equation \eqref{nonlin u}. We know that 
	\begin{align*} \mu\nabla \varphi = (a\varphi - \J\ast \varphi + \F'(\varphi)) \nabla \varphi= \nabla \left(\F(\varphi) + a\frac{\varphi^2}{2}\right) - \nabla a\frac{\varphi^2}{2} - (\J\ast \varphi)\nabla \varphi. \end{align*}
	Hence  one can rewrite \eqref{nonlin u} as
	\begin{equation}\label{nonlin u rewritten}
	\u_t - \nu \Delta \u + (\u\cdot \nabla )\u + \nabla \widetilde{\uppi}_{\u} = - \nabla a\frac{\varphi^2}{2} - (\J\ast \varphi)\nabla \varphi +\h + \U,
	\end{equation} 
	where $\widetilde{\uppi}_{\u} = \uppi -\left( \F(\varphi) + a\frac{\varphi^2}{2}\right)$. It should be noted that the pressure is also an unknown quantity. In order to linearize, we substitute $\u= \w+ \widehat{\u},$ $\uppi=\widetilde\uppi+\widehat{\uppi}$ and $\varphi = \psi + \widehat{\varphi}$ in  \eqref{nonlin u rewritten} and \eqref{nonlin phi} to  obtain 
	$$\w_t - \nu \Delta \w + (\w \cdot \nabla)\widehat{\u} + (\widehat{\u} \cdot \nabla )\w + \nabla \widetilde{\uppi}_\w = -\nabla a\psi \widehat{\varphi} -(\J\ast \psi) \nabla \widehat{\varphi} - (\J \ast \widehat{\varphi}) \nabla \psi +\widetilde \h +\U, $$ 
	where $ \widetilde{\uppi}_\w = \widetilde{\uppi}-(\F'(\widehat{\varphi})+a\widehat{\varphi})\psi$, $\widetilde{\h}=\h-\widehat{\h}$. Also, we have 
	$$\psi_t + \w\cdot \nabla \widehat{\varphi} + \widehat{\u} \cdot \nabla \psi = \Delta \widetilde{\mu} $$   
	where $\widetilde{\mu} = a \psi - \J\ast \psi + \F''( \widehat{\varphi}) \psi$.
	Hence, we consider the following linearized system:
	\begin{subequations}
		\begin{align}
		\w_t - \nu \Delta \w + (\w \cdot \nabla )\widehat{\u} + (\widehat{\u} \cdot \nabla )\w + \nabla \widetilde{\uppi}_\w &= \ -\nabla a\psi \widehat{\varphi}  - (\J\ast \psi) \nabla \widehat{\varphi} \nonumber \\ &- (\J \ast \widehat{\varphi}) \nabla \psi  + \widetilde\h +\U, \ \text{in } \  \Omega \times (0,T)\label{lin w} \\
		\psi_t + \w\cdot \nabla \widehat{\varphi} + \widehat{\u} \cdot\nabla \psi &= \Delta \widetilde{\mu}, \ \text{in } \  \Omega \times (0,T)\label{lin psi} \\ 
		\widetilde{\mu} &= a \psi - \J\ast \psi + \F''(\widehat{\varphi})\psi,\ \text{in } \  \Omega \times (0,T) \label{lin mu}\\
		\text{div }\w &= 0,\ \text{in } \  \Omega \times (0,T) \label{lin div zero}\\
		\frac{\partial \widetilde{\mu}}{\partial\mathbf{n}} &= 0, \ \w=\mathbf{0}  \ \text{on } \ \partial \Omega \times (0,T), \label{lin boundary conditions}\\
		\w(0) &= \w_0, \ \psi(0) = \psi _0 \ \text{ in } \ \Omega. \label{lin initial conditions}
		\end{align} 
	\end{subequations}
	Note that in \eqref{lin mu}, we used the Taylor formula:
	$\F'(\psi+\widehat{\varphi})=\F'(\widehat{\varphi})+\F''(\widehat{\varphi})\psi+\F'''(\widehat{\varphi})\frac{\psi^2}{2}+\cdots,$ and ignored the second order terms in $\psi$, since we are considering a linear system, and we know that $\widehat{\varphi}\in\mathrm{L}^{\infty}((0,T)\times\Omega)$ and $\F(\cdot)$ has a polynomial growth as discussed in Remark \ref{remark F}. 
	In order to establish the Pontryagin maximum principle in the next section, we take $\widetilde{\h}=\w_0=\mathbf{0}$ and $\psi_0=0$. 
	
	Next, we discuss the unique global solvability results for the system (\ref{lin w})-(\ref{lin initial conditions}). 
	
\begin{theorem}[Existence and Uniqueness of Linearized System]\label{linearized}
	Suppose that the Assumption \ref{prop of F and J} is satisfied. Let us assume $(\widehat{\u},\widehat{\varphi})$ is the unique strong solution of the system \eqref{nonlin phi}-\eqref{initial conditions} with the regularity given in (\ref{sol1}) and $\F \in\C^3(\mathbb{R}),\ a \in\mathrm{H}^2(\Omega).$ Let $\w_0 \in \G_{\text{div}}$ and $\psi_0 \in \mathrm{H}$ with $\widetilde{\h},\U \in \mathrm{L}^2(0,T;\G_{\text{div}})$.  Then, for a given $T >0$, there exists \emph{a unique weak solution} $(\w,\psi)$ to the system \eqref{lin w}-\eqref{lin initial conditions} such that
	$$\w \in \mathrm{L}^{\infty}(0,T;\G_{\text{div}}) \cap \mathrm{L}^{2}(0,T;\V_{\text{div}})\ \text{ and 
	} \ \psi \in \mathrm{L}^{\infty}(0,T;\mathrm{H}) \cap  \mathrm{L}^{2}(0,T;\mathrm{V}).$$ and for every $\v \in \V_{\text{div}}$ and $\xi \in \mathrm{V}$ and for all $t \in (0,T)$, we have  
	\begin{eqnarray}\label{weak_lin}
	\left\{
	\begin{aligned}
	\langle\w_t,\v\rangle + \nu (\nabla\w,\nabla \v) + b(\w,\widehat{\u},\v) + b(\widehat{\u},\w,\v)  &= - (\nabla a\psi \widehat{\varphi},\v) - ((\J\ast \psi) \nabla \widehat{\varphi},\v)  \\  
	&	\quad - ((\J \ast \widehat{\varphi}) \nabla \psi,\v) + (\widetilde\h,\v) + (\U,\v),  \\ 
	\langle\psi_t,\xi\rangle + (\w\cdot \nabla \widehat{\varphi},\xi) + (\widehat{\u} \cdot \nabla \psi,\xi) &= -(\nabla \widetilde{\mu},\nabla \xi), 
	\end{aligned}
	\right.
	\end{eqnarray}
	where $\w(0) = \w_0, \ \psi(0) = \psi _0$ are satisfied in the weak sense. 
\end{theorem}
\begin{proof}
	In order to prove the existence and uniqueness of the linearized system \eqref{lin w}-\eqref{lin initial conditions}, we use a Faedo-Galerkin approximation scheme and show that the approximate solutions converge in appropriate spaces to the solution of the linearized system in the weak sense (see \eqref{weak_lin}). 
	
	Let us introduce the family $\{\me_j\}_{j=1}^\infty$ of eigen functions of the Stokes operator as a Galerkin basis of $\V_{\text{div}}$ and the family $\{\xi_j\}_{j=1}^\infty$ of eigen functions of the Neumann operator $\mathcal{B}_1=-\Delta+\I$ as a Galerkin basis of $\mathrm{V}$. Let us define the n-dimensional subspaces $\V_n$ and $\mathrm{V}_n$ as span of $\{\me_j\}_{j=1}^n$ and $\{e_j\}_{j=1}^n$ respectively. Let us take orthogonal projector on $\V_n$ and $\mathrm{V}_n$ in $\G_{\text{div}}$ and $\mathrm{H}$ respectively as $\mathrm{P}_n=\P_{\V_n}$ and $\widetilde{\P}_n=\P_{\mathrm{V}_n}$. Then we can  find   $(\w_n, \psi_n)$ of the form  $$\w_n = \sum_{i=1}^n \alpha_i^{(n)}(t)\me_i \ \text{     and    } \ \psi_n = \sum_{i=1}^n \beta_i^{(n)}(t)e_i,$$ 	as a solution of the following approximated system:
	\begin{align}
&\langle{(\w_n)}_t,\v_n\rangle+ \nu (\nabla\w_n,\nabla \v_n) + b(\w_n,\widehat{\u},\v_n) + b(\widehat{\u},\w_n,\v_n) + \langle\nabla \widetilde{\uppi}_{\w_n},\v_n\rangle\nonumber\\&\quad= - (\nabla a\psi_n \widehat{\varphi},\v_n)  
	 - ((\J\ast \psi_n)\nabla \widehat{\varphi},\v_n)- ((\J \ast \widehat{\varphi})\nabla \psi_n,\v_n) + (\widetilde\h_n,\v_n) +( \U_n,\v_n), \label{f1} \\ 
&	\langle{(\psi_n)}_t,\xi_n\rangle+ (\w_n\cdot \nabla \widehat{\varphi},\xi_n) + (\widehat{\u} \cdot \nabla \psi_n,\xi_n) = -(\nabla \widetilde{\mu}_n,\nabla \xi_n), \label{f3} \\ 
&	\widetilde{\mu}_n = \widetilde{\mathrm{P}}_n(a\psi_n - \J\ast \psi_n + \F''(\widehat{\varphi})\psi_n) ,\label{f4}  \\ 
	&\w_n(0) = \P_n\w_0=\sum_{i=1}^n(\w_0,\me_i)\me_i=:\w_{n,0}, \ \psi_n(0) = \widetilde{\mathrm{P}}_n\psi_
	0=\sum_{i=1}^n(\psi_0,e_i)e_i=:\psi_{n,0}, \label{f5}
\end{align}	
	for every $\v_n\in \V_n$ and $\xi_n \in \mathrm{V}_n$. Observe that for every $\v_n\in\V_n$, $\P_n\v_n=\v_n$ and for every $\xi_n\in\mathrm{V}_n$, $\widetilde{\P}_n\xi_n=\xi_n$. Note that in  \eqref{f1}, $\mathrm{U}_n=\mathrm{P}_n\U$ and $\widetilde{\h}_n=\P_n\widetilde{\h}$.  In \eqref{f1}, we also used the fact that $(\P_n(\nabla a\psi_n \widehat{\varphi}),\v_n) =(\nabla a\psi_n \widehat{\varphi},\P_n\v_n)=(\nabla a\psi_n \widehat{\varphi},\v_n)$, and similarly for other terms.  Since the system is linear, using the Cauchy-Lipschitz theorem or Picard-Lindel\"of theorem, we can easily show that there exists a local in time unique solution $(\w_n,\psi_n)$ to the approximated system \eqref{f1}-\eqref{f5}. Now we show that $\w_n$ and $\psi_n$ are uniformly bounded in some suitable function spaces. 
	
	Let us take  test function $\v_n$ as $\w_n$ in \eqref{f1} to obtain 
	\begin{align}\label{3.3}
	&\frac{1}{2}\frac{\d}{\d t}\|\w_n(t)\|^2 + \nu \|\nabla\w_n(t)\|^2 + b(\w_n,\widehat{\u},\w_n) + \langle\nabla \widetilde{\uppi}_{\w_n},\w_n\rangle \nonumber\\&= - (\nabla a\psi_n \widehat{\varphi},\w_n) - ((\J\ast \psi_n) \nabla \widehat{\varphi},\w_n) 
	- ((\J \ast \widehat{\varphi}) \nabla \psi_n,\w_n) + (\widetilde\h,\w_n) +( \U,\w_n),
	\end{align}
	where we used the fact that $b(\widehat{\u},\w_n,\w_n)=0$. Now we estimate each term of the above equality. 
	In order to estimate $b(\w_n,\widehat{\u},\w_n)$, we use  H\"older's, Ladyzhenskaya and Young's inequalities to find
	\begin{align}\label{3.4}
	|b(\w_n,\widehat{\u},\w_n)| \leq \|\nabla\widehat{\u}\|\|\w_n\|_{\mathbb{L}^4}^2 \leq\sqrt{2}\|\nabla\widehat{\u}\|\|\w_n\|\|\nabla\w_n\|&\leq \frac{\nu}{12}\|\nabla\w_n\|^2\nonumber \\
&\qquad+\frac{6}{\nu}\|\nabla\widehat{\u}\|^2\|\w_n\|^2.
	\end{align}
	Using an integration by parts and divergence free condition $\text{div }\w_n=0$,  we get, 
	$$ \langle\nabla \widetilde{\uppi}_{\w_n},\w_n\rangle = (\widetilde{\uppi}_{\w_n}, \text{div }\w_n) = 0.$$
	Once again using H\"older's, Ladyzhenskaya and Young's inequalities, we obtain 
	\begin{align}\label{3.5}
	|(\nabla a\psi_n \widehat{\varphi},\w_n)| &\leq \|\nabla a\|_{\mathbb{L}^{\infty}}\|\psi_n\|\|\widehat{\varphi}\|_{\mathrm{L}^{4}}\|\w_n\|_{\mathbb{L}^4} \no\\&\leq 2^{1/4}\|\nabla a\|_{\mathbb{L}^{\infty}}\|\psi_n\|\|\widehat{\varphi}\|_{\mathrm{L}^{4}}\|\w_n\|^{1/2}\|\nabla\w_n\|^{1/2} \no\\&\leq \frac{C_0}{10}\|\psi_n\|^2+\frac{5}{\sqrt{2}C_0}\|\nabla a\|_{\mathbb{L}^{\infty}}^2\|\widehat{\varphi}\|_{\mathrm{L}^4}^2\|\w_n\|\|\nabla\w_n\|\no \\&\leq  \frac{C_0}{10}\|\psi_n\|^2+\frac{\nu}{12}\|\nabla\w_n\|^2+\frac{75}{2\nu C_0^2}\|\nabla a\|_{\mathbb{L}^{\infty}}^4\|\widehat{\varphi}\|_{\mathrm{L}^4}^4\|\w_n\|^2.
	\end{align}
	In order to estimate $((\J\ast \psi_n) \nabla \widehat{\varphi},\w_n)$ and $((\J\ast \widehat{\varphi} \nabla \psi_n),\w_n)$, we integrate by parts and use divergence free condition to find  
	\begin{align*}((\J\ast \psi_n) \nabla \widehat{\varphi},\w_n) &= -( ( \nabla \J \ast \psi_n)\widehat{\varphi},\w_n),\\
	((\J\ast \widehat{\varphi}) \nabla \psi_n,\w_n) &= -(( \nabla \J \ast \widehat{\varphi})\psi_n. \w_n),\end{align*}
	We estimate $( ( \nabla \J \ast \psi_n)\widehat{\varphi},\w_n)$  using H\"older's inequality, Ladyzhenskaya inequality and Young's inequality for convolution as
	\begin{align}\label{3.6}
	|( \nabla \J \ast \psi_n) \widehat{\varphi},\w_n)| &\leq \| \nabla \J*\psi_n\|_{\mathbb{L}^2}\|\widehat{\varphi}\|_{\mathrm{L}^4} \|\w_n\|_{\mathbb{L}^4}\leq \| \nabla \J\|_{\mathbb{L}^1}\|\psi_n\|\|\widehat{\varphi}\|_{\mathrm{L}^4} \|\w_n\|_{\mathbb{L}^4}\no\\
	&\leq 2^{1/4}\| \nabla \J\|_{\mathbb{L}^1}\|\psi_n\|\|\widehat{\varphi}\|_{\mathrm{L}^4} \|\w_n\|^{1/2}\|\nabla\w_n\|^{1/2}\no \\
	&\leq \frac{C_0}{10}\|\psi_n\|^2 + \frac{\nu}{12}\|\nabla \w_n\|^2+ \frac{75}{2\nu C_0^2}\| \nabla \J\|_{\mathbb{L}^1}^4\|\widehat{\varphi}\|_{\mathrm{L}^4}^4\|\w_n\|^2.
	\end{align}  
	Similarly, we obtain
	\begin{align}\label{3.7}
	\left|(( \nabla \J \ast \widehat{\varphi})\psi_n, \w_n)\right| &\leq \frac{C_0}{10}\|\psi_n\|^2 + \frac{\nu}{12}\|\nabla \w_n\|^2+ \frac{75}{2\nu C_0^2}\| \nabla \J\|_{\mathbb{L}^1}^4\|\widehat{\varphi}\|_{\mathrm{L}^4}^4\|\w_n\|^2.
	\end{align} 
	Using the Cauchy-Schwarz inequality and Young's inequality, we also get
	\begin{equation}\label{3.8}
	\begin{aligned}
	\left|(\widetilde\h_n,\w_n )\right| &\leq \frac{3}{\nu} \|\widetilde\h_n\|^2+ \frac{\nu}{12}\|\nabla\w_n\|^2\leq \frac{3}{\nu} \|\widetilde\h\|^2+ \frac{\nu}{12}\|\nabla\w_n\|^2,\\
	\left|	(\U_n,\w_n )\right| &\leq  \frac{3}{\nu} \|\U_n\|^2 + \frac{\nu}{12}\|\nabla\w_n\|^2\leq  \frac{3}{\nu} \|\U\|^2 + \frac{\nu}{12}\|\nabla\w_n\|^2.\end{aligned}
	\end{equation}
	Combining (\ref{3.4})-(\ref{3.8}) and substituting in (\ref{3.3}) to find 
	\begin{align}\label{2.22}
	\frac{1}{2}\frac{\d}{\d t}\|\w_n\|^2 +  \frac{\nu}{2}\|\nabla\w_n\|^2  &\leq \frac{3C_0}{10}\|\psi_n\|^2 + \frac{3}{\nu} \|\wi\h\|^2 + \frac{3}{\nu} \|\U\|^2 +\frac{6}{\nu}\|\nabla\widehat{\u}\|^2\|\w_n\|^2\nonumber\\ &\quad+\left[\frac{75}{2\nu C_0^2}\|\widehat{\varphi}\|_{\mathrm{L}^4}^4\left(\|\nabla a\|_{\mathbb{L}^{\infty}}^4+2\|\nabla \J\|_{\mathbb{L}^1}^4\right)\right]\|\w_n\|^2.
	\end{align}
	
	Let us now take the test function $\xi_n$ as $\widetilde{\mu}_n+\widetilde{\P}_n(\J*\psi_n)$ in \eqref{f3} to obtain 
	\begin{align}\label{l29}
	&\left(\frac{\d\psi_n}{\d t},\widetilde{\mu}_n+\widetilde{\P}_n(\J*\psi_n)\right)+\|\nabla\widetilde{\mu}_n\|^2=(\Delta\widetilde{\mu}_n,\widetilde{\P}_n(\J*\psi_n))\nonumber \\
	&\quad-(\w_n\cdot \nabla \widehat{\varphi},\widetilde{\mu}_n+\widetilde{\P}_n(\J*\psi_n))-( \widehat{\u} \cdot\nabla \psi_n ,\widetilde{\mu}_n+\widetilde{\P}_n(\J*\psi_n)).
	\end{align}
	Using \eqref{f4}, we infer that
	\begin{align}\label{l30}
	\left(\frac{\d\psi_n}{\d t},\widetilde{\mu}_n+\widetilde{\P}_n(\J*\psi_n)\right)=\left(\frac{\d\psi_n}{\d t},\widetilde{\P}_n(a  + \F''(\widehat{\varphi}))\psi_n\right)=\left(\frac{\d\psi_n}{\d t},(a  + \F''(\widehat{\varphi}))\psi_n\right).
	\end{align}
	Note that $0<C_0\leq a+\F''(s)$ for all $s\in\R$ and a.e. $x\in\Omega$. Thus, we get 
	\begin{align}\label{l32}
	\frac{\d}{\d t}\left\|\sqrt{a+\F''(\widehat{\varphi})}\psi_n\right\|^2&=\frac{\d}{\d t}((a+\F''(\widehat{\varphi}))\psi_n,\psi_n)\nonumber\\&=\left(\frac{\d}{\d t}((a+\F''(\widehat{\varphi}))\psi_n),\psi_n\right)+\left((a+\F''(\widehat{\varphi}))\frac{\d\psi_n}{\d t},\psi_n\right)\nonumber\\&=(\F'''(\widehat{\varphi})\widehat{\varphi}_t\psi_n,\psi_n)+2\left(\frac{\d\psi_n}{\d t},(a  + \F''(\widehat{\varphi}))\psi_n\right).
	\end{align}
	Substituting(\ref{l32}) in (\ref{l29}), we obtain  
	\begin{align}\label{l33}
	\frac{1}{2}\frac{\d}{\d t}\left\|\sqrt{a+\F''(\widehat{\varphi})}\psi_n\right\|^2+\|\nabla\widetilde{\mu}_n\|^2&=\frac{1}{2}(\F'''(\widehat{\varphi})\widehat{\varphi}_t\psi_n,\psi_n)+(\Delta\widetilde{\mu}_n,\widetilde{\P}_n(\J*\psi_n))\no\\&\quad-(\w_n\cdot \nabla \widehat{\varphi},\widetilde{\mu}_n+\widetilde{\P}_n(\J*\psi_n)) -( \widehat{\u} \cdot\nabla \psi_n ,\widetilde{\mu}_n+\widetilde{\P}_n(\J*\psi_n)).
	\end{align}
	An integration by parts, Cauchy-Schwarz, Young's inequality for convolution and Young's inequality yields
	\begin{align}\label{l34}
	|(\Delta\widetilde{\mu}_n,\widetilde{\P}_n(\J*\psi_n))|&=|(\nabla\widetilde{\mu}_n,\nabla\J*\psi_n)|\leq \|\nabla\widetilde{\mu}_n\|\|\nabla\J*\psi_n\|\leq \|\nabla\widetilde{\mu}_n\|\|\nabla\J\|_{\mathbb{L}^1}\|\psi_n\|\nonumber\\&\leq \frac{1}{6}\|\nabla\widetilde{\mu}_n\|^2+\frac{3}{2}\|\nabla\J\|_{\mathbb{L}^1}^2\|\psi_n\|^2.
	\end{align}
	Using an integration by parts, the divergence free condition of $\w_n$, H\"older, Ladyzhenskaya and Young's inequalities, we estimate $|(\w\cdot \nabla \widehat{\varphi},\widetilde{\mu}_n)|$ as 
	\begin{align}
	|(\w_n\cdot \nabla \widehat{\varphi},\widetilde{\mu}_n)|&=|(\w_n\cdot \nabla \widetilde{\mu}_n,\widehat{\varphi})|\leq \|\w_n\|_{\mathbb{L}^4}\|\nabla\widetilde{\mu}_n\|\|\widehat{\varphi}\|_{\mathrm{L}^4}\leq \frac{1}{6}\|\nabla\widetilde{\mu}_n\|^2+\frac{3}{2}\|\w_n\|_{\mathbb{L}^4}^2\|\widehat{\varphi}\|_{\mathrm{L}^4}^2\nonumber\\&\leq \frac{1}{6}\|\nabla\widetilde{\mu}_n\|^2+\frac{\nu}{8}\|\nabla\w_n\|^2+\frac{9}{\nu}\|\widehat{\varphi}\|^4_{\mathrm{L}^4}\|\w_n\|^2.
	\end{align}
	Once again an integration by parts, H\"older and Young's inequalities yield
	\begin{align}\label{l35}
	|( \widehat{\u} \cdot\nabla \psi_n ,\widetilde{\mu}_n)|&=|( \widehat{\u} \cdot\nabla \widetilde{\mu}_n,\psi_n )|\leq \|\widehat{\u}\|_{\mathbb{L}^4}\|\nabla\widetilde{\mu}_n\|\|\psi_n\|_{\mathrm{L}^4}\leq \frac{1}{6}\|\nabla\widetilde{\mu}_n\|^2+\frac{3}{2}\|\widehat{\u}\|_{\mathbb{L}^4}^2\|\psi_n\|_{\mathrm{L}^4}^2.
	\end{align}
	We use (\ref{l34})-(\ref{l35}) in (\ref{l33}) to obtain 
	\begin{align}\label{l36}
	&\frac{1}{2}\frac{\d}{\d t}\left\|\sqrt{a+\F''(\widehat{\varphi})}\psi_n\right\|^2+\frac{1}{2}\|\nabla\widetilde{\mu}_n\|^2\nonumber\\&\leq \frac{1}{2}(\F'''(\widehat{\varphi})\widehat{\varphi}_t\psi_n,\psi_n)+\frac{\nu}{8}\|\nabla\w_n\|^2+\frac{3}{2}\|\nabla\J\|_{\mathbb{L}^1}^2\|\psi_n\|^2+\frac{9}{\nu}\|\widehat{\varphi}\|^4_{\mathrm{L}^4}\|\w_n\|^2\nonumber\\&\quad +\frac{3}{2}\|\widehat{\u}\|_{\mathbb{L}^4}^2\|\psi_n\|_{\mathrm{L}^4}^2 -(\w_n\cdot \nabla \widehat{\varphi},\widetilde{\P}_n(\J*\psi_n)))-( \widehat{\u} \cdot\nabla \psi_n ,\widetilde{\P}_n(\J*\psi_n))).
	\end{align}
	Note that $C_0\leq a(x)+\F''(s),$ for all $s\in\R$ and a.e. $x\in\Omega$. Using H\"older's, Young's convolution, Ladyzhenskaya and Young's inequalities, we get
	\begin{align}\label{l37}
	\langle-\Delta\psi_n,\widetilde{\mu}_n\rangle&=(\nabla\psi_n,\nabla\widetilde{\mu}_n)=(\nabla\psi_n,\nabla\widetilde{\P}_n(a\psi_n-\J*\psi_n+\F''(\widehat{\varphi})\psi_n))\nonumber\\&=(\nabla\psi_n,a\nabla\psi_n+\psi_n\nabla a-\nabla\J*\psi_n+\F''(\widehat{\varphi})\nabla\psi_n+\F'''(\widehat{\varphi})\nabla\widehat{\varphi}\psi_n)\nonumber\\&\geq C_0\|\nabla\psi_n\|^2-(\|\nabla a\|_{\mathbb{L}^{\infty}}+\|\nabla\J\|_{\mathbb{L}^1})\|\psi_n\|\|\nabla\psi_n\|\no\\&\qquad-\|\F'''(\widehat{\varphi})\|_{\mathrm{L}^{\infty}}\|\nabla\widehat{\varphi}\|_{\mathbb{L}^4}\|\nabla\psi_n\|\|\psi_n\|_{\mathrm{L}^4}
	\nonumber\\&\geq C_0\|\nabla\psi_n\|^2-\frac{C_0}{3}\|\nabla\psi_n\|^2-\frac{3}{2C_0}(\|\nabla a\|_{\mathbb{L}^{\infty}}^2+\|\nabla\J\|_{\mathbb{L}^1}^2)\|\psi_n\|^2\nonumber\\&\quad -\sqrt{2}\|\F'''(\widehat{\varphi})\|_{\mathrm{L}^{\infty}}\|\nabla\widehat{\varphi}\|_{\mathbb{L}^4}\|\nabla\psi_n\|^{3/2}\|\psi_n\|^{1/2}\nonumber\\&\geq \frac{C_0}{2}\|\nabla\psi_n\|^2-\frac{3}{2C_0}(\|\nabla a\|_{\mathbb{L}^{\infty}}^2+\|\nabla\J\|_{\mathbb{L}^1}^2)\|\psi_n\|^2\no\\&\qquad-\frac{729}{8C_0^3}\|\F'''(\widehat{\varphi})\|_{\mathrm{L}^{\infty}}^4\|\nabla\widehat{\varphi}\|_{\mathbb{L}^4}^4\|\psi_n\|^2.
	\end{align}
	But, we also have 
	\begin{equation}\label{l38}
	(\nabla\psi_n,\nabla\widetilde{\mu}_n)\leq \|\nabla\psi_n\|\|\nabla\widetilde{\mu}_n\|\leq \frac{C_0}{4}\|\nabla\psi_n\|^2+\frac{1}{C_0}\|\nabla\widetilde{\mu}_n\|^2. 
	\end{equation}
	Combining (\ref{l37}) and (\ref{l38}), we find
	\begin{equation}\label{l39}
	\frac{C_0}{4}\|\nabla\psi_n\|^2\leq \frac{1}{C_0}\|\nabla\widetilde{\mu}_n\|^2+\frac{3}{2C_0}(\|\nabla a\|_{\mathbb{L}^{\infty}}^2+\|\nabla\J\|_{\mathbb{L}^1}^2)\|\psi_n\|^2+\frac{729}{8C_0^3}\|\F'''(\widehat{\varphi})\|_{\mathrm{L}^{\infty}}^4\|\nabla\widehat{\varphi}\|_{\mathbb{L}^4}^4\|\psi_n\|^2.
	\end{equation}
	Let us substitute (\ref{l39}) in (\ref{l36}) to get 
	\begin{align}\label{l40}
	&\frac{1}{2}\frac{\d}{\d t}\left\|\sqrt{a+\F''(\widehat{\varphi})}\psi_n\right\|^2+\frac{C_0^2}{8}\|\nabla\psi_n\|^2\nonumber\\&\leq \frac{1}{2}(\F'''(\widehat{\varphi})\widehat{\varphi}_t\psi_n,\psi_n)+\frac{\nu}{8}\|\nabla\w_n\|^2+\frac{3}{2}\|\nabla\J\|_{\mathbb{L}^1}^2\|\psi_n\|^2+\frac{9}{\nu}\|\widehat{\varphi}\|^4_{\mathrm{L}^4}\|\w_n\|^2\nonumber\\&\quad+\frac{3}{2}\|\widehat{\u}\|_{\mathbb{L}^4}^2\|\psi_n\|_{\mathrm{L}^4}^2 +\frac{3}{4}(\|\nabla a\|_{\mathbb{L}^{\infty}}^2+\|\nabla\J\|_{\mathbb{L}^1}^2)\|\psi_n\|^2+\frac{729}{16C_0^2}\|\F'''(\widehat{\varphi})\|_{\mathrm{L}^{\infty}}^4\|\nabla\widehat{\varphi}\|_{\mathbb{L}^4}^4\|\psi_n\|^2\nonumber\\&\quad  -(\w_n\cdot \nabla \widehat{\varphi},\widetilde{\P}_n(\J*\psi_n))-( \widehat{\u} \cdot\nabla \psi_n ,\widetilde{\P}_n(\J*\psi_n)).
	\end{align}
	We use  H\"older's, Ladyzhenskaya, Poincar\'e, convolution and Young's inequalities to estimate the terms $\frac{3}{2}\|\widehat{\u}\|_{\mathbb{L}^4}^2\|\psi_n\|_{\mathrm{L}^4}^2$, $(\F'''(\widehat{\varphi})\widehat{\varphi}_t\psi_n,\psi_n)$, $(\w_n\cdot \nabla \widehat{\varphi},\widetilde{\P}_n(\J*\psi_n))$ and $( \widehat{\u} \cdot\nabla \psi_n ,\widetilde{\P}_n(\J*\psi_n))$  as 	
\begin{align}
	\frac{3}{2}\|\widehat{\u}\|_{\mathbb{L}^4}^2\|\psi_n\|_{\mathrm{L}^4}^2&\leq \frac{3}{\sqrt{2}}\|\widehat{\u}\|_{\mathbb{L}^4}^2\|\psi_n\|\|\nabla\psi_n\|\leq \frac{C_0^2}{48}\|\nabla\psi_n\|^2+\frac{54}{C_0^2}\|\widehat{\u}\|_{\mathbb{L}^4}^4\|\psi_n\|^2,\label{l41}
\\
	|(\F'''(\widehat{\varphi})\widehat{\varphi}_t\psi_n,\psi_n)|&\leq \|\F'''(\widehat{\varphi})\|_{\mathrm{L}^{\infty}}\|\widehat{\varphi}_t\|\|\psi_n\|_{\mathrm{L}^4}^2\nonumber\\&\leq \frac{C_0^2}{48}\|\nabla\psi_n\|^2+\frac{24}{C_0^2}\|\F'''(\widehat{\varphi})\|_{\mathrm{L}^{\infty}}^2\|\widehat{\varphi}_t\|^2\|\psi_n\|^2,
\\
	|(\w_n\cdot \nabla \widehat{\varphi},\widetilde{\P}_n(\J*\psi_n))|&\leq \|\w_n\|_{\mathbb{L}^4}\|\nabla\widehat{\varphi}\|_{\mathbb{L}^4}\|\J*\psi_n\|  \nonumber\\&\leq \frac{\nu}{8}\|\nabla\w_n\|^2+\frac{2}{\nu}\sqrt{\frac{2}{\lambda_1}} \|\nabla\widehat{\varphi}\|_{\mathbb{L}^4}^2\|\J\|_{\mathrm{L}^1}^2\|\psi_n\|^2,\\
	|( \widehat{\u} \cdot\nabla \psi_n ,\widetilde{\P}_n(\J*\psi_n))|&\leq \|\widehat{\u}\|_{\mathbb{L}^4}\|\nabla\psi_n\|\|\J\|_{\mathrm{L}^1}\|\psi_n\|_{\mathbb{L}^4}\nonumber\\&\leq \frac{C_0^2}{48}\|\nabla\psi_n\|^2+\frac{1296}{C_0^6}\|\widehat{\u}\|_{\mathbb{L}^4}^4\|\J\|_{\mathrm{L}^1}^4\|\psi_n\|^2.\label{l44}
	\end{align}
	We substitute (\ref{l41})-(\ref{l44}) in (\ref{l40}) to obtain 
	\begin{align}\label{l45}
	&\frac{1}{2}\frac{\d}{\d t}\left\|\sqrt{a+\F''(\widehat{\varphi})}\psi_n\right\|^2+\frac{C_0^2}{16}\|\nabla\psi_n\|^2\nonumber\\&\leq \frac{\nu}{4}\|\nabla\w_n\|^2+\frac{9}{\nu}\|\widehat{\varphi}\|^4_{\mathrm{L}^4}\|\w_n\|^2+\left(\frac{3}{4}\|\nabla a\|_{\mathbb{L}^{\infty}}^2+\frac{9}{4}\|\nabla\J\|_{\mathbb{L}^1}^2\right)\|\psi_n\|^2\nonumber\\&\quad+\left(\frac{54}{C_0^2}\|\widehat{\u}\|_{\mathbb{L}^4}^4+ \frac{729}{16C_0^2}\|\F'''(\widehat{\varphi})\|_{\mathrm{L}^{\infty}}^4\|\nabla\widehat{\varphi}\|_{\mathbb{L}^4}^4+\frac{24}{C_0^2}\|\F'''(\widehat{\varphi})\|_{\mathrm{L}^{\infty}}^2\|\widehat{\varphi}_t\|^2\right.\nonumber\\&\qquad\left.+\frac{2}{\nu}\sqrt{\frac{2}{\lambda_1}} \|\nabla\widehat{\varphi}\|_{\mathbb{L}^4}^2\|\J\|_{\mathrm{L}^1}^2+\frac{1296}{C_0^6}\|\widehat{\u}\|_{\mathbb{L}^4}^4\|\J\|_{\mathrm{L}^1}^4\right)\|\psi_n\|^2.
	\end{align}
Let us now add the inequalities (\ref{2.22}) and (\ref{l45}) to get  
	\begin{align}\label{l46}
	&\frac{1}{2}\frac{\d}{\d t}\|\w_n\|^2 +  \frac{\nu}{4}\|\nabla\w_n\|^2 + \frac{1}{2}\frac{\d}{\d t}\left\|\sqrt{a+\F''(\widehat{\varphi})}\psi_n\right\|^2+\frac{C_0^2}{16}\|\nabla\psi_n\|^2\nonumber\\&\leq  \frac{3C}{\nu} \|\wi\h\|^2 + \frac{3}{\nu} \|\U\|^2+ \left(\frac{3C_0}{10}+\frac{3}{4}\|\nabla a\|_{\mathbb{L}^{\infty}}^2+\frac{9}{4}\|\nabla\J\|_{\mathbb{L}^1}^2\right)\|\psi_n\|^2\nonumber\\&\quad+\left(\frac{54}{C_0^2}\|\widehat{\u}\|_{\mathbb{L}^4}^4+ \frac{729}{16C_0^2}\|\F'''(\widehat{\varphi})\|_{\mathrm{L}^{\infty}}^4\|\nabla\widehat{\varphi}\|_{\mathbb{L}^4}^4+\frac{24}{C_0^2}\|\F'''(\widehat{\varphi})\|_{\mathrm{L}^{\infty}}^2\|\widehat{\varphi}_t\|^2\right.\nonumber\\&\qquad\left.+\frac{2}{\nu}\sqrt{\frac{2}{\lambda_1}} \|\nabla\widehat{\varphi}\|_{\mathbb{L}^4}^2\|\J\|_{\mathrm{L}^1}^2+\frac{1296}{C_0^6}\|\widehat{\u}\|_{\mathbb{L}^4}^4\|\J\|_{\mathrm{L}^1}^4\right)\|\psi_n\|^2 \nonumber\\ &\quad+\left[\frac{6}{\nu}\|\nabla\widehat{\u}\|^2+\left(\frac{9}{\nu}+\frac{75}{2\nu C_0^2}\left(\|\nabla a\|_{\mathbb{L}^{\infty}}^4+2\|\nabla \J\|_{\mathbb{L}^1}^4\right)\right)\|\widehat{\varphi}\|_{\mathrm{L}^4}^4\right]\|\w_n\|^2.
	\end{align} 
	Now we integrate the inequality (\ref{l46}) from $0$ to $t$ to obtain
	\begin{align}\label{l47}
	&\|\w_n(t)\|^2+\left\|\sqrt{a+\F''(\widehat{\varphi}(t))}\psi_n(t)\right\|^2+\frac{\nu}{2}\int_0^t\|\nabla\w_n(s)\|^2\d s+\frac{C_0^2}{8}\int_0^t\|\nabla\psi_n(s)\|^2\d s\nonumber\\&\leq \|\w_{n,0}\|^2+\left\|\sqrt{a+\F''(\widehat{\varphi}(0))}\psi_{n,0}\right\|^2+\frac{6C}{\nu}\int_0^t\|\widetilde{\h}(s)\|^2\d s+\frac{6}{\nu}\int_0^t\|\U(s)\|^2\d s\nonumber\\&\quad+\int_0^t \left[\frac{6}{\nu}\|\nabla\widehat{\u}(s)\|^2+\left(\frac{9}{\nu}+\frac{75}{2\nu C_0^2}\left(\|\nabla a\|_{\mathbb{L}^{\infty}}^4+2\|\nabla \J\|_{\mathbb{L}^1}^4\right)\right)\|\widehat{\varphi}(s)\|_{\mathrm{L}^4}^4\right]\|\w_n(s)\|^2\d s\nonumber\\&\quad + \left(\frac{3C_0}{10}+\frac{3}{4}\|\nabla a\|_{\mathbb{L}^{\infty}}^2+\frac{9}{4}\|\nabla\J\|_{\mathbb{L}^1}^2\right)\int_0^t\|\psi_n(s)\|^2\d s \nonumber\\&\quad+\int_0^t\left(\frac{54}{C_0^2}\|\widehat{\u}(s)\|_{\mathbb{L}^4}^4+ \frac{729}{16C_0^2}\|\F'''(\widehat{\varphi}(s))\|_{\mathrm{L}^{\infty}}^4\|\nabla\widehat{\varphi}(s)\|_{\mathbb{L}^4}^4+\frac{24}{C_0^2}\|\F'''(\widehat{\varphi}(s))\|_{\mathrm{L}^{\infty}}^2\|\widehat{\varphi}_t(s)\|^2\right.\nonumber\\&\qquad\left.+\frac{2}{\nu}\sqrt{\frac{2}{\lambda_1}} \|\nabla\widehat{\varphi}(s)\|_{\mathbb{L}^4}^2\|\J\|_{\mathrm{L}^1}^2+\frac{1296}{C_0^6}\|\widehat{\u}(s)\|_{\mathbb{L}^4}^4\|\J\|_{\mathrm{L}^1}^4\right)\|\psi_n(s)\|^2 \d s.
	\end{align}
An application of Gronwall's inequality in (\ref{l47}) yields 
	\begin{align}\label{l48}
	&\|\w_n(t)\|^2+\left\|\sqrt{a+\F''(\widehat{\varphi}(t))}\psi_n(t)\right\|^2\nonumber\\&\leq \left(\|\w_0\|^2+\left\|\sqrt{a+\F''(\widehat{\varphi}(0))}\psi_{n,0}\right\|^2+\frac{6C}{\nu}\int_0^t\|\widetilde{\h}(s)\|^2\d s+\frac{6}{\nu}\int_0^t\|\U(s)\|^2\d s\right)\nonumber\\&\quad \times \exp\left(\int_0^t\left[\frac{6}{\nu}\|\nabla\widehat{\u}(s)\|^2+\left(\frac{9}{\nu}+\frac{75}{2\nu C_0^2}\left(\|\nabla a\|_{\mathbb{L}^{\infty}}^4+2\|\nabla \J\|_{\mathbb{L}^1}^4\right)\right)\|\widehat{\varphi}(s)\|_{\mathrm{L}^4}^4\right]\d s\right)\nonumber\\&\quad \times \exp\left[\left(\frac{3C_0}{10}+\frac{3}{4}\|\nabla a\|_{\mathbb{L}^{\infty}}^2+\frac{9}{4}\|\nabla\J\|_{\mathbb{L}^1}^2\right)T\right]\nonumber\\&\quad\times \exp\left(\int_0^t\left(\frac{54}{C_0^2}\|\widehat{\u}(s)\|_{\mathbb{L}^4}^4+ \frac{729}{16C_0^2}\|\F'''(\widehat{\varphi}(s))\|_{\mathrm{L}^{\infty}}^4\|\nabla\widehat{\varphi}(s)\|_{\mathbb{L}^4}^4\right.\right.\nonumber\\&\qquad\left.\left.+\frac{24}{C_0^2}\|\F'''(\widehat{\varphi}(s))\|_{\mathrm{L}^{\infty}}^2\|\widehat{\varphi}_t(s)\|^2+\frac{2}{\nu}\sqrt{\frac{2}{\lambda_1}} \|\nabla\widehat{\varphi}(s)\|_{\mathbb{L}^4}^2\|\J\|_{\mathrm{L}^1}^2+\frac{1296}{C_0^6}\|\widehat{\u}(s)\|_{\mathbb{L}^4}^4\|\J\|_{\mathrm{L}^1}^4\right)\d s\right),
	\end{align}
	for all $t\in[0,T]$. Note that the right hand side of the inequality (\ref{l48}) is finite, since $(\widehat{\u},\widehat{\varphi})$ is a unique strong solution of the  uncontrolled system (\ref{nonlin phi})-(\ref{initial conditions}). Using the Assumption \ref{prop of F and J} (2), we get 
	\begin{equation}\label{l49}
	\|\w_n(t)\|^2+C_0\|\psi_n(t)\|^2\leq C,
	\end{equation}
	since we have 
	\begin{equation*}
	\left\|\sqrt{a+\F''(\widehat{\varphi}(0))}\psi_n(0)\right\|^2\leq \left(\|a\|_{\mathrm{L}^{\infty}}+\|\F''(\widehat{\varphi}_0)\|_{\mathrm{L}^{\infty}}\right)\|\psi_{n,0}\|^2\leq C(\|a\|_{\mathrm{L}^{\infty}},\|\widehat{\varphi}_0\|_{\mathrm{L}^{\infty}})\|\psi_0\|^2.
	\end{equation*}
	Substituting \eqref{l49} in \eqref{l47} we obtain 
	\begin{equation}\label{l50}
	\int_0^t\|\nabla\w_n(s)\|^2\d s+\int_0^t\|\nabla\psi_n(s)\|^2\d s \leq C.
	\end{equation}
	Thus, from the above estimates \eqref{l49} and \eqref{l50}, using the Banach-Alaglou theorem, we can extract a subsequence, still denoted by $(\w_{n},\psi_{n})$ such that 
	\begin{equation*}
	\left\{
	\begin{aligned}
	\w_{n}&\xrightharpoonup{w^*}\w \ \text{ in }\ \mathrm{L}^{\infty}(0,T;\G_{\text{div}}),\\
	\w_{n}&\xrightharpoonup{w}\w \ \text{ in }\ \mathrm{L}^{2}(0,T;\V_{\text{div}}),\\
	\psi_{n}&\xrightharpoonup{w^*}\psi \ \text{ in }\ \mathrm{L}^{\infty}(0,T;\mathrm{H}),\\
	\psi_{n}&\xrightharpoonup{w}\psi \ \text{ in }\ \mathrm{L}^{2}(0,T;\mathrm{V}),
	\end{aligned}
	\right.
	\end{equation*}
	as $n\to\infty$. From the above convergences, it is immediate that $(\w_{n})_t\xrightharpoonup{w}\w_t \ \text{ in }\ \mathrm{L}^{2}(0,T;\V_{\text{div}}')$ and $(\psi_{n})_t\xrightharpoonup{w}\psi_t \ \text{ in }\ \mathrm{L}^{2}(0,T;\mathrm{V}')$.  Thus, we can guarantee the existence of a weak solution $(\w,\psi)$ to the linearized system \eqref{lin w}-\eqref{lin initial conditions} in $(\mathrm{L}^{\infty}(0,T;\G_{\text{div}}) \cap \mathrm{L}^2(0,T;\V_{\text{div}}))\times(\mathrm{L}^{\infty}(0,T;\mathrm{H})\cap\mathrm{L}^2(0,T;\mathrm{V}))$. Also, by the Aubin-Lions  compactness theorem,  we have $\w_{n}\to\w \ \text{ in }\ \mathrm{L}^{2}(0,T;\G_{\text{div}}),$
	$\psi_{n}\to\psi \ \text{ in }\ \mathrm{L}^{2}(0,T;\mathrm{H})$, and  $\w\in\C([0,T];\G_{\text{div}})$ and $\psi\in\C([0,T];\mathrm{H})$. As in the case of Navier-Stokes equations (see \cite{Te}), we get the pressure  $\widetilde{\uppi}\in\mathrm{L}^2(0,T;\mathrm{L}^2_{(0)}(\Omega))$, where $\mathrm{L}_{(0)}^2(\Omega):=\left\{f\in\mathrm{L}^2(\Omega):\int_{\Omega}f(x)\d x =0\right\}$. Finally, we have 
	$$(\w,\psi)\in(\C([0,T];\G_{\text{div}})\cap\mathrm{L}^2(0,T;\V_{\text{div}}))\times (\C([0,T];\mathrm{H})\cap\mathrm{L}^2(0,T;\mathrm{V})),$$
and the initial datum $\w(0)=\w_0\in\G_{\text{div}}$ and $\psi(0)=\psi_0\in\mathrm{H}$ are satisfied in the weak sense, since $\w$ and $\psi$ are right continuous at $0$. The uniqueness of the weak solution follows easily as the system \eqref{lin w}-\eqref{lin initial conditions} is linear. 
\end{proof} 
	\subsection{The adjoint system} In this subsection, we formally derive the adjoint system corresponding to the problem \eqref{control problem}. Let us take $\h=\mathbf{0}$ in \eqref{nonlin phi}-\eqref{nonlin u}  and define
	\begin{equation}\label{n1n2}
	\left\{
	\begin{aligned} 
	\mathscr{N}_1(\u,\varphi, \U) &:= \nu \Delta \u - (\u \cdot \nabla)\u - \nabla \widetilde{\uppi } -(\J\ast \varphi) \nabla \varphi -\nabla a\frac{\varphi^2}{2} +\U,  \\
	\mathscr{N}_2(\u,\varphi) &:= -\u\cdot \nabla \varphi + \Delta (a\varphi -\J\ast \varphi + \F'(\varphi)),
	\end{aligned}
	\right.
	\end{equation}
	where $\widetilde{\uppi} = \uppi -\left( \F(\varphi) + a\frac{\varphi^2}{2}\right)$.
	Then the system \eqref{nonlin phi}-\eqref{nonlin u} can be written as
	$$(\partial_t \u,\partial_t \varphi) = (\mathscr{N}_1(\u,\varphi, \U), \mathscr{N}_2(\u,\varphi)).$$
	It is well known from the control theory literature that in order to get the first order necessary conditions for the existence of an optimal control to the Problem \eqref{control problem}, we need the adjoint equations corresponding to the system \eqref{nonlin phi}-\eqref{initial conditions}. With this motivation, we define the \emph{augmented cost functional} $\widetilde{\mathcal{J}}$ associated with the cost functional $\mathcal{J}$ defined in (\ref{cost}) by
	$$\widetilde{\mathcal{J}}(\u,\varphi,\U,\p,\eta) = \mathcal{J}(\u,\varphi,\U) + \int_0^T \langle \p,\partial_t \u - \mathscr{N}_1(\u,\varphi, \U) \rangle \d t+ \int_0^T \langle \eta, \partial_t \varphi - \mathscr{N}_2(\u,\varphi) \rangle \d t, $$
	where $\p$ and $\eta$ denote the adjoint variables to $\u$ and $\varphi$, respectively. 
	Before establishing the Pontryagin maximum principle, we derive the adjoint equations formally by differentiating the augmented cost functional $\widetilde{\mathcal{J}}$ in the G\^ateaux  sense with respect to the  variables $(\u,\varphi,\U)$ and get the following system. Note that differentiating $\widetilde{\mathcal{J}}$ with respect to the adjoint variables recovers the original system. 
	
	\begin{equation}
	\left\{
	\begin{aligned}
	\mathcal{J}_{\u} + \int_0^T \langle \p,  \partial_t \u -[\partial_{\u} \mathscr{N}_1]  \rangle \d t+ \int_0^T  \langle \eta, \partial_t \varphi - [\partial_{\u} \mathscr{N}_2] \rangle \d t = 0, \\  
	\mathcal{J}_{\varphi} + \int_0^T \langle \p,  \partial_t {\u} -[\partial_\varphi \mathscr{N}_1]  \rangle \d t+ \int_0^T  \langle \eta, \partial_t \varphi - [\partial_\varphi \mathscr{N}_2] \rangle \d t = 0,  \\  
	\mathcal{J}_{\U} + \int_0^T \langle \p,  \partial_t \u -[\partial_{\U} \mathscr{N}_1]  \rangle \d t+ \int_0^T  \langle \eta, \partial_t \varphi - [\partial_{\U} \mathscr{N}_2] \rangle \d t = 0.
	\end{aligned}
	\right.
	\end{equation}
	Therefore the adjoint variables $(\p,\eta)$ satisfy the following adjoint system (see \cite{BDM} for more details): 
	\begin{equation}\label{adj}
	\left\{
	\begin{aligned}
	-\p_t- \nu \Delta \p +(\p \cdot \nabla^T)\u - (\u \cdot \nabla)\p - \eta\nabla \varphi +\nabla q &= -\Delta(\u-\u_d),\ \text{in } \  \Omega \times (0,T)\\
	-\eta_t   +\J \ast (\p \cdot \nabla \varphi) - (\nabla{\J} \ast  \varphi)\cdot \p+ \nabla a\cdot \p \varphi 
	- \u \cdot\nabla \eta - a \Delta \eta&\\ + \J\ast \Delta \eta - \F''(\varphi)\Delta \eta &= (\varphi - \varphi_d),\ \text{in } \  \Omega \times (0,T) \\ 
	\text{div }\p&=0,\ \text{in } \  \Omega \times (0,T)\\
	 \p\big|_{\partial\Omega}=\frac{\partial \eta}{\partial \mathbf{n}}\Big|_{\partial \Omega}&=\mathbf{0} ,\ \text{on } \  \partial\Omega \times (0,T)\\ \p(T,\cdot)&=\u(T)-\u_f,\ \text{in } \  \partial\Omega \\ \eta(T,\cdot)&=\varphi(T)-\varphi_f,\ \text{in } \  \partial\Omega.
	\end{aligned}
	\right.
	\end{equation}
	Note that $q:\Omega\to\R$ is also an unknown. Let us assume that 
	\begin{align}\label{fes}
	\F \in\C^3(\mathbb{R}),\ a \in\mathrm{H}^2(\Omega),
	\end{align}
	and the initial data  
	\begin{align}\label{initial}
	\u_0\in\V_{\text{div}}\text{ and }\varphi_0\in \mathrm{H}^2(\Omega).
	\end{align} By the embedding of $\mathrm{V}$ and $\mathrm{L}^{\infty}(\Omega)$ in $\mathrm{H}^{2}(\Omega)$, the initial concentration  $\varphi_0\in \mathrm{H}^2(\Omega)$ implies $\varphi_0\in \mathrm{V}\cap\mathrm{L}^{\infty}(\Omega)$. From now onwards, along with the Assumption \ref{prop of F and J}, we also assume (\ref{fes}). The following theorem gives the unique solvability results for the system (\ref{adj}). 
\begin{theorem}[Existence and Uniqueness of Adjoint System]\label{adjoint}
	Let the Assumption \ref{prop of F and J}, \eqref{fes}, \eqref{initial} along  with $\J \in \W^{2,1}(\R^2,\R)$ be satisfied. Also let us assume that $\p_T \in \G_{\text{div}}$ and $\eta_T \in \mathrm{V}$ and $(\u,\varphi)$ be a unique strong solution of the nonlinear system \eqref{nonlin phi}-\eqref{initial conditions}. Then, there exists \emph{a unique weak solution} of the system \eqref{adj} satisfying
	\begin{align}\label{space}
	(\p,\eta)\in(\mathrm{L}^{\infty}(0,T;\G_{\text{div}})\cap\mathrm{L}^2(0,T;\V_{\text{div}}))\times (\mathrm{L}^{\infty}(0,T;\mathrm{V})\cap \mathrm{L}^2(0,T;\mathrm{H}^2)),\end{align} 
	and for all $\v \in \V$ and $\zeta \in \mathrm{H}$ and for almost all $t \in (0,T)$, we have
	\begin{equation}\label{0.1}
	\left\{
	\begin{aligned}
	- \langle \p_t, \v \rangle + \nu(\nabla \p,\nabla \v) -((\p \cdot \nabla^T)\u,\v) + ((\u \cdot \nabla)\p,\v) - (\eta \nabla \varphi,\v) &= -(\Delta(\u-\u_d),\v), \\ 
	-(\eta_t,\zeta) + (\J \ast (\p \cdot \nabla \varphi),\zeta) - ((\nabla{\J} \ast  \varphi)\cdot \p,\zeta) + (\nabla a\cdot \p \varphi 
	,\zeta)\qquad \ \, \, \\  - (\u \cdot\nabla \eta,\zeta) - (a \Delta \eta,\zeta)  +(\J\ast \Delta \eta,\zeta) - (\F''(\varphi)\Delta \eta,\zeta) &= (\varphi - \varphi_d,\zeta),  
	\end{aligned}	
	\right.
	\end{equation}
	where $\p(T)=\p_T\in\G_{\text{div}}$,  $\eta(T)=\eta_T\in\mathrm{V}$ are satisfied in the weak sense. 
\end{theorem}

	\begin{proof}
	In order to prove the existence and uniqueness of the adjoint system \eqref{adj}, we  use a Faedo-Galerkin approximation scheme and show that these approximations converge in the space given in (\ref{space}) to the solution of the system \eqref{adj}. We consider the finite dimensional approximation $\p_n$ and $\eta_n$ of $\p$ and $\eta$ respectively in the finite dimensional subspace $\V_n$ and $\mathrm{V}_n$. The rest of the approximation technique can be done similarly as in Theorem \ref{linearized} and we omit it in this context. 
	
	First we find a-priori energy estimates satisfied by $(\p,\eta)$.  Let us take the test function in the first equation \eqref{0.1} as $\p$ to obtain 
	\begin{align}\label{0.2}
	-\frac{1}{2}\frac{\d}{\d t}\|\p\|^2+\nu\|\nabla\p\|^2&=((\p\cdot\nabla^T)\u,\p)+(\eta\nabla \varphi ,\p)-(\Delta(\u-\u_d),\p)=:I_1+I_2+I_3,
	\end{align}
	where we used $((\u\cdot\nabla)\p,\p)=0$, since $\p$ is divergence free. Let us estimate  $I_1$ using H\"older's, Ladyzhenskaya and Young's inequalities as 
	\begin{align}\label{0.6}
	|I_1|\leq \|\p\|_{\mathbb{L}^4}^2\|\nabla\u\|\leq \sqrt{2}\|\p\|\|\nabla\p\|\nabla\u\|\leq \frac{\nu}{12}\|\nabla\p\|^2+\frac{6}{\nu}\|\nabla\u\|^2\|\p\|^2. 
	\end{align}
	Using an integration by parts, the divergence free condition, H\"older, Ladyzhenskaya and Young's inequalities, we estimate $I_2$ as 
	\begin{align}
	|I_2|=|(\nabla\eta \cdot \p,\varphi)|&\leq \|\nabla \eta\|\|\varphi\|_{\mathrm{L}^4}\|\p\|_{\mathbb{L}^4}\leq \frac{1}{2}\|\nabla\eta\|^2+\frac{1}{2}\|\varphi\|_{\mathrm{L}^4}^2\|\p\|\|\nabla\p\|\nonumber\\&\leq \frac{1}{2}\|\nabla\eta\|^2+\frac{\nu}{12}\|\nabla\p\|^2+\frac{3}{4\nu}\|\varphi\|_{\mathrm{L}^4}^4\|\p\|^2.
	\end{align}
	Let us use the Cauchy-Schwarz, Poincar\'e and Young's inequalities to estimate $I_3$ as 
	\begin{align}\label{0.8}
	|I_3|\leq \|\nabla(\u-\u_d)\|\|\nabla\p\|\leq \frac{\nu}{12}\|\nabla\p\|^2+\frac{3}{\nu}\|\nabla(\u-\u_d)\|^2. 
	\end{align}
	Using \eqref{0.6}-\eqref{0.8} in \eqref{0.2} we get,
	\begin{align}\label{0.9}
	-\frac{1}{2}\frac{\d}{\d t}\|\p\|^2+\nu\|\nabla\p\|^2 &\leq \frac{\nu}{4}\|\nabla\p\|^2 + \left( \frac{6}{\nu}\|\nabla\u\|^2 +\frac{3}{4\nu}\|\varphi\|_{\mathrm{L}^4}^4  \right)\|\p\|^2 +\frac{1}{2}\|\nabla\eta\|^2 \no\\&\qquad+\frac{3}{\nu}\|\nabla(\u-\u_d)\|^2.
	\end{align}
	We now choose the test function in the second equation in \eqref{0.1} as $\eta$ to find 
	\begin{align}\label{0.9a}
	-\frac{1}{2}\frac{\d}{\d t}\|\eta\|^2-((a+\F''(\varphi))\Delta \eta,\eta)&=-(\J*(\p\cdot\nabla\varphi),\eta)+ ((\nabla{\J} \ast  \varphi)\cdot \p,\eta) 
	\nonumber\\&\quad  -( \nabla a\cdot \p \varphi ,\eta)-( \J\ast \Delta \eta,\eta) +(\varphi - \varphi_d,\eta),
	\end{align}
	where we used $(\u \cdot\nabla \eta,\eta)=0$, since $\u$ is divergence free. Since $\frac{\partial\eta}{\partial\mathbf{n}}\big|_{\Omega}=0$, an integration by parts yields 
	\begin{align}
	&-((a+\F''(\varphi))\Delta \eta,\eta)\no\\ &=-\sum_{i=1}^2\int_{\Omega}(a(x)+\F''(\varphi(x)))\frac{\partial^2\eta(x)}{\partial x_i^2}\eta(x)\d x\no\\&=\sum_{i=1}^2\int_{\Omega}(a(x)+\F''(\varphi(x)))\left(\frac{\partial\eta(x)}{\partial x_i}\right)^2\d x+\sum_{i=1}^2\int_{\Omega}\frac{\partial}{\partial x_i}(a+\F''(\varphi(x)))\frac{\partial\eta(x)}{\partial x_i}\eta(x)\d x\no\\&\geq C_0\sum_{i=1}^2\int_{\Omega}\left(\frac{\partial\eta(x)}{\partial x_i}\right)^2\d  x+\sum_{i=1}^2\int_{\Omega}\frac{\partial}{\partial x_i}(a+\F''(\varphi(x)))\frac{\partial\eta(x)}{\partial x_i}\eta(x)\d x\no\\&=C_0\|\nabla\eta\|^2+(\nabla(a+\F''(\varphi))\cdot\nabla\eta,\eta).
	\end{align}
	Thus, from (\ref{0.9a}), we obtain 
	\begin{align}\label{0.11}
	-\frac{1}{2}\frac{\d}{\d t}\|\eta\|^2+C_0\|\nabla\eta\|^2&\leq-(\J*(\p\cdot\nabla\varphi),\eta)+ ((\nabla{\J} \ast  \varphi)\cdot \p,\eta) -( \nabla a\cdot \p \varphi ,\eta)
	\nonumber\\&\quad  -( \J\ast \Delta \eta,\eta) +(\varphi - \varphi_d,\eta)-(\nabla a\cdot\nabla\eta,\eta)-(\nabla\F''(\varphi)\cdot\nabla\eta,\eta)\no\\&=:\sum_{k=4}^{10}I_k.
	\end{align}
	Note that the properties of $\J(\cdot)$, an integration by parts, H\"older, Ladyzhenskaya and Young's inequalities yields 
	\begin{align}\label{0.12}
	|I_4|&=|((\J*\eta)\nabla\varphi,\p)|=|(\varphi(\nabla\J*\eta),\p)|\leq \|\nabla\J\|_{\mathbb{L}^1}\|\eta\|\|\p\|_{\mathbb{L}^4}\|\varphi\|_{\mathrm{L}^4}\nonumber\\&\leq \frac{1}{2}\|\nabla\J\|_{\mathbb{L}^1}^2\|\eta\|^2+\frac{1}{2}\|\p\|_{\mathbb{L}^4}^2\|\varphi\|_{\mathrm{L}^4}^2\leq \frac{1}{2}\|\nabla\J\|_{\mathbb{L}^1}^2\|\eta\|^2+\frac{1}{\sqrt{2}}\|\p\|\|\nabla\p\|\|\varphi\|_{\mathrm{L}^4}^2\nonumber\\&\leq \frac{1}{2}\|\nabla\J\|_{\mathbb{L}^1}^2\|\eta\|^2+\frac{\nu}{12}\|\nabla\p\|^2+\frac{3}{2\nu}\|\p\|^2\|\varphi\|_{\mathrm{L}^4}^4.
	\end{align}
	In a similar way, we estimate  $I_5$ as
	\begin{align}
	|I_5|\leq \|\nabla\J\|_{\mathbb{L}^1}\|\varphi\|_{\mathrm{L}^4}\|\p\|_{\mathbb{L}^4}\|\eta\|\leq \frac{1}{2}\|\nabla\J\|_{\mathbb{L}^1}^2\|\eta\|^2+\frac{\nu}{12}\|\nabla\p\|^2+\frac{3}{2\nu}\|\p\|^2\|\varphi\|_{\mathrm{L}^4}^4.
	\end{align}
	Once again using H\"older's and Young's inequalities to estimate $I_6$ as 
	\begin{align}
	|I_6|\leq \|\nabla a\|_{\mathbb{L}^{\infty}}\|\varphi\|_{\mathrm{L}^4}\|\p\|_{\mathbb{L}^4}\|\eta\|\leq \frac{1}{2}\|\nabla a\|_{\mathbb{L}^{\infty}}^2\|\eta\|^2+\frac{\nu}{12}\|\nabla\p\|^2+\frac{3}{2\nu}\|\p\|^2\|\varphi\|_{\mathrm{L}^4}^4.
	\end{align}
	We use the properties of $\J(\cdot)$, an integration by parts, Cauchy-Schwarz, H\"older and Young's inequality to estimate $I_7$ as 
	\begin{align}
	|I_7|&=|(\J\ast \eta,\Delta\eta)|=|(\nabla\J*\eta,\nabla\eta)|\leq \|\nabla\J\|_{\mathbb{L}^{1}}\|\eta\|\|\nabla\eta\|\nonumber\\&\leq \frac{C_0}{10}\|\nabla\eta\|^2+\frac{5}{2C_0}\|\nabla\J\|_{\mathbb{L}^{1}}^2\|\eta\|^2.
	\end{align}
	Using the  Cauchy-Schwarz and Young's inequalities, we estimate $I_8$ as 
	\begin{align}
	|I_8|\leq \|\varphi-\varphi_d\|_{\mathrm{V}'}\|\eta\|_{\mathrm{V}}\leq \frac{5C}{2C_0}\|\varphi-\varphi_d\|^2+\frac{C_0}{10}\|\eta\|^2+\frac{C_0}{10}\|\nabla\eta\|^2.
	\end{align}
	Similarly, we have 
	\begin{align}
	|I_{9}|\leq \|\nabla a\|_{\mathbb{L}^{\infty}}\|\nabla\eta\|\|\eta\|\leq \frac{C_0}{10}\|\nabla\eta\|^2+\frac{5}{2C_0}\|\nabla a\|_{\mathbb{L}^{\infty}}^2\|\eta\|^2.
	\end{align}
	The most difficult term, which is not possible to estimate using the weak solution regularity of $(\u,\varphi)$ is $(\nabla\F''(\varphi)\cdot\nabla\eta,\eta)$. Thus we need the solution $(\u,\varphi)$ to be a strong solution. Using  H\"older's,  Young's and Ladyzhenskaya inequalities, it is immediate that 
	\begin{align}
	\label{0.18}
	|I_{10}|&=|(\F'''(\varphi)\nabla\varphi\cdot\nabla\eta,\eta)|\leq \sup_{x\in \Omega}|\F'''(\varphi)|\|\nabla\varphi\|_{\mathbb{L}^4}\|\nabla\eta\|\|\eta\|_{\mathrm{L}^4}\nonumber\\&\leq 2^{1/4}\sup_{x\in \Omega}|\F'''(\varphi)|\|\nabla\varphi\|_{\mathbb{L}^4}\|\nabla\eta\|^{3/2}\|\eta\|^{1/2} \nonumber\\&\leq  \frac{C_0}{10}\|\nabla\eta\|^2+C(C_0)\|\F'''(\varphi)\|_{\mathrm{L}^{\infty}}^4\|\nabla\varphi\|_{\mathbb{L}^4}^4\|\eta\|^2.
	\end{align}
	Let us now combine (\ref{0.12})-(\ref{0.18}) and substitute in (\ref{0.11}), and adding with \eqref{0.9} to find
	\begin{align}
	&-\frac{1}{2}\frac{\d}{\d t}(\|\p\|^2+\|\eta\|^2)+\frac{\nu}{2}\|\nabla\p\|^2+\frac{C_0}{2}\|\nabla\eta\|^2\nonumber\\&\quad \leq \left[\frac{6}{\nu}\|\nabla\u\|^2+\frac{3}{2\nu}\left(\frac{25}{C_0^2}+3\right)\|\varphi\|_{\mathbb{L}^4}^4\right]\|\p\|^2+\frac{3}{\nu}\|\nabla(\u-\u_d)\|^2+\frac{5C}{2C_0}\|\varphi-\varphi_d\|^2\nonumber\\&\qquad+\left[\frac{1}{2}\|\nabla\J\|_{\mathbb{L}^1}^2+\frac{1}{2}\left(1+\frac{5}{C_0}\right)\left(\|\nabla a\|_{\mathbb{L}^{\infty}}^2+\|\nabla\J\|_{\mathbb{L}^{1}}^2\right)+\frac{C_0}{10}\right]\|\eta\|^2\nonumber\\&\qquad +C(C_0)\|\F'''(\varphi)\|_{\mathrm{L}^{\infty}}^4\|\nabla\varphi\|_{\mathbb{L}^4}^4\|\eta\|^2.
	\end{align}
	Integrating the above inequality from $t$ to $T$, we get 
	\begin{align}\label{0.19}
	&\|\p(t)\|^2+\|\eta(t)\|^2+\nu\int_t^T\|\nabla\p(s)\|^2\d s+C_0\int_t^T\|\nabla\eta(s)\|^2\d s\nonumber\\&\quad \leq\|\p_T\|^2+\|\eta_T\|^2+ \int_t^T\left[\frac{12}{\nu}\|\nabla\u(s)\|^2+\frac{3}{\nu}\left(\frac{25}{C_0^2}+3\right)\|\varphi(s)\|_{\mathbb{L}^4}^4\right]\|\p(s)\|^2\d s\nonumber\\&\qquad+ \left[\|\nabla\J\|_{\mathbb{L}^1}^2+\left(1+\frac{5}{C_0}\right)\left(\|\nabla a\|_{\mathbb{L}^{\infty}}^2+\|\nabla\J\|_{\mathbb{L}^{1}}^2+\frac{C_0}{5}\right)\right]\int_t^T\|\eta(s)\|^2\d s\nonumber\\&\qquad +C(C_0)\int_t^T\|\F'''(\varphi(s))\|_{\mathrm{L}^{\infty}}^4\|\nabla\varphi(s)\|_{\mathbb{L}^4}^4\|\eta(s)\|^2\d s\nonumber\\&\qquad +\frac{6}{\nu}\int_t^T\|\nabla(\u(s)-\u_d(s))\|^2\d s+\frac{5C}{C_0}\int_t^T\|\varphi(s)-\varphi_d(s)\|^2\d s.
	\end{align}
	An application of  Gronwall's inequality in (\ref{0.19}) yields 
	\begin{align}\label{0.20}
	&\|\p(t)\|^2+\|\eta(t)\|^2\no\\&\quad\leq\left(\|\p_T\|^2+\|\eta_T\|^2+\frac{6}{\nu}\int_0^T\|\nabla(\u(t)-\u_d(t))\|^2\d t+\frac{5C}{C_0}\int_0^T\|\varphi(t)-\varphi_d(t)\|^2\d t\right)\nonumber\\&\quad \times \exp\left\{\left[\|\nabla\J\|_{\mathbb{L}^1}^2+\left(1+\frac{5}{C_0}\right)\left(\|\nabla a\|_{\mathbb{L}^{\infty}}^2+\|\nabla\J\|_{\mathbb{L}^{1}}^2\right)+\frac{C_0}{5}\right]T\right\}\nonumber\\&\quad \times\exp\left\{\int_0^T\left[\frac{12}{\nu}\|\nabla\u(t)\|^2+\frac{3}{\nu}\left(\frac{25}{C_0^2}+3\right)\|\varphi(t)\|_{\mathbb{L}^4}^4\right]\d t\right\}\nonumber\\&\quad \times \exp\left(C(C_0)\sup_{t\in[0,T]}\|\F'''(\varphi(t))\|_{\mathrm{L}^{\infty}}^4\int_t^T\|\nabla\varphi(t)\|_{\mathbb{L}^4}^4\d t\right),
	\end{align}
	for all $t\in[0,T]$. Since $(\u,\varphi)$ is the unique strong solution of the nonlinear system, the right hand side of the inequality (\ref{0.20}) is finite. Note that $\varphi\in\mathrm{L}^{\infty}((0,T)\times\Omega)\cap\mathrm{L}^{\infty}(0,T;\mathrm{W}^{1,p})\cap\mathrm{L}^2(0,T;\mathrm{H}^2(\Omega)),$ for $2\leq p<\infty$. Thus, we get $\p\in\mathrm{L}^{\infty}(0,T;\G_{\text{div}})$ and $\eta\in\mathrm{L}^{\infty}(0,T;\mathrm{H})$. Let us substitute (\ref{0.20}) in (\ref{0.19}) to obtain the regularity of $\p$ and $\eta$  namely, $\p\in\mathrm{L}^2(0,T;\V_{\text{div}})$ and $\eta\in\mathrm{L}^2(0,T;\mathrm{V})$. 
	
	To show that $\eta \in \mathrm{L}^{\infty}(0,T;\mathrm{V})\cap\mathrm{L}^2(0,T;\mathrm{H}^2)$, we take the test function in the second equation in \eqref{0.1} as $-\Delta\eta$ to find  
	\begin{align}\label{p11}
	&-\frac{1}{2}\frac{\d}{\d t}\|\nabla\eta\|^2+((a+\F''(\varphi))\Delta\eta,\Delta\eta)=( \J \ast (\p \cdot \nabla \varphi),\Delta\eta) - ((\nabla{\J} \ast  \varphi)\cdot \p,\Delta\eta)\nonumber \\  &\quad+ (\nabla a\cdot \p \varphi ,\Delta\eta)
	-( \u \cdot\nabla \eta,\Delta\eta) +( \J\ast \Delta \eta,\Delta\eta)+(\varphi - \varphi_d,\Delta\eta) =:\sum_{i=11}^{16}I_i.
	\end{align}
	We estimate $|I_{11}|$ using the H\"older, Ladyzhenskaya, Poincar\'e and Young's inequalities as
	\begin{align}\label{p13}
	|I_{11}|&\leq \|\J*\p\|_{\mathbb{L}^4}\|\nabla\varphi\|_{\mathbb{L}^4}\|\Delta\eta\|\leq\frac{C_0}{12}\|\Delta\eta\|^2+\frac{3}{C_0} \|\J\|_{\mathrm{L}^1}^2\|\p\|_{\mathbb{L}^4}^2\|\nabla\varphi\|_{\mathbb{L}^4}^2\nonumber\\&\leq \frac{C_0}{12}\|\Delta\eta\|^2+\frac{3\sqrt{2}}{C_0} \|\J\|_{\mathrm{L}^1}^2\|\p\|\|\nabla\p\|\|\nabla\varphi\|_{\mathbb{L}^4}^2\nonumber\\&\leq \frac{C_0}{12}\|\Delta\eta\|^2+ \frac{\nu}{12}\|\nabla\p\|^2 +\frac{36}{\nu C_0^2}\|\J\|_{\mathrm{L}^1}^4\|\nabla\varphi\|_{\mathbb{L}^4}^4\|\p\|^2,
	\end{align}
	where we also used the Young's inequality for convolutions.  Similarly, we estimate  $|I_{12}|$  as
	\begin{align}\label{p14}
	|I_{12}|&\leq \|\nabla{\J} \ast  \varphi\|_{\mathbb{L}^4}\|\p\|_{\mathbb{L}^4}\|\Delta\eta\|
	\leq\frac{C_0}{12}\|\Delta\eta\|^2+\frac{3}{C_0} \|\nabla\J\|_{\mathbb{L}^1}^2\|\varphi\|_{\mathrm{L}^4}^2\|\p\|_{\mathbb{L}^4}^2
	\nonumber\\&\leq \frac{C_0}{12}\|\Delta\eta\|^2+ \frac{\nu}{12}\|\nabla\p\|^2 +\frac{36}{\nu C_0^2}\|\nabla \J\|_{\mathrm{L}^1}^4\|\nabla\varphi\|_{\mathbb{L}^4}^4\|\p\|^2.
	\end{align}
	We use the  H\"older, Ladyzhenskaya, Poincar\'e and Young's inequalities to estimate $|I_{13}|$ as
	\begin{align}\label{p15}
	|I_{13}|&\leq \|\nabla a\|_{\mathbb{L}^{\infty}}\|\p\|_{\mathbb{L}^4}\|\varphi\|_{\mathrm{L}^4}\|\Delta\eta\|\leq \frac{C_0}{12}\|\Delta\eta\|^2+\frac{3}{C_0}\|\nabla a\|_{\mathbb{L}^{\infty}}^2\|\p\|_{\mathbb{L}^4}^2\|\varphi\|^2_{\mathrm{L}^4}\nonumber\\&\leq \frac{C_0}{12}\|\Delta\eta\|^2+ \frac{\nu}{12}\|\nabla\p\|^2 +\frac{36}{\nu C_0^2}\|\nabla a\|_{\mathbb{L}^{\infty}}^4\|\nabla\varphi\|_{\mathbb{L}^4}^4\|\p\|^2.
	\end{align}
	Using the H\"older, Agmon and Young's inequalities to estimate $|I_{14}|$ as 
	\begin{align}\label{p16}
	|I_{14}| \leq \|\u\|_{\mathbb{L}^{\infty}}\|\nabla\eta\|\|\Delta\eta\|\leq C\|\u\|_{\H^2}\|\nabla\eta\|\|\Delta\eta\|\leq \frac{C_0}{12}\|\Delta\eta\|^2+C(C_0)\|\u\|_{\mathbb{H}^2}^2\|\nabla\eta\|^2.
	\end{align}
	We use an integration by parts, Cauchy-Schwarz inequality, convolution inequality and Young's inequality  to estimate $|I_{15}|$ as 
	\begin{align}\label{p17}
	|I_{15}|&=\left|\sum_{i=1}^2\int_{\Omega}\left(\int_{\Omega}\J(x-y)\Delta\eta(y)\d y\right)\frac{\partial^2\eta(x)}{\partial x_i^2}\d x\right|\nonumber\\&=\left|\sum_{i=1}^2\int_{\Omega}\frac{\partial}{\partial x_i}\left(\int_{\Omega}\J(x-y)\Delta\eta(y)\d y\right)\frac{\partial\eta(x)}{\partial x_i}\d x\right|=|(\nabla\J*\Delta\eta,\nabla\eta)|\nonumber\\&\leq \|\nabla\J\|_{\mathbb{L}^1}\|\Delta\eta\|\|\nabla\eta\|\leq  \frac{C_0}{12}\|\Delta\eta\|^2+\frac{3}{C_0}\|\nabla\J\|_{\mathbb{L}^1}^2\|\nabla\eta\|^2.
	\end{align}
	Finally, $|I_{16}|$ can be estimated using the Cauchy-Schwarz and Young's inequalities as 
	\begin{align}\label{p18}
	|I_{16}|\leq \|\varphi-\varphi_d\|\|\Delta\eta\|\leq \frac{C_0}{12}\|\Delta\eta\|^2+\frac{3}{C_0}\|\varphi-\varphi_d\|^2.
	\end{align}
	But we know that (see Assumption \ref{prop of F and J} (2))
	\begin{align}\label{p12}
	C_0\|\Delta\eta\|^2\leq ((a+\F''(\varphi))\Delta\eta,\Delta\eta).
	\end{align}
	Combining (\ref{p13})-(\ref{p12}) and substituting it in (\ref{p11}) to find 
	\begin{align}\label{p19}
	-\frac{1}{2}\frac{\d}{\d t}\|\nabla\eta\|^2&+\frac{C_0}{2}\|\Delta\eta\|^2\leq \left[\frac{18}{\nu C_0^2}\|\J\|_{\mathrm{L}^1}^4\|\nabla\varphi\|_{\mathbb{L}^4}^4 + \frac{36}{\nu C_0^2}\|\nabla a\|_{\mathbb{L}^{\infty}}^4\|\nabla\varphi\|_{\mathbb{L}^4}^4\right]\|\p\|^2\nonumber\\& +C(C_0)\left(\|\u\|_{\mathbb{H}^2}^2+\|\nabla\J\|_{\mathbb{L}^1}^2\right)\|\nabla\eta\|^2+ \frac{\nu}{4}\|\nabla\p\|^2 +\frac{3}{C_0}\|\varphi-\varphi_d\|^2.
	\end{align}
	Let us integrate the inequality \eqref{p19} from $t$ to $T$ to find 
	\begin{align}\label{p21}
	&\|\nabla\eta(t)\|^2+C_0\int_t^T\|\Delta\eta(s)\|^2\d s\nonumber\\&\leq \|\nabla\eta_T\|^2+\frac{6}{C_0}\int_t^T\|\varphi(s)-\varphi_d(s)\|^2\d s+C(C_0)\int_t^T\left(\|\u(s)\|_{\mathbb{H}^2}^2+\|\nabla\J\|_{\mathbb{L}^1}^2\right)\|\nabla\eta(s)\|^2\d s\nonumber\\&\quad +\int_t^T\left[\frac{36}{\nu C_0^2}\|\J\|_{\mathrm{L}^1}^4 + \frac{72}{\nu C_0^2}\|\nabla a\|_{\mathbb{L}^{\infty}}^4\right]\|\nabla\varphi(s)\|_{\mathbb{L}^4}^4\|\p(s)\|^2\d s + \frac{\nu}{2}\int_t^T\|\nabla\p(s)\|^2 \d s. 
	\end{align}
	An application of the Gronwall inequality in (\ref{p21}) yields 
	\begin{align}\label{p22}
	\|\nabla\eta(t)\|^2&\leq \left[\|\nabla\eta_T\|^2 +\frac{12}{C_0}\left(\int_0^T\|\varphi(t)\|^2\d t+\int_0^T\|\varphi_d(t)\|^2\d t\right)+ \frac{\nu}{2}\int_0^T\|\nabla\p(t)\|^2 \d t\right.\nonumber\\&\qquad \left. +\left[\frac{36}{\nu C_0^2}\|\J\|_{\mathrm{L}^1}^4 + \frac{72}{\nu C_0^2}\|\nabla a\|_{\mathbb{L}^{\infty}}^4\right]\sup_{t\in[0,T]}\|\p(t)\|^2\int_0^T\|\nabla\varphi(t)\|_{\mathbb{L}^4}^4\d t \right]\nonumber\\&\quad \times\exp\left(C(C_0)\int_0^T\left(\|\u(t)\|_{\mathbb{H}^2}^2+\|\nabla\J\|_{\mathbb{L}^1}^2\right)\d t\right),
	\end{align}
	for all $t\in[0,T]$. Note that the right hand side of the above inequality is  finite, since $(\u,\varphi)$ is the unique strong solution of the system \eqref{nonlin phi}-\eqref{nonlin u}, $\eta_T\in\mathrm{V}$ and $\p\in\mathrm{L}^{\infty}(0,T;\G_{\text{div}})\cap \mathrm{L}^{2}(0,T;\V_{\text{div}})$. Thus, we have from previous estimate $\eta\in\mathrm{L}^{\infty}(0,T;\mathrm{V})$. From \eqref{p21}, we also have $\eta\in\mathrm{L}^2(0,T;\mathrm{H}^2)$. Now it is immediate that  $\p_t\in\mathrm{L}^2(0,T;\V_{\text{div}}')$ and $\eta_t\in\mathrm{L}^2(0,T;\mathrm{H})$.
	Thus $\p$ is almost everywhere equal to a continuous function from $[0,T]$ to $\V_{\text{div}}$ and $\eta$ is almost everywhere equal to a continuous function from $[0,T]$ to $\mathrm{V}$. Thus as in Theorem \ref{linearized}, one can also get $\p\in\C([0,T];\G_{\text{div}})$ and $\eta\in\C([0,T];\mathrm{V})$ and the left continuity at $T$ implies that $\p(T)=\p_T\in\G_{\text{div}}$ and $\eta(t)=\eta_T\in\mathrm{V}$ in the weak sense. 	Hence $(\p,\eta)$ is a weak solution of the system  \eqref{adj} satisfying \eqref{0.1} with the regularity given in \eqref{space}. The uniqueness of weak solution follows from the linearity of  the system \eqref{adj}. 
\end{proof}

	\subsection{Pontryagin maximum principle}
	In this subsection, we prove Pontryagin maximum principle  for the  optimal control problem \eqref{control problem}. 
	Pontryagin Maximum principle  gives a first order necessary condition for the optimal control problems having differential equations as constraints. We also characterize the optimal control in terms of the adjoint variables, which we obtained in the previous subsection.
	Even though we published the subsection title as Pontryagin maximum principle, our problem is a minimization of the cost functional given in (\ref{cost}) and hence we obtain  a  minimum principle. The following minimum principle is satisfied by the optimal triplet (if it exists) $(\u^*,\varphi^*,\U^*)\in\mathscr{A}_{\text{ad}}$:
	\begin{eqnarray}
	\frac{1}{2}\|\U^*(t)\|^2+( \p(t),\U^*(t) ) \leq \frac{1}{2}\|\mathrm{W}\|^2+( \p(t),\mathrm{W} ),
	\end{eqnarray} 
	for all  $\mathrm{W}\in\mathcal{U},$ and a.e. $t\in[0,T]$. 
	
	Let us now recall the definition of subgradient  of a function and its subdifferential, which is useful in the sequel. Let $\X$ be a real Banach space, $\X'$ be its topological dual and $\langle\cdot,\cdot\rangle_{\X'\times\X}$ be the duality pairing between $\X'$ and $\X$. 
	\begin{definition}[Subgradient, Subdifferential] Let $f:\X\to(-\infty,\infty],$  a functional on $\X$. A linear functional $u'\in\X'$  is called \emph{subgradient} of $f$ at $u$ if $f(u)\neq +\infty$ and for all $v \in\X$
		$$f(v)\geq f(u)+\langle u',v-u\rangle_{\X'\times \X},$$ holds. The set of all subgradients of $f$ at u is called \emph{subdifferential} $\partial f(u)$ of $f$ at $u$.
	\end{definition}

	\begin{theorem}[Pontryagin Minimum Principle]\label{main}
		Let $(\u^*,\varphi^*,\U^*)\in\mathscr{A}_{\text{ad}}$ be the optimal solution of the Problem \ref{control problem}. Then there exists a unique weak solution $(\p,\eta)$ of the adjoint system \eqref{adj} and for almost every $t \in [0,T]$ and $\mathrm{W} \in \mathcal{U}$, we have 
		\begin{eqnarray}\label{3.41}
		\frac{1}{2}\|\U^*(t)\|^2 +( \p(t),\U^*(t) ) \leq \frac{1}{2}\|\mathrm{W}\|^2 +( \p(t),\mathrm{W} ).  
		\end{eqnarray}
	\end{theorem}
	The proof of Theorem \ref{main} follows as a corollary of Theorem \ref{3.10} below. Note that one can rewrite \eqref{3.41} as
	\begin{eqnarray}
	( -\p(t),  \mathrm{W} -\U^*(t)) \leq \frac{1}{2}\|\mathrm{W}\|^2 - \frac{1}{2}\|\U^*(t)\|^2,\ \mbox{  a.e.,   }  \ t \in [0,T].  
	\end{eqnarray}
	Note that $-\p \in \partial \frac{1}{2}\|{\U^*}(t)\|^2$, where $\partial$ denotes the subdifferential. Since, $\frac{1}{2}\|\cdot\|^2$ is G\^ateaux differentiable, the subdifferential consists of a single point and it follows that
	\begin{equation*}
	-\p(t) = \U^*(t), \ \text{ a.e.}\  t \in [0,T].
	\end{equation*}
If we take $ \mathscr{U}_{ad}$ as in example \ref{example}; then we get 
$$ ( \p(t) + \U^*(t), \U(t) ) = 0 ;\ \text{ for all } \ \U \in \mathcal{U},  \mbox{  a.e.,   }  \ t \in [0,T] .$$
Since the above equality is true for all $\U \in \mathcal{U}$, we get $\p(t) + \U^*(t)\in\mathcal{U}^{\perp}$, a.e. $t\in[0,T]$. Since $(\U^*+\p,\U) =0$, for all $\U\in\mathcal{U}$, we obtain $\U^*$ is the orthogonal projection of $-\p$ onto $\mathcal{U}$, i.e., $\U^*(t)=\Pi_{\mathcal{U}}(-\p(t))$, a.e. $t\in[0,T]$, where $\Pi_{\mathcal{U}}$ is the projection onto $\mathcal{U}$. (see \cite{BDM}). 

	Equivalently the above minimum principle may be  written in terms of the \emph{Hamiltonian formulation}. 
	Let us first define the \emph{Lagrangian} by
	$$\mathscr{L}(\u,\varphi,\U) = \frac{1}{2} \left[\|\nabla(\u-\u_d)\|^2 + \|\varphi - \varphi_d\|^2+ \|\U\|^2\right].$$
	Then, we can define the corresponding \emph{Hamiltonian} by
	$$\mathscr{H}(\u,\varphi,\U,\p,\eta)= \mathscr{L}(\u,\varphi,\U) + \langle \p, \mathscr{N}_1(\u,\varphi,\U) \rangle + \langle \eta, \mathscr{N}_2(\u,\varphi)\rangle,$$ where $\mathscr{N}_1$ and $\mathscr{N}_2$ are defined by (\ref{n1n2}).
	Hence, we have the following Pontryagin minimum principle:
	\begin{eqnarray}
	\mathscr{H}(\u^*(t),\varphi^*(t),\U^*(t),\p(t),\eta(t)) \leq \mathscr{H}(\u^*(t),\varphi^*(t),\mathrm{W},\p(t),\eta(t)),
	\end{eqnarray}
	for all $\mathrm{W} \in \G_{\text{div}}$, \ a.e., \ $t\in[0,T]$ along with the adjoint system \eqref{adj}. 
	
	In order to prove Theorem \ref{main}, we first prove the existence of approximate minimizer of the cost functional which is essence of Ekeland's variational principle.
	In the rest of this section,  we use Ekeland's variational principle to prove the existence of an approximate optimal triplet to the problem (\ref{control problem}).  
	
	The \emph{Ekeland variational principle} in general Banach space can be stated as:
	\begin{theorem}[Ekeland variational principle, Theorem 1.1,  \cite{EKE}]\label{th1}
		Let $(\mathbb{X},d)$ be a complete metric space, and $\mathcal{F}: \mathbb{X}\rightarrow \mathbb{R}\cup\{+\infty\}$ a lower semi-continuous function, not identically $+\infty$, and bounded from below. Then for every point $u\in \mathbb{X}$ such that
		\begin{equation}\label{emin}
		\inf_{\mathbb{X}} \mathcal{F}\le \mathcal{F}(u)\le \inf_{\mathbb{X}}\mathcal{F}+\e,
		\end{equation}
		and every $\lambda>0$, there exists some point $v\in \mathbb{X}$ such that
		\begin{equation*}\left\{\begin{aligned}
		\mathcal{F}(v)&\le \mathcal{F}(u),\\
		d(u,v)&\le \lambda,\\
		\forall w\neq v,\ \mathcal{F}(w)&>\mathcal{F}(v)-(\e/\lambda)d(v,w).
		\end{aligned}\right.\end{equation*}
		A point $u \in \mathbb{X}$ such that \eqref{emin} holds is called an \emph{$\varepsilon-$minimizer} of $\mathcal{F}$.
	\end{theorem}

	Now we define the so called \emph{spike variation} in the control which helps us to establish the existence of an approximate optimal triplet to the problem \eqref{control problem}. 
	\begin{definition}
		Choose $\tau \in (0,T)$. Let $\W \in \G_{\text{div}}$ be arbitary and $h$ be chosen such that $0<h \leq \tau$. The spike variation $\U_{\tau,h,\W}$ of a control $\U \in \mathscr{U}_{\text{ad}}$, the set of admissible controls, is defined by
		\begin{equation}\label{3.15}
		\U_{\tau,h,\W}(t) =\begin{cases}
		\W, & \text{if $t \in [0,T] \cap (\tau -h, \tau)$},\\
		\U(t), & \text{if $t \notin [0,T] \cap (\tau -h, \tau)$}.
		\end{cases}
		\end{equation}
	\end{definition}
	For brevity, we denote the spike variation as $\U^h$, in which case $\tau$ and $\W$ are assumed to be fixed. 
	Recalling that the set of admissible controls $\mathscr{U}_{\text{ad}}$ is  consisting of controls $\U \in\mathrm{L}^{2}(0,T;\G_{\text{div}})$, we equip $\mathscr{U}_{\text{ad}}$ with a metric $d$, given by
	\begin{align*}
	d(\u,\v) = \text{meas}\big\{t \in [0,T] : \u(t) \neq \v(t)\big\}, 
	\end{align*}
	for all $\u,\v \in \mathscr{U}_{\text{ad}}$, where we use the Lebegue measure on $\mathbb{R}$. This is called the \emph{ Ekeland metric} (see \cite{EKE}). As proved in \cite{EKE}, it can be easily shown that $(\mathscr{U}_{\text{ad}},d)$ is a complete metric space. We can observe that $\U^h \in \mathscr{U}_{\text{ad}}$ and $d(\U^h,\U) = h$. Furthermore, one can verify  that $\U^h \rightarrow \U$ strongly in $\mathrm{L}^2(0,T;\G_{\text{div}})$ as $h \rightarrow 0$.
	\begin{definition}
		Let $\mathbb{X}$ be Banach space. A \emph{Lebesgue point} of an integrable $\mathbb{X}$-valued function $f(t)$ defined on $0\leq t\leq T$ is any point $\tau$, where (two sided limit)$$\lim_{h \rightarrow 0}\frac{1}{h}\int_{\tau-h}^{\tau}\|f(s)-f(\tau)\|_{\mathbb{X}}\d s=0.$$ 
	\end{definition}
	Almost every $t\in[0,T]$ is a Lebesgue point of each component of $f(\cdot)$, and a Lebesgue point of all components is a Lebesgue point of $f(\cdot)$: thus almost every $t\in[0,T]$ is a Lebesgue point of $f(\cdot)$. Since we take only limits as $h\to 0^+$; points where $\lim\limits_{h\to 0^+}$ exists are \emph{left Lebesgue points}. 
	\begin{theorem}[\emph{$\varepsilon$-optimal solution for \eqref{control problem}} and minimum principle] \label{3.10} 
		There exist an $\varepsilon$-optimal solution $(\u_{\varepsilon},\varphi_{\varepsilon},\U_{\varepsilon})$ of the problem \ref{control problem}, in the sense that 
		\begin{align}
		\mathcal{J}(\u_{\varepsilon},\varphi_{\varepsilon},\U_{\varepsilon}) \leq \inf_{(\u,\varphi,\U) \in \mathcal{A}_{\text{ad}}}\mathcal{J}(\u,\varphi,\U) + \varepsilon.
		\end{align}
		Furthermore, there exist a weak solution $(\p,\eta)$ to the adjoit system \eqref{adj} and for almost every $t \in [0,T]$ and $\mathrm{W} \in \G_{\text{div}}$, we have 
		\begin{align*}
		\frac{1}{2}\|\U_{\varepsilon}(t)\|^2 -\langle \p(t),\U_{\varepsilon}(t) \rangle \leq \frac{1}{2}\|\mathrm{W}\|^2 -\langle \p(t),\mathrm{W} \rangle + \varepsilon.  
		\end{align*}
	\end{theorem}
	\begin{proof}
		The cost functional $\mathcal{J}$ is lower semi-continuous on the admissiable class $\mathcal{A}_{\text{ad}}$, and is bounded below. So by Ekeland's variational principle (see Theorem \ref{th1}), we get that for all $\varepsilon >0$, there exists $(\u_{\varepsilon},\varphi_{\varepsilon},\U_{\varepsilon}) \in \mathcal{A}_{\text{ad}}$ such that
		\begin{align*}
		\mathcal{J}(\u_{\varepsilon},\varphi_{\varepsilon},\U_{\varepsilon}) \leq \inf_{(\u,\varphi,\U) \in \mathcal{A}_{\text{ad}}}\mathcal{J}(\u,\varphi,\U) + \varepsilon,
		\end{align*}
		and for all $\U \in \mathscr{U}_{\text{ad}}$, we have 
		\begin{align}\label{3.18}
		\mathcal{J}(\u,\varphi,\U) \geq  \mathcal{J}(\u_{\varepsilon},\varphi_{\varepsilon},\U_{\varepsilon}) - \varepsilon d(\U, \U_{\varepsilon}),
		\end{align}
		where $(\u,\varphi)$ is the solution of \eqref{nonlin phi}-\eqref{initial conditions} with control $\U$.
		Since \eqref{3.18} holds for any solution $(\u,\varphi,\U)$, we choose the spike variation $\U^h_{\varepsilon}$ of $\U_{\varepsilon}$ with corresponding trajectory $(\u^h_{\varepsilon},\varphi^h_{\varepsilon})$ and from \eqref{3.18}, we deduce 
		\begin{align}\label{ep}
		-\varepsilon &\leq \frac{1}{h} [\mathcal{J}(\u_{\varepsilon}^h,\varphi_{\varepsilon}^h,\U_{\varepsilon}^h) - \mathcal{J}(\u_{\varepsilon},\varphi_{\varepsilon},\U_{\varepsilon})] \nonumber\\
		&= \frac{1}{2h} \int_0^T\left( \|\nabla(\u_{\varepsilon}^h(t)-\u_d(t))\|^2  -  \|\nabla(\u_{\varepsilon}(t)-\u_d(t))\|^2\right) \d t \nonumber\\
		&\quad+\frac{1}{2h} \int_0^T\left( \|\varphi_{\varepsilon}^h(t) -\varphi_d(t)\|^2  -  \|\varphi_{\varepsilon}(t) -\varphi_d(t)\|^2\right)\d t +\frac{1}{2h} \int_0^T\left(\|\U_{\varepsilon}^h(t)\|^2 - \|\U_{\varepsilon}(t)\|^2\right)\d t\nonumber \\ &\quad+\frac{1}{2h}\left[\|\u_{\e}^h(T)-\u_f\|^2-\|\u_{\e}(T)-\u_f\|^2\right]+\frac{1}{2h}\left[\|\varphi_{\e}^h(T)-\varphi_f\|^2-\|\varphi_{\e}(T)-\varphi_f\|^2\right]\nonumber\\
		&=: \sum_{j=1}^5 I_j^h,
		\end{align}
		where the $I_j^h$ correspond to the terms in the right hand side of (\ref{ep}). 
		Now we obtain the limits as $h \rightarrow 0$ of $I_j^h(j = 1,\ldots,5)$. We know by the definition of spike variation that $(\u^h_{\varepsilon},\varphi^h_{\varepsilon}) = (\u_{\varepsilon},\varphi_{\varepsilon} )$ on $(0,\tau -h)$, since the controls coincide. Therefore we deduce	from (\ref{ep}) that 
		\begin{align}\label{3.20}
		\lim_{h \rightarrow 0}I_1^h &= \lim_{h \rightarrow 0} \left[\frac{1}{2h} \int_{\tau -h}^\tau \left(\|\nabla \u_{\varepsilon}^h(t)\|^2  -  \|\nabla \u_{\varepsilon}(t)\|^2\right) \d t - \frac{1}{h} \int_{\tau -h}^\tau (\nabla(\u_{\varepsilon}^h(t) - \u_{\varepsilon}(t)), \nabla \u_d(t)) \d t\right] \nonumber \\
		&\quad + \lim_{h \rightarrow 0} \left[\frac{1}{2h} \int_{\tau}^T\left( \|\nabla \u_{\varepsilon}^h(t)\|^2  -  \|\nabla \u_{\varepsilon}(t)\|^2\right) \d t - \frac{1}{h} \int_{\tau}^T (\nabla(\u_{\varepsilon}^h(t) - \u_{\varepsilon}(t)), \nabla\u_d(t)) \d t\right] \nonumber \\
		&= \lim_{h \rightarrow 0} \frac{1}{2} \int_{\tau}^T \left(\frac{\nabla(\u_{\varepsilon}^h(t) - \u_{\varepsilon}(t))}{h}, \nabla(\u_{\varepsilon}^h(t) +  \u_{\varepsilon}(t) -2\u_d) \right) \d t,
		\end{align}
		where we used the left Lebesgue convergence theorem to get rid of the integral supported on $(\tau-h,\tau)$. We know that  \begin{align}\label{ep1}\frac{\u_{\varepsilon}^h(t) - \u_{\varepsilon}(t)}{h} \rightarrow \w_{\varepsilon}(t)\text{ and }\u^h_{\varepsilon}(t) \rightarrow \u_{\varepsilon}(t)\text{ as } h \rightarrow 0,\end{align} uniformly for all $\tau\leq t\leq T$  in $\G_{\text{div }}$ (see Lemmas \ref{lem3.12} and \ref{lem3.13} below). Thus, we have 
		\begin{align}\label{3.25a}
		&\left|\frac{1}{2}\int_{\tau}^{T}\left(\frac{\nabla(\u_{\varepsilon}^h(t) - \u_{\varepsilon}(t))}{h}, \nabla(\u_{\varepsilon}^h(t) +  \u_{\varepsilon}(t) -2\u_d) \right) \d t- \int_{\tau}^{T}(\nabla  \w_{\varepsilon}(t), \nabla  (\u_{\varepsilon} (t)- \u_d(t)))\d t \right|\nonumber\\& \leq \left|\int_{\tau}^{T}\left(\frac{\nabla(\u_{\varepsilon}^h(t) - \u_{\varepsilon}(t))}{h}-\nabla  \w_{\varepsilon} (t), \frac{1}{2}\nabla(\u_{\varepsilon}^h(t) +  \u_{\varepsilon}(t) -2\u_d) \right)\d t\right|\nonumber\\&\quad +\left|\int_{\tau}^{T}\left(\nabla\w_{\varepsilon} (t),\frac{1}{2}\nabla(\u_{\varepsilon}^h(t) +  \u_{\varepsilon}(t) -2\u_d) -\nabla  (\u_{\varepsilon} (t)- \u_d(t))\right)\d t\right|\nonumber\\& \leq \frac{1}{2}\left(\int_{\tau}^{T}\left\|\nabla\left(\frac{\u_{\varepsilon}^h(t) - \u_{\varepsilon}(t)}{h}-\w^{\e}(t)\right)\right\|^2\d t\right)^{1/2}\left(\int_{\tau}^{T}\left\|\nabla(\u_{\varepsilon}^h(t) +  \u_{\varepsilon}(t) -2\u_d)\right\|^2\d t\right)^{1/2}\nonumber\\&\quad +\frac{1}{2}\left(\int_{\tau}^{T}\left\|\nabla\w_{\varepsilon} (t)\right\|^2\d t\right)^{1/2}\left(\int_{\tau}^{T}\left\|\nabla(\u_{\varepsilon}^h(t) -  \u_{\varepsilon}(t) )\right\|^2\d t\right)^{1/2}.
		\end{align}
		Passing to the limit in \eqref{3.25a}, using (\ref{ep1}), the regularity of $\u_{\e}^h$, $\u_{\e}$ and $\w_{\varepsilon}$, and using an integration by parts, we derive 
		\begin{align*}
		\lim_{h \rightarrow 0}I_1^h = \int_{\tau}^{T}(\nabla  \w_{\varepsilon}(t), \nabla  (\u_{\varepsilon} (t)- \u_d(t)))\d t = \int_{\tau}^{T}\langle\w_{\varepsilon}(t), -\Delta (\u_{\varepsilon}(t) - \u_d(t))\rangle\d t, 
		\end{align*}
		where we used the fact that $(\u_{\varepsilon} - \u_d)\cdot\mathbf{n}\big|_{\partial\Omega}=0$. Now, we consider  
		\begin{align}
		\lim_{h \rightarrow 0}I_2^h &= \lim_{h \rightarrow 0}\left[ \frac{1}{2h} \int_{\tau -h}^\tau \left(\| \varphi_{\varepsilon}^h(t)\|^2  -  \| \varphi_{\varepsilon}(t)\|^2\right) \d t - \frac{1}{h} \int_{\tau -h}^\tau (\varphi_{\varepsilon}^h(t) - \varphi_{\varepsilon}(t),  \varphi_d(t)) \d t\right] 
		\nonumber\\&\quad + \lim_{h \rightarrow 0} \left[\frac{1}{2h} \int_{\tau}^T \left(\| \varphi_{\varepsilon}^h(t)\|^2  -  \| \varphi_{\varepsilon}(t)\|^2 \right)\d t - \frac{1}{h} \int_{\tau}^T (\varphi_{\varepsilon}^h(t) - \varphi_{\varepsilon}(t),  \varphi_d(t)) \d t \right]\nonumber\\
		&= \lim_{h \rightarrow 0} \frac{1}{2} \int_{\tau}^T \left(\frac{\varphi_{\varepsilon}^h(t) - \varphi_{\varepsilon}(t)}{h},\varphi_{\varepsilon}^h(t) +  \varphi_{\varepsilon}(t) -2\varphi_d \right) \d t.
		\end{align}
		A calculation similar to \eqref{3.25a}, we obtain 
		\begin{align}\label{3.25b}
		&\left|\frac{1}{2} \int_{\tau}^T \left(\frac{\varphi_{\varepsilon}^h(t) - \varphi_{\varepsilon}(t)}{h},\varphi_{\varepsilon}^h(t) +  \varphi_{\varepsilon}(t) -2\varphi_d \right) \d t-\int_{\tau}^T(\psi_{\varepsilon}(t),\varphi_{\varepsilon}(t) - \varphi_d(t) )\d t\right|\nonumber\\&\leq 
		\frac{1}{2}\left(\int_{\tau}^{T}\left\|\frac{\varphi_{\varepsilon}^h(t) - \varphi_{\varepsilon}(t)}{h}-\psi^{\e}(t)\right\|^2\d t\right)^{1/2}\left(\int_{\tau}^{T}\left\|\varphi_{\varepsilon}^h(t) +  \varphi_{\varepsilon}(t) -2\varphi_d\right\|^2\d t\right)^{1/2}\nonumber\\&\quad +\frac{1}{2}\left(\int_{\tau}^{T}\left\|\psi_{\varepsilon} (t)\right\|^2\d t\right)^{1/2}\left(\int_{\tau}^{T}\left\|\varphi_{\varepsilon}^h(t) -  \varphi_{\varepsilon}(t)\right\|^2\d t\right)^{1/2}.
		\end{align}
		Using the fact that (see Lemma \ref{lem3.13} and Remark \ref{rem3.15} below)
		\begin{align}\label{ep2}
		\frac{\varphi_{\varepsilon}^h(t)- \varphi_{\varepsilon}(t)}{h} \rightarrow \psi_{\varepsilon}(t) \ \text{ and }\ \varphi^h_{\varepsilon}(t) \rightarrow \varphi_{\varepsilon} (t)\text{ as }h \rightarrow 0,
		\end{align} 
		uniformly for all $\tau\leq t\leq T$ in $\mathrm{H}$, from  \eqref{3.25b}, we obtain
		\begin{align}
		\lim_{h \rightarrow 0}I_2^h = \int_{\tau}^T(\psi_{\varepsilon}(t),\varphi_{\varepsilon}(t) - \varphi_d(t) )\d t.
		\end{align}
		Now by the definition of spike variation, we know that $\U^h_{\e}=\W$ for $t\in(\tau-h,\tau)$ and $\U^h_{\e}=\U^{\e}$, elsewhere, so that we obtain 
		\begin{align*}
		\lim_{h \rightarrow 0}I_3^h &= \lim_{h \rightarrow 0} \frac{1}{2h} \int_{\tau -h}^\tau \left(\|\W\|^2 - \|\U_{\varepsilon}(t)\|^2\right)\d t 
		= \frac{1}{2} \|\W\|^2 - \frac{1}{2} \|\U_{\varepsilon}(\tau)\|^2,
		\end{align*}
		where $\tau$ is a (left) Lebesgue point. Using (\ref{ep1}) and (\ref{ep2}), it is also clear that 
		\begin{align*}
		\lim_{h \rightarrow 0}I^h_4=	\lim_{h \rightarrow 0}\frac{1}{2}\left(\frac{\u^h_{\e}(T)-\u_{\e}(T)}{h},\u^h_{\e}(T)+\u_{\e}(T)-2\u_f\right)&=(\w_{\varepsilon}(T),\u_{\e}(T)-\u_f),\\ \lim_{h \rightarrow 0}I^h_5=\lim_{h \rightarrow 0}\frac{1}{2}\left(\frac{\varphi^h_{\e}(T)-\varphi_{\e}(T)}{h},\varphi^h_{\e}(T)+\varphi_{\e}(T)-2\varphi_f\right)&=(\psi_{\varepsilon}(T),\varphi_{\e}(T)-\varphi_f).
		\end{align*}
		Combining all the above estimates and using it in (\ref{ep}), we have 
		\begin{align}
		-\varepsilon &\leq \int_{\tau}^T (\w_{\varepsilon}(t), -\Delta(\u_{\varepsilon}(t) - \u_d(t)))\d t + (\psi_{\varepsilon}(t) ,\varphi_{\varepsilon}(t) - \varphi_d(t))\d t\nonumber\\&\quad +(\w_{\varepsilon}(T),\u_{\e}(T)-\u_f)+ (\psi_{\varepsilon}(T),\varphi_{\e}(T)-\varphi_f)+ \frac{1}{2} \|\W\|^2 - \frac{1}{2} \|\U_{\varepsilon}(\tau)\|^2.
		\end{align}
		Using the adjoint system $(\p_{\e},\eta_{\e})$ given in (\ref{adj}), we deduce that 
		\begin{align*}
		-\varepsilon &\leq \int_{\tau}^T (\w_{\varepsilon}, -\partial_t\p_{\e} - \nu \Delta \p_{\e} +(\p_{\e} \cdot \nabla)\u_{\e} + (\u_{\e} \cdot \nabla)\p_{\e} - (\nabla \varphi_{\e})^{\top} \eta_{\e} +\nabla q_{\e}) \d t \\
		&\quad +\int_{\tau}^T (\psi_{\varepsilon},-\partial_t\eta_{\e}  +\J \ast (\p_{\e} \cdot \nabla \varphi_{\e}) - (\nabla{\J} \ast  \varphi_{\e})\cdot \p_{\e} + \nabla a\cdot \p_{\e} \varphi_{\e} 
		- \u_{\e} \cdot\nabla \eta_{\e} - a \Delta \eta_{\e} )\d t\\ &\quad +\int_{\tau}^{T}(\psi_{\varepsilon}, \J\ast \Delta \eta_{\e} - \F''(\varphi_{\e})\Delta \eta_{\e})\d t +(\w_{\varepsilon}(T),\u_{\e}(T)-\u_f)+ (\psi_{\varepsilon}(T),\varphi_{\e}(T)-\varphi_f) \\&\quad  + \frac{1}{2} \|\W\|^2 - \frac{1}{2} \|\U_{\varepsilon}(\tau)\|^2. 
		\end{align*}
		Since $\text{div }\p_{\e}=0$, an integration by parts yields 
		\begin{align*}
		-\e&\leq \int_{\tau}^T (\partial_t  \w_{\varepsilon} - \nu \Delta \w_{\varepsilon} + (\w_{\varepsilon} \cdot \nabla ) \u_{\e} + (\u_{\e} \cdot \nabla )\w_{\varepsilon} + \nabla\widetilde{\uppi}_{\w_{\varepsilon}} ,\p_{\varepsilon}) \d t \\
		&\quad+\int_{\tau}^T (-\nabla a\psi_{\varepsilon} \varphi_{\e} - (\J\ast \psi_{\varepsilon}) \nabla \varphi_{\e} - (\J \ast\varphi_{\e}) \nabla \psi_{\varepsilon} ,\p_{\varepsilon}) \d t  \\
		&\quad+\int_{\tau}^T (\partial_t\psi_{\varepsilon} + \w_{\varepsilon} \cdot \nabla \varphi_{\e} + \u_{\e} \cdot\nabla \psi_{\varepsilon} - \Delta \widetilde{\mu}_{\e},\eta_{\varepsilon}) - (\w_{\varepsilon}(T),\p_{\varepsilon}(T)) - (\psi_{\varepsilon}(T),\eta_{\varepsilon}(T))\\
		&\quad +(\w_{\varepsilon}(T),\u_{\e}(T)-\u_f) + (\psi_{\varepsilon}(T),\varphi_{\e}(T)-\varphi_f) + (\w_{\varepsilon}(\tau),\p_{\varepsilon}(\tau)) + (\psi_{\varepsilon}(\tau),\eta_{\varepsilon}(\tau))  \\
		&\quad + \frac{1}{2} \|\W\|^2 - \frac{1}{2} \|\U_{\varepsilon}(\tau)\|^2, \\
		&\leq  (\W -\U_{\varepsilon}(\tau),\p_{\varepsilon}(\tau)) + \frac{1}{2} \|\W\|^2 - \frac{1}{2} \|\U_{\varepsilon}(\tau)\|^2,
		\end{align*}
		where we have used the fact that $\w_{\varepsilon}(\tau) = \W -\U_{\varepsilon}(\tau)$ (see \eqref{lin w1}) and $\psi_{\varepsilon}(\tau) = 0$. This easily implies 
		\begin{align*}
		\frac{1}{2}\|\U_{\varepsilon}(t)\|^2 + (\p_{\varepsilon}(t),\U_{\varepsilon}(t) )\leq \frac{1}{2}\|\mathrm{W}\|^2 + (\p_{\varepsilon}(t),\mathrm{W}) + \varepsilon,
		\end{align*}
		for all $\mathrm{W} \in \G_{\text{div}}$ and a.e. $t\in[0,T]$.
	\end{proof}
	
	Now we prove the convergence results used in the proof of the above theorem in terms of two lemmas. The next lemma proves the strong convergence of solution of the system \eqref{nonlin phi}-\eqref{initial conditions} with spike control to the strong solution of the system \eqref{nonlin phi}-\eqref{initial conditions} with control $\U$. 
	
	\begin{lemma}\label{lem3.12}
		Let $(\u,\varphi)$ be the unique strong solution of \eqref{nonlin phi}-\eqref{initial conditions} corresponding to the control $\U$ and $(\u^h,\varphi^h)$ be the unique strong solution of the system \eqref{nonlin phi}-\eqref{initial conditions} corresponding to the control $\U^h$. If $\U^h \rightarrow \U$ strongly in $\mathrm{L}^2(0,T;\G_{\text{div}})$, then we have the following strong convergence: $$(\u^h,\varphi^h) \rightarrow (\u,\varphi)\ \text{  in }\ (\mathrm{L}^{\infty}(0,T;\G_{\text{div}}) \cap \mathrm{L}^2(0,T;\V_{\text{div}})) \times (\mathrm{L}^{\infty}(0,T;\mathrm{V}') \cap \mathrm{L}^2(0,T;\mathrm{H})).$$
	\end{lemma}
	\begin{proof}
		From the uniqueness theorem (see Theorem \ref{unique}),  we know that 
		\begin{align*}
		&\| \u^h(t) - \u(t) \|^2 + \| \varphi^h(t) - \varphi(t) \|^2_{\mathrm{V}'} 
		+ \int_0^t \left( \frac{C_0}{2} \| \varphi^h(s) - \varphi(s) \|^2 + \frac{\nu}{4} \|\nabla( \u^h(s) - \u(s)) \|^2 \right) \d s \\
		&\quad \leq C\int_0^t\| \U^h (s)- \U(s) \|^2\d s,
		\end{align*}
		where $C$ is a generic constant depending on the various norms of the solutions $(\u^h,\varphi^h)$, $(\u,\varphi)$ and the control $\upchi_{(\tau-h,\tau]}(t)(\W - \U(t))$. We can show that $C$ doesn't depent on $h$. For that let us estimate the following terms.
$$\int_0^T \|\upchi_{(\tau-h,\tau]}(t)(\W - \U(t))\|^2 \d t \leq \int_0^T \|(\W - \U(t))\|^2 \d t.$$	
\begin{align*}
\int_0^T \|\varphi^h(t)\|^2 \d t \leq \int_0^T \|\U^h(t)\|^2 \d t &\leq \int_{\tau}^{\tau-h} \|\W\|^2 \d t + \int_0^{\tau}\|\U(t)\|^2 \d t + \int_{\tau}^T \|\U(t)\|^2 \d t \\
& \qquad \leq h \|\W\|^2  + \int_0^{\tau}\|\U(t)\|^2 \d t + \int_{\tau}^T \|\U(t)\|^2 \d t \\
& \qquad\leq \tau \|\W\|^2  + \int_0^T\|\U(t)\|^2 \d t + \int_0^T \|\U(t)\|^2 \d t .
\end{align*}
		Similar as above we can show that $C$ does not depend on $h$. If $\U^h \rightarrow \U$ strongly in $\mathrm{L}^2(0,T;\G_{\text{div}})$, from the above relation, it can be easily seen that $\u^h \rightarrow \u$ strongly in $\mathrm{L}^{\infty}(0,T;\G_{\text{div}}) \cap \mathrm{L}^2(0,T;\V_{\text{div}})$ and $\varphi^h \rightarrow \varphi$ strongly in $\mathrm{L}^{\infty}(0,T;\mathrm{V}') \cap \mathrm{L}^2(0,T;\mathrm{H}).$ Hence, we have $\u^h\to \u$ uniformly in $\G_{\text{div }}$ for all $0\leq t\leq T$ and $\varphi^h\to\varphi$ uniformly in $\mathrm{V}'$ for all $0\leq t\leq T$. 
	\end{proof}
	
	\begin{remark}\label{rem3.15}
		By using Theorem 6, \cite{unique}, one can also show that 
		if $\U^h \rightarrow \U$ strongly in $\mathrm{L}^2(0,T;\G_{\text{div}})$, then we also have the following strong convergence: $$(\u^h,\varphi^h) \rightarrow (\u,\varphi)\ \text{  in }\ (\mathrm{L}^{\infty}(0,T;\G_{\text{div}}) \cap \mathrm{L}^2(0,T;\V_{\text{div}})) \times (\mathrm{L}^{\infty}(0,T;\mathrm{H}) \cap \mathrm{L}^2(0,T;\mathrm{V})).$$ Since from Theorem 6, \cite{unique}, we have 
		\begin{align}
		&\| \u^h(t) - \u(t) \|^2 + \| \varphi^h(t) - \varphi(t) \|^2 
		+ \int_0^t \left( \frac{C_0}{2} \| \nabla(\varphi^h(s) - \varphi(s) )\|^2 + \nu \|\nabla( \u^h(s) - \u(s)) \|^2 \right) \d s\nonumber \\
		&\quad \leq C\int_0^t\| \U^h (s)- \U(s) \|^2\d s,
		\end{align}
		for all $t\in[0,T]$, where the constant $C$ depending on various norms of $(\u^h,\varphi^h)$ and $(\u,\varphi)$ given in Theorem \ref{strongsol} and the convergence easily follows.
	\end{remark}
	
	We will now establish the convergence stated in \eqref{ep1} and \eqref{ep2}  as our next lemma. Let us define
	\begin{equation}\label{zvdef}
	\w^h := \frac{\u^h-\u}{h},\quad \psi^h := \frac{\varphi^h-\varphi}{h},
	\end{equation}
	where $(\u,\varphi)$ is the strong solution of \eqref{nonlin phi}-\eqref{initial conditions} corresponding to $\U$ and $(\u^h,\varphi^h)$ is the strong solution corresponding
	to the spike $\U^h$. We show that as $h\uparrow 0$, the respective limits $\w$ and $\varphi$ satisfy an appropriate linearized problem. Formally, one may write the system satisfied by $(\w^h$, $\psi^h)$ as  
	\begin{equation}\label{4.18a}
	\left\{
	\begin{aligned}
	\w^h_t - \nu \Delta \w^h  + (\w^h \cdot \nabla )\u + (\u \cdot \nabla )\w^h + \nabla \widetilde{\uppi}_{\w^h} &= - h(\w^h \cdot \nabla )\w^h - h\frac{\nabla a}{2} (\psi^h)^2  - \nabla a \psi^h \varphi \\&\quad- h(\J \ast \psi^h) \nabla \psi^h    -(\J\ast \psi^h) \nabla \varphi - (\J \ast \varphi) \nabla \psi^h  \\&\quad + \upchi_{(\tau-h,\tau]}(t)\frac{(\W - \U(t))}{h}, \ \text{ in } \ \Omega\times(0,T),\\
	\psi^h_t + h\w^h \cdot\nabla \psi^h + \w^h \cdot \nabla \varphi + \u \cdot\nabla \psi^h &= \Delta \mu^h, \ \text{ in } \ \Omega\times(0,T), \\ 
	\mu^h &= a \psi^h - \J\ast \psi^h + \F''(\varphi)  \psi^h,\ \text{ in } \ \Omega\times(0,T),  \\
	\text{div }\w^h &= 0, \ \text{ in } \ \Omega\times(0,T), \\
	\frac{\partial \mu^h}{\partial\mathbf{n}} = \mathbf{0}, \ \w^h&=\mathbf{0}, \ \text{ on } \ \partial \Omega \times (0,T),\\
	\w^h(0) = \mathbf{0}, \ \psi^h(0) &= 0, \ \text{ in } \ \Omega,
	\end{aligned}
	\right.
	\end{equation}
	where $\upchi_{(\tau-h,\tau]}(t)$ is the characteristic function on $\tau -h < t \leq \tau$. Also we can show that \eqref{4.18a} has a unique strong solution since $(\u^h,\varphi^h)$ and $(\u,\varphi)$ are the unique strong solutions of \eqref{nonlin phi}-\eqref{initial conditions} with controls  $\U^h$ and $\U$ respectively. For that it is enough if we show the r.h.s. of \eqref{4.18a} is in $\mathrm{L}^2(0,T;\G_{\text{div}})$.
	Let us estimate the following terms.  
	$$\int_0^T  \|(\w^h \cdot \nabla )\w^h\|^2 \d t \leq\int_0^T  \|\w^h\|_{\mathbb{L}^4}^2\|\nabla \w^h\|_{\mathbb{L}^4}^2 \d t \leq \frac{C}{2}\int_0^T  \|\w^h\|_{\mathbb{L}^4}^4 \d t + \frac{1}{2C}\int_0^T\|\nabla \w^h\|_{\mathbb{L}^4}^4 \d t< \infty.$$
	$$\int_0^T  \|(\J \ast \psi^h) \nabla \psi^h\|^2 \d t \leq\int_0^T \|\J\|_{\mathrm{L}^1}^2\|\psi^h\|_{\mathrm{L}^4}^2\|\nabla \psi^h\|_{\mathrm{L}^4}^2 \d t\leq \frac{C}{2}\int_0^T  \|\J\|_{\mathrm{L}^1}^4\|\psi^h\|_{\mathrm{L}^4}^4 \d t + \frac{1}{2C}\int_0^T\|\nabla \psi^h\|_{\mathrm{L}^4}^4 \d t< \infty$$
	$$\int_0^T \|\upchi_{(\tau-h,\tau]}(t)\frac{(\W - \U(t))}{h}\|^2 \d t \leq \int_\tau^{\tau -h} \|\frac{(\W - \U(t))}{h}\|^2 \d t \leq \frac{2}{h} \|\W\|^2 + \frac{2}{h^2} \|\U\|_{\mathrm{L}^2(0,T;\G_{\text{div}})} < \infty.$$
All the above terms are finite since $(\w^h,\psi^h)$ belongs to $(\mathrm{L}^{\infty}(0,T;\G_{\text{div}})\cap\mathrm{L}^2(0,T;\V_{\text{div}}))\times (\mathrm{L}^{\infty}(0,T;\mathrm{V})\cap \mathrm{L}^2(0,T;\mathrm{H}^2)).$	
	In the lemma \eqref{lem3.13} we show that taking the limit $h\uparrow 0$, we arrive at  
	\begin{eqnarray}\label{lin w1}
	\left\{
	\begin{aligned}
	\w_t - \nu \Delta \w + (\w \cdot \nabla )\u + (\u \cdot \nabla )\w + \nabla \widetilde{\uppi}_\w &= -\nabla a\psi {\varphi} -(\J\ast \psi) \nabla{\varphi} - (\J \ast{\varphi}) \nabla \psi \\
	&\quad +(\W -\U)\delta(t- \tau),\ \text{  in } \  \Omega \times (0,T), \\
	\psi_t + \w\cdot \nabla {\varphi} +{\u} \cdot\nabla \psi &= \Delta \widetilde{\mu},\ \text{  in } \  \Omega \times (0,T), \\ 
	\widetilde{\mu} &= a \psi - \J\ast \psi + \F''({\varphi})\psi, \\
	\text{div }\w &= 0,\ \text{  in } \  \Omega \times (0,T), \\
	\frac{\partial \widetilde{\mu}}{\partial\mathbf{n}} &= \mathbf{0}, \ \w=\mathbf{0},  \ \text{ on } \ \partial \Omega \times (0,T),  \\
	\w(0) &= \mathbf{0}, \ \psi(0) = 0, \ \text{ in } \ \Omega,
	\end{aligned}
	\right. 
	\end{eqnarray}
	where the singular term involving the Dirac delta $\delta(\cdot)$ corresponds to a jump due to the spike variation. Since the existence of a weak solution of the  limiting system \eqref{lin w1} is known, one can cast it in an abstract semigroup framework, where the singular term instead contributes as an initial data 
	at time $\tau$. We explore these observations  in the following theorem by taking ideas and methodology from \cite{HOF,HFSS,SDMTS}.

	\begin{lemma}\label{lem3.13}
		Let $(\u,\varphi)$ be the solution of the system \eqref{nonlin phi}-\eqref{initial conditions} corresponding to the control $\U$ and $(\u^h,\varphi^h)$ is the solution of \eqref{nonlin phi}-\eqref{initial conditions} corresponding to the spike control $\U^h$. Let $\tau \in (0,T]$ be a left lebegue point of the function $(\W -\U(t))$. Define $\w^h = \frac{\u^h - \u}{h}$ and $\psi^h = \frac{\varphi^h - \varphi}{h}$. Then $\w^h \rightarrow \w$ and $\psi^h \rightarrow \psi$ uniformly in time $\tau \leq t \leq T$ and strongly with respect to $\G_{\text{div}} \times \mathrm{H}$ convergence over $\Omega$. The limit $(\w,\psi)$ satisfies the following system,
		\begin{equation}\label{3.24}
		\begin{pmatrix}
		\w \\ 
		\psi
		\end{pmatrix}(t) =\begin{cases}
		(\mathbf{0},0)^{\top}, & \text{if $t \in 0 \leq t < \tau$},\\
		\mathcal{S}(t,\tau) \mathcal{C}_0(\tau), & \text{ if $\tau \leq t \leq T$},
		\end{cases}
		\end{equation}
		with
		\begin{equation*} 
		\mathcal{C}_0(\tau) = \begin{pmatrix}
		(\W - \U(\tau))\\ 0
		\end{pmatrix}
		\end{equation*}  
		where
		$\mathcal{S}(t,\tau)$ is the evolution operator of the system
		\begin{equation}\label{3.25}
		\left\{
		\begin{aligned}
		\w_t - \nu \Delta \w + (\w \cdot \nabla ){\u} + ({\u} \cdot \nabla )\w + \nabla \widetilde{\uppi}_\w &= -\nabla a\psi {\varphi} -(\J\ast \psi) \nabla {\varphi} - (\J \ast {\varphi}) \nabla \psi  \ \text{  in } \  \Omega \times (0,T), \\
		\psi_t + \w\cdot \nabla {\varphi} + {\u} \cdot\nabla \psi &= \Delta \widetilde{\mu}\ \text{  in } \  \Omega \times (0,T), \\ 
		\widetilde{\mu} &= a \psi - \J\ast \psi + \F''({\varphi})\psi . 
		\end{aligned}
		\right. 
		\end{equation}
	\end{lemma}
	\begin{proof}
		Let us define a bilinear operator $\mathcal{B}(\w,\psi) = (\mathcal{B}_1(\w,\psi),\mathcal{B}_2(\w,\psi)),$ where
		\begin{align*}
		\mathcal{B}_1(\w,\psi) &=   \nu \Delta \w - (\w \cdot \nabla )\u- (\u \cdot \nabla )\w - \nabla \widetilde{\uppi}_{\w}  - \nabla a \psi \varphi  -(\J\ast \psi) \nabla \varphi - (\J \ast \varphi ) \nabla \psi, \\
		\mathcal{B}_2(\w,\psi) &= -\w \cdot \nabla \varphi- \u \cdot\nabla \psi + \Delta (a \psi - \J\ast \psi + \F''(\varphi)  \psi) .
		\end{align*}
		We can write the system for $(\w^h,\psi^h)$ as following:
		\begin{equation}
		\left\{
		\begin{aligned}
		\frac{\d}{\d t}\begin{pmatrix}
		\w^h \\ 
		\psi^h
		\end{pmatrix} &= \mathcal{B} \begin{pmatrix}
		\w^h \\ 
		\psi^h
		\end{pmatrix} + \mathcal{C}_h(t) + \mathcal{F}_h(t) , \\
		\begin{pmatrix}
		\w^h \\ 
		\psi^h
		\end{pmatrix}(0) &= \begin{pmatrix}
		\mathbf{0}\\ 
		0
		\end{pmatrix},
		\end{aligned}
		\right.
		\end{equation}
		where 
		\begin{align*}
		\mathcal{C}_h(t) = \begin{pmatrix} 
		\upchi_{(\tau-h,\tau]}(t)\frac{(\W - \U(t))}{h}\\ 0
		\end{pmatrix}.
		\end{align*}
		and 
		\begin{align*}
		\mathcal{F}_h(t) = \begin{pmatrix} 
		- h(\w^h(t) \cdot \nabla)\w^h(t) - h\frac{\nabla a}{2}(\psi^h(t))^2  - h(\J \ast \psi^h(t)) \nabla \psi^h(t) \\
		- h\w^h \cdot \nabla \psi^h             
		\end{pmatrix}.
		\end{align*}
		Similarly the system \eqref{3.25} for $(\w,\psi)$ over $\tau \leq t \leq T$ can be written as 
		\begin{equation}
		\left\{
		\begin{aligned}
		\frac{\d}{\d t}\begin{pmatrix}
		\w \\ 
		\psi
		\end{pmatrix} &= \mathcal{B} \begin{pmatrix}
		\w \\ 
		\psi
		\end{pmatrix} , \\
		\begin{pmatrix}
		\w \\ 
		\psi
		\end{pmatrix}(\tau) &= \mathcal{C}_0(\tau).
		\end{aligned}
		\right.
		\end{equation}
		We know that for each $(\w_0,\psi_0)\in\G_{\text{div }}\times\mathrm{H},$ a weak solution $(\w(t),\psi(t))$ of (\ref{lin w1}) exists if and only if $\mathcal{B}$ is the generator of a strongly continuous semigroup $\{\mathcal{S}(t)\}$ of bounded linear operators on $\G_{\text{div }}\times\mathrm{H}$ (see Theorem 1, \cite{JMB}). 
		Observe that for $0 \leq t < \tau-h$,  $(\w,\psi)(t) = (\mathbf{0},0)$ and for $0\leq t<\tau-h$, we have $(\w^h,\psi^h)(t) = (\mathbf{0},0)$. 
		Next let us define 
		\begin{equation}\label{y}
		\y (t,h) := \begin{pmatrix}
		\w^h \\ 
		\psi^h
		\end{pmatrix}(t) - \begin{pmatrix}
		\w \\ 
		\psi
		\end{pmatrix}(t).
		\end{equation} 
		For any $0 \leq t < \tau -h$, $\mathcal{C}_h(t) = (\mathbf{0},0)^{\top}$. It follows that for $0\leq t < \tau -h$, we have $\y(t,h)\equiv 0$, since the control is same. Letting $h\uparrow 0$, we deduce the convergence
		for $0\le t< \tau$. 
		
		Now, for $\tau-h \leq t\leq T$, we know that 
		\begin{equation}
		\left\{
		\begin{aligned}
		\frac{\d}{\d t}\y (t,h) &= \mathcal{B} \y (t,h) + \mathcal{F}_h(t)  , \\
		\y (0,h) &= \begin{pmatrix}
		\mathbf{0}\\ 
		0
		\end{pmatrix}.
		\end{aligned}
		\right.
		\end{equation}
		By definition of the evolution operator $\mathcal{S}(t,s)$, we have 
		\begin{align}\label{yh}
		\begin{pmatrix}
		\w^h \\ 
		\psi^h
		\end{pmatrix}(t) &= \mathcal{S}(t,0)\begin{pmatrix}
		\mathbf{0} \\ 
		0 
		\end{pmatrix} + \int_0^t \mathcal{S}(t,s)\mathcal{C}_h(s) \d s + \int_0^t \mathcal{S}(t,s)\mathcal{F}_h(s) \d s  
		\nonumber\\&	= \int_0^t \mathcal{S}(t,s)\mathcal{C}_h(s) \d s + \int_0^t \mathcal{S}(t,s)\mathcal{F}_h(s) \d s
		\end{align}
		and 
		\begin{equation*}
		\begin{pmatrix}
		\w \\ 
		\psi
		\end{pmatrix}(t) = \mathcal{S}(t,\tau)\mathcal{C}_0(\tau).
		\end{equation*}
		Using (\ref{y}) and (\ref{yh}), we derive
		\begin{align*}
		\y(t,h) &= \int_0^t \mathcal{S}(t,s)\mathcal{C}_h(s) \d s + \int_0^t \mathcal{S}(t,s)\mathcal{F}_h(s) \d s - \mathcal{S}(t,\tau)\mathcal{C}_0(\tau) \\
		&= \frac{1}{h} \int_{\tau -h}^{\tau} (\mathcal{S}(t,s)\mathcal{C}_0(s) - \mathcal{S}(t,\tau)\mathcal{C}_0(\tau)) \d s  + \int_{\tau-h}^t \mathcal{S}(t,s)\mathcal{F}_h(s) \d s.
		\end{align*}
		Therefore, we estimate
		\begin{align*}
		\|\y(t,h)\|_{\G_{\text{div }}\times\mathrm{H}}
		&  \leq \frac{1}{h} \int_{\tau -h}^{\tau} \|\mathcal{S}(t,s)\mathcal{C}_0(s) - \mathcal{S}(t,\tau)\mathcal{C}_0(\tau) \|_{\G_{\text{div}} \times \mathrm{H}} \d s + \int_{\tau-h}^t \|\mathcal{S}(t,s)\mathcal{F}_h(s)\|_{\G_{\text{div}} \times \mathrm{H}} \d s \\
		&  \leq \frac{1}{h} \int_{\tau -h}^{\tau} \|\mathcal{S}(t,s)(\mathcal{C}_0(s) - \mathcal{C}_0(\tau)) \|_{\G_{\text{div}} \times \mathrm{H}}  \d s  \\&\quad+ \frac{1}{h} \int_{\tau -h}^{\tau} \|(\mathcal{S}(t,s) - \mathcal{S}(t,\tau))\mathcal{C}_0(\tau)\|_{\G_{\text{div}} \times \mathrm{H}} \d s  + \int_{\tau-h}^{t} \|\mathcal{S}(t,s) \|\|\mathcal{F}_h(s)\|_{\G_{\text{div}} \times \mathrm{H}} \d s \\
		& \leq \frac{1}{h} \int_{\tau -h}^{\tau} \|\mathcal{S}(t,s) \| \|\mathcal{C}_0(s) - \mathcal{C}_0(\tau) \|_{\G_{\text{div}} \times \mathrm{H}} \d s \\&\quad + \frac{1}{h} \int_{\tau -h}^{\tau} \|(\mathcal{S}(t,s) - \mathcal{S}(t,\tau))\mathcal{C}_0(\tau)\|_{\G_{\text{div}} \times \mathrm{H}} \d s   +\int_{\tau-h}^{t} \|\mathcal{S}(t,s) \|\|\mathcal{F}_h(s)\|_{\G_{\text{div}} \times \mathrm{H}} \d s \\
		&=  \sum_{j=1}^3 I_j.
		\end{align*}
		Note that $\|\mathcal{S}(\cdot,\cdot)\|$ denotes the operator norm of $\mathcal{S}(\cdot,\cdot)$. Taking $h \uparrow 0$, $I_1$ and $I_2$ converges to zero, since $\tau$ is the left Lebesgue point of $\mathcal{C}_0$ and the evolution operator $\mathcal{S}$ is strongly continuous in $\G_{\text{div}} \times \mathrm{H}$. We now have to show that $I_3 \rightarrow 0$ as $h\uparrow 0$. Since $\G_{\text{div }}\times\mathrm{H}$ is a closed subspace of $\V_{\text{div }}'\times\mathrm{V}'$, it is enough to check
		\begin{align*}
		\lim_{h \uparrow 0}\int_{\tau-h}^{t}\|\mathcal{S}(t,s) \|\|\mathcal{F}_h(s)\|_{\V_{\text{div}'} \times \mathrm{V}'} \d s=0.
		\end{align*}
		Using the H\"older and Ladyzhenskaya inequality, we obtain 
		\begin{align}
		\|(\w^h\cdot\nabla)\w^h\|_{\V_{\text{div}}'}\leq\sqrt{2} \|\w^h\|_{\G_{\text{div }}}\|\w^h\|_{\V_{\text{div}}}\leq \sqrt{\frac{2}{\lambda_1}}\|\w^h\|_{\V_{\text{div}}}^2.
		\end{align}
		Using the H\"older inequality and Gagliardo-Nirenberg inequality (see Lemma \ref{GNI}), we deduce 
		\begin{align}\label{1}
		\left\|\frac{\nabla a}{2} (\psi^h)^2\right\|_{\V_{\text{div}}'} 	&\leq \left\|\frac{\nabla a}{2} (\psi^h)^2\right\|_{\G_{\text{div}}} \leq \frac{1}{2}\|\nabla a\|_{\mathbb{L}^{\infty}}\|(\psi^h)^2\|_{\mathrm{H}}\leq C\|\psi^h\|_{\mathrm{L}^4}^2\nonumber\\&\leq C\left(\|\nabla\psi^h\|^{1/2}\|\psi^h\|^{1/2}+\|\psi^h\|\right)^2\nonumber\\&\leq C\left(\|\nabla\psi^h\|\|\psi^h\|+\|\psi^h\|^2\right)\leq C\left(\|\nabla\psi^h\|^2+\|\psi^h\|^2\right). 
		\end{align}
		An integration by parts,  the H\"older and Poincar\'e inequalities yields  
		\begin{align}
		\|(\J\ast \psi^h) \nabla \psi^h\|_{\V_{\text{div}}'} &\leq \|\nabla\J\|_{\mathbb{L}^1}\|\psi^h \|_{\mathrm{L}^{4}}^2\leq C\left(\|\nabla\psi^h\|^2+\|\psi^h\|^2\right).
		\end{align}
		Once again using an integration by parts, H\"older, Poincar\'e and Gagliardo-Nirenberg inequalities, we get (see \cite{BDM})
		\begin{align}\label{4}
		\|\w^h \cdot \nabla \psi^h \|_{\mathrm{V}'} &\leq C\|\w^h\|_{\mathbb{L}^{4}}\|\nabla\psi^h\|_{\mathrm{L}^4}\leq C\left(\|\w^h\|_{\mathbb{L}^{4}}^2+\|\nabla\psi^h\|_{\mathrm{L}^4}^2\right)\leq C\left(\|\w^h\|_{\V_{\text{div }}}^2+\|\nabla\psi^h\|^2+\|\psi^h\|^2\right).
		\end{align}	
		Combining (\ref{1})-(\ref{4}), we get 
		\begin{align}
		& \|\mathcal{F}_h(s) \|_{\V_{\text{div }'}\times\mathrm{V}'}\leq Ch\left(\|\w^h\|_{\V_{\text{div }}}^2+\|\nabla\psi^h\|^2+\|\psi^h\|^2\right).
		\end{align}
		Using Lemma \ref{lem3.12} and remark \ref{rem3.15}, we find 
		\begin{align}\label{3.45}
	&	\int_{\tau-h}^{t}\|\mathcal{S}(t,s) \|\|\mathcal{F}_h(s)\|_{\V_{\text{div }'}\times\mathrm{V}'}\d s\nonumber\\&\leq Ch\int_0^t\|\mathcal{S}(t,s)\|\left(\|\w^h(s)\|_{\V_{\text{div }}}^2+\|\nabla\psi^h(s)\|^2+\|\psi^h(s)\|^2\right)\d s\nonumber\\&\leq \frac{C}{h}\int_0^t\left(\|\nabla(\u^h(s)-\u(s))\|^2+\|\nabla(\varphi^h(s)-\varphi(s))\|^2+\|\varphi^h(s)-\varphi(s)\|^2\right)\d s\nonumber\\&\leq  \frac{C}{h}\int_{\tau-h}^{\tau}\|\U^h(s)-\U(s)\|^2\d s.
		\end{align}
		Since $\tau$ is the left Lebesgue point of the function $\W- \U(t)$, for sufficiently small $h$, we have
		\begin{align}\label{uval}
		\|\U^h(t)-\U(t)\|\leq Ch, \ \text{ for all }\ t\in[0,T].
		\end{align}
		Using \eqref{uval} in \eqref{3.45} we get
		\begin{align*}
		\int_{\tau-h}^{t}\|\mathcal{S}(t,s) \|\|\mathcal{F}_h(s)\|_{\V_{\text{div }'}\times\mathrm{V}'}\d s \leq Ch^2.
		\end{align*} Letting $h\uparrow 0,$ we deduce the convergence of $I_3$ to zero for $\tau\leq t< T$.  This completes the proof.
	\end{proof} 

\begin{remark}
	In the approximation of $\mu^h$ in \eqref{4.18a}, we used the Taylor series expansion for $0<\theta<1$ to get 
	\begin{align*}
	\frac{1}{h}[\mathrm{F}'(\varphi^h)-\mathrm{F}'(\varphi)]&=\frac{1}{h}[\mathrm{F}'(\varphi^h-\varphi+\varphi)-\mathrm{F}'(\varphi)]\nonumber\\&=\frac{1}{h}\left[\mathrm{F}'(\varphi)+(\varphi^h-\varphi)\mathrm{F}''(\varphi)+\frac{1}{2}(\varphi^h-\varphi)^2\mathrm{F}'''(\varphi+\theta(\varphi^h-\varphi))-\mathrm{F}'(\varphi)\right]\\&=\psi^h\mathrm{F}''(\varphi)+\frac{h}{2}(\psi^h)^2\mathrm{F}'''(\theta\varphi^h+(1-\theta)\varphi).
	\end{align*}
	Since $\mathrm{F}(\cdot)$ has a polynomial growth and $\varphi,\varphi^h\in\mathrm{L}^{\infty}(\Omega)$ for all $t\in[0,T]$, one can estimate the second term in the above equality as 
	\begin{align}\label{3.46}
	&\frac{h}{2}\int_{\tau-h}^{t}\|(\psi^h(s))^2\mathrm{F}'''(\theta\varphi^h(s)+(1-\theta)\varphi(s))\|\d s\nonumber\\& \leq \frac{h}{2}\int_{\tau-h}^{t}\|\mathrm{F}'''(\theta\varphi^h(s)+(1-\theta)\varphi(s))\|_{\mathrm{L}^{\infty}}\|(\psi^h(s))^2\|\d s\nonumber\\&\leq \frac{h}{2}\sup_{s\in[\tau-h,t]}\|\mathrm{F}'''(\theta\varphi^h(s)+(1-\theta)\varphi(s))\|_{\mathrm{L}^{\infty}}\int_{\tau-h}^{t}\|\psi^h(s)\|_{\mathrm{L}^4}^2\d s\nonumber\\&\leq \frac{C}{h}\sup_{s\in[0,T]}\|\Psi(\theta\varphi^h(s)+(1-\theta)\varphi(s))\|_{\mathrm{L}^{\infty}}\int_{\tau-h}^{t}\left(\|\varphi^h(s)-\varphi(s)\|^2+\|\nabla(\varphi^h(s)-\varphi(s))\|^2\right)\d s\nonumber\\&\leq \frac{CM(h)}{h}\int_{\tau-h}^{\tau}\|\U^h(s)-\U(s)\|^2\d s\leq CM(h)h^2,
	\end{align}
	by using \eqref{uval}. Here $\Psi(\cdot)$ is a polynomial of degree $r>1$ and note that 
	\begin{align*}
	\sup_{t\in[0,T]}\|\Psi(\theta\varphi^h(t)+(1-\theta)\varphi(t))\|_{\mathrm{L}^{\infty}}&\leq C(r) \sup_{t\in[0,T]}\left(\|\varphi^h(t)\|_{\mathrm{L}^{\infty}}^r+\|\varphi(t)\|_{\mathrm{L}^{\infty}}^r\right)\\&\leq C\left(\|\varphi_0\|_{\mathrm{H}^{2}},\|\u_0\|_{\V_{\text{div }}},\|\U\|_{\mathrm{L}^2(0,T;\mathbb{G}_{\text{div}})},\|\U^h\|_{\mathrm{L}^2(0,T;\mathbb{G}_{\text{div}})}\right)\\&=:M(h),
	\end{align*}
	and $\lim\limits_{h\uparrow 0}M(h)=C\left(\|\varphi_0\|_{\mathrm{H}^{2}},\|\u_0\|_{\V_{\text{div }}},\|\U\|_{\mathrm{L}^2(0,T;\mathbb{G}_{\text{div}})}\right)=:M,$ since $\U^h\to\U$ strongly in $\mathrm{L}^2(0,T;\mathbb{G}_{\text{div}})$ as $h\to 0$. Thus passing limit as $h\uparrow 0$ in \eqref{3.46}, we deduce the required convergence. 
\end{remark}
	
	\begin{remark}\label{rem 3.11}
		By virtue of Ekeland's variational principle, Theorem \ref{3.10} can be proved for more general problems where the existence of minimizer is not known or not needed. In the case of the cost functional which we are considering (see \eqref{cost}), a minimizer  of $\mathcal{J}$ exists (see Theorem \ref{optimal} below) and therefore Theorem \ref{main} holds true. 
	\end{remark}

	In order to complete the proof of Theorem \ref{main}, it remains to show the existence of optimal control $(\u^*,\varphi^*,\U^*)$. This is the content of the next theorem.	
	
	\begin{theorem}[Existence of an Optimal Triplet]\label{optimal}
		Let the Assumption \ref{prop of F and J} along with the condition (\ref{fes}) holds true and the initial data $(\u_0,\varphi_0)$ satisfying (\ref{initial}) be given. Then there exists at least one triplet  $(\u^*,\varphi^*,\U^*)\in\mathscr{A}_{\text{ad}}$  such that the functional $ \mathcal{J}(\u,\varphi,\U)$ attains its minimum at $(\u^*,\varphi^*,\U^*)$, where $(\u^*,\varphi^*)$ is the unique strong solution of \eqref{nonlin phi}-\eqref{initial conditions}  with the control $\U^*$.
	\end{theorem}
	
	\begin{proof}
		\textbf{Claim (1):} \emph{$\mathscr{A}_{\text{ad}}$ is nonempty.}
		If the control $\U=\textbf{0}$, then by the existence and uniqueness theorem (see Theorems \ref{exist} and \ref{unique}), a  unique weak solution $(\u,\varphi)$ exists. Since the condition (\ref{fes}) holds true and the initial data $(\u_0,\varphi_0)$ satisfies (\ref{initial}), the unique weak solution we obtained is also a strong solution. Hence, $\mathcal{J}(\u,\varphi,\textbf{0})<+\infty$ exists and $(\u,\varphi,\mathbf{0})$ belongs to $\mathscr{A}_{\text{ad}}$. Therefore the set $\mathscr{A}_{\text{ad}}$ is nonempty. 
		
		\noindent\textbf{Claim (2):} \emph{The existence of an optimal triplet $(\u^*,\varphi^*,\U^*)\in\mathscr{A}_{\text{ad}}$.}	
		Let us define $$\mathscr{J} := \inf \limits _{\U \in \mathscr{A}_{\text{ad}}}\mathcal{J}(\u,\varphi,\U).$$
		Since, $0\leq \mathscr{J} < +\infty$, there exists a minimizing sequence $\{\U_n\} \in \mathscr{U}_{\text{ad}}$ such that $$\lim_{n\to\infty}\mathcal{J}(\u_n,\varphi_n,\U_n) = \mathscr{J},$$ where $(\u_n,\varphi_n)$ is the unique strong solution of \eqref{nonlin phi}-\eqref{initial conditions} with the control $\U_n$ and 
		\begin{align}\label{initial1}
		\u_n(0)=	\u_0\in\V_{\text{div}}\ \text{ and }\ \varphi_n(0)=\varphi_0\in \mathrm{H}^2.
		\end{align}
		Without loss of generality, we assume that $\mathcal{J}(\u_n,\varphi_n,\U_n) \leq \mathcal{J}(\u,\varphi,\mathbf{0})$, where $(\u,\varphi,\mathbf{0})\in\mathscr{A}_{\text{ad}}$. From the definition of $\mathcal{J}(\cdot,\cdot,\cdot)$, this gives
		\begin{align}\label{bound}
		&\frac{1}{2} \int_0^T \|\nabla  (\u_n(t)-\u_d(t))\|^2 \d t+ \frac{1}{2} \int_0^T \|\varphi_n(t) -\varphi_d(t)\|^2 \d t+ \frac{1}{2} \int_0^T\|\U_n(t)\|^2\d t\nonumber\\& \leq \frac{1}{2} \int_0^T \|\nabla  (\u(t)-\u_d(t))\|^2\d t + \frac{1}{2} \int_0^T \|\varphi(t) -\varphi_d(t)\|^2\d t . \end{align}
		Since $\u,\u_d \in \mathrm{L}^{\infty}(0,T;\V_{\text{div}})$ and $\varphi,\varphi_d \in \mathrm{L}^{\infty}(0,T;\mathrm{H})$, from the above relation, it is clear that, there exist a $K>0$, large enough such that
		$$0 \leq \mathcal{J}(\u_n,\varphi_n,\U_n) \leq K < +\infty.$$
		In particular, there exists a large $C>0,$ such that
		$$ \int_0^T \|\U_n(t)\|^2 \d t \leq  C < +\infty .$$
		Therefore the sequence $\{\U_n\}$ is uniformly bounded in the space $\mathrm{L}^2(0,T;\G_{\text{div}})$. Since $(\u_n,\varphi_n)$ is a unique weak solution of the system \eqref{nonlin phi}-\eqref{initial conditions} with control $\U_n$, from the energy estimates,
		one can easily show that the sequence $\{\u_n\} $ is uniformly bounded in $\mathrm{L}^{\infty}(0,T;\G_{\text{div}})\cap \mathrm{L}^2(0,T;\V_{\text{div}})$ and $\{\varphi_n\} $ is uniformly bounded in $\mathrm{L}^{\infty}(0,T;\mathrm{H})\cap\mathrm{L}^2(0,T;\mathrm{V})$. 
		Hence, by using the Banach–Alaoglu Theorem, we can extract a subsequence $\{(\u_{n},\varphi_{n}, \U_n)\}$ such that 
		\begin{equation}\label{3.47}
		\left\{
	\begin{aligned} 
	\u_n&\xrightharpoonup{w^*}\u^*\text{ in } \mathrm{L}^{\infty}(0,T;\G_{\text{div}}), \\ \u_n &\rightharpoonup \u^*\text{ in  }\textrm{L}^2(0,T;\V_{\text{div}}),\\ 
	(\u_n)_t &\rightharpoonup \u^*_t\text{ in  }\textrm{L}^2(0,T;\V_{\text{div}}'),\\\varphi_n &\xrightharpoonup{w^*} \varphi^*\text{ in }\mathrm{ L}^{\infty}(0,T;\mathrm{H}),\\ 
	(\varphi_{n}) _t&\rightharpoonup \varphi^*_t\text{ in }\mathrm{ L}^2(0,T;\mathrm{V}'),\\ \U_n &\rightharpoonup  \U^*\text{ in  }\mathrm{L}^2(0,T;\G_{\text{div}}).
	\end{aligned}
	\right.
		\end{equation}
		A calculation similar to proof of Theorem \ref{exist} , [\cite{weak}, Theorem 2] and theorem, \ref{unique} [ \cite{unique} Theorem 2] and using the Aubin-Lion's compactness arguments and the convergence in \eqref{3.47}, we get  
		\begin{equation}
		\left\{
		\begin{aligned}
		\u_n&\to \u^*\text{ in }\mathrm{L}^{2}(0,T;\G_{\text{div}}), \text{ a. e. }\text{ in }\Omega\times(0,T),\\
		\varphi_n &\to \varphi^*\text{ in }\mathrm{ L}^{2}(0,T;\mathrm{H}), \text{ a. e. }\text{ in }\Omega\times(0,T).
		\end{aligned}
		\right. 
		\end{equation}
		Proceeding similarly as in Theorem 1, \cite{weak} and Theorem 2, \cite{unique}, we obtain  $(\u^*,\varphi^*)$ is a unique weak solution of \eqref{nonlin phi}-\eqref{initial conditions} with control $\U^*$. Note that the initial condition (\ref{initial1}) and (\ref{fes}) also imply that 
		\begin{equation*}
		\u_n\in \mathrm{L}^{\infty}(0,T;\V_{\text{div}})\cap \mathrm{L}^2(0,T;\H^2), \ \ \varphi_n\in \mathrm{L}^{\infty}((0,T) \times\Omega)\cap \mathrm{L}^{\infty}(0,T;\mathrm{V})).
		\end{equation*}
		and 
		\begin{equation*}\varphi_n \in\mathrm{L}^{\infty}(0,T;\mathrm{W}^{1,p}),\ 2\leq p<\infty.
		\end{equation*} 
		Thus, we have (see Theorem 2, \cite{strong} also)
	\begin{equation}
	\left\{
	\begin{aligned}
	\u_n&\xrightharpoonup{w^*}\u^*\ \text{ in }\ \mathrm{L}^{\infty}(0,T;\V_{\text{div}}), \\
	\u_n&\rightharpoonup\u^*\ \text{ in }\ \mathrm{L}^2(0,T;\H^2(\Omega)),\\
	(\u_n)_t&\rightharpoonup\u^*_t\ \text{ in }\ \mathrm{L}^{2}(0,T;\G_{\text{div}}), \\
	\varphi_n&\to\varphi^*\ \text{ a.e., }(x,t) \ \in\Omega\times(0,T), \\
	(\varphi_{n})_t&\rightharpoonup\varphi^*_t\ \text{ in }\ \mathrm{L}^{2}(0,T;\mathrm{H})
	\end{aligned}
	\right.
	\end{equation}
		Thus the above convergences and Remark \ref{rem2.12} imply that  $$\u^*\in\mathrm{C}([0,T];\V_{\text{div }}),\ \varphi^*\in\mathrm{C}([0,T];\mathrm{H}^2).$$ Hence, it is immediate that 
		\begin{align}\label{initial2}
		\u^*(0)=	\u_0\in\V_{\text{div}}\ \text{ and }\ \varphi^*(0)=\varphi_0\in \mathrm{H}^2,
		\end{align}
		by using the continuity. Hence $(\u^*,\varphi^*,\U^*)$ is a unique strong solution of  \eqref{nonlin phi}-\eqref{initial conditions} with control $\U^*\in\mathscr{U}_{\text{ad}}$. This easily implies  $(\u^*,\varphi^*,\U^*)\in \mathscr{A}_{\text{ad}}$. 
		
		\noindent \textbf{Claim (3):} $\mathscr{J}=\mathcal{J}(\u^*,\varphi^*,\U^*)$.	Recall that $\mathscr{U}_{\text{ad}}$ is a closed and convex subset of $\mathrm{L}^2(0,T;\G_{\text{div}})$. Since the cost functional $\mathcal{J}(\cdot,\cdot,\cdot)$ is continuous and convex on $\mathrm{L}^2(0,T;\V_{\text{div}}) \times \mathrm{L}^2(0,T;\mathrm{V}) \times \mathscr{U}_{\text{ad}}$, it follows that $\mathcal{J}(\cdot,\cdot,\cdot)$ is weakly lower semi-continuous (see Proposition 1, Chapter 5, \cite{AbEk}). That is, for a sequence 
		$$(\u_n,\varphi_n,\U_n)\xrightharpoonup{w}(\u^*,\varphi^*,\U^*)\text{ in }\mathrm{L}^2(0,T;\V_{\text{div}})\times \mathrm{L}^2(0,T;\mathrm{V})\times \mathrm{L}^2(0,T;\G_{\text{div}}),$$
		we have 
		\begin{align*}
		\mathcal{J}(\u^*,\varphi^*,\U^*) \leq  \liminf \limits _{n\rightarrow \infty} \mathcal{J}(\u_n,\varphi_n,\U_n).
		\end{align*}
		Therefore, we get
		\begin{align*}\mathscr{J} \leq \mathcal{J}(\u^*,\varphi^*,\U^*) \leq  \liminf \limits _{n\rightarrow \infty} \mathcal{J}(\u_n,\varphi_n,\U_n)=  \lim \limits _{n\rightarrow \infty} \mathcal{J}(\u_n,\varphi_n,\U_n) = \mathscr{J},\end{align*}
		and hence $(\u^*,\varphi^*,\U^*)$ is a minimizer.
	\end{proof}

	\section{Data Assimilation Problem }\label{se5}\setcounter{equation}{0}
	In this section, we consider a problem similar to the data assimilation problems of meteorology. In the data assimilation problems coming from meteorology determining the correct initial condition for the future predictions is the key step. Taking inspiration from this scenario, we can pose a problem to find the unknown optimal data initialization problem (see \cite{sritharan} for the case of incompressible Navier-Stokes equations).  
	Let us consider an initial data optimization problem for the system governed by non local Cahn-Hilliard-Navier-Stokes system.  We formulate the problem as finding the optimal initial velocity  $\U\in\V_{\text{div}}$ such that $(\u, \varphi, \U) $ satisfies the following system:
	\begin{equation}\label{5.1} 
	\left\{
	\begin{aligned}
	\varphi_t + \u\cdot \nabla \varphi &= \Delta \mu,\ \text{ in } \ \Omega\times(0,T),\\
	\mu &= a \varphi - \J\ast \varphi + \F'(\varphi), \ \text{ in } \ \Omega\times(0,T), \\
	\u_t - \nu \Delta \u + (\u\cdot \nabla )\u + \nabla \uppi &= \mu \nabla \varphi + \mathbf{h}, \ \text{ in } \ \Omega\times(0,T), \\
	\text{div }\u&= 0,\ \text{ in } \ \Omega\times(0,T), \\
	\frac{\partial \mu}{\partial\mathbf{n}} &= 0 \ , \u=0 \ \text{ on } \ \partial \Omega \times (0,T),\\
	\u(0) &= \U, \  \ \varphi(0) = \varphi _0 \ \text{ in } \ \Omega,
	\end{aligned}  
	\right.
	\end{equation}
	and minimizes the cost functional 
	\begin{equation}\label{cost2}
	\begin{aligned}
	\mathcal{J}(\u,\varphi,\U)&:=  \frac{1}{2} \|\U\|^2+\frac{1}{2} \int_0^T \|\u(t) - \u_M(t)\|^2 \d t+ \frac{1}{2} \int_0^T \|\varphi(t) -\varphi_M(t)\|^2 \d t\\ &\quad+\frac{1}{2}\|\u(T)-\u_M^f\|^2+\frac{1}{2}\|\varphi(T)-\varphi_M^f\|^2,
	\end{aligned}
	\end{equation}
	where $\u_M$ is the measured average velocity of the fluid and $\varphi_M$  is the measured difference of the concentration of two fluids. $\u_M^f$ and $\varphi_M^f$ are measured velocity and concentration at time T respectively. We assume that the desired states
	\begin{align}\label{um}\u_M\in\mathrm{L}^2(0,T;\V_{\text{div }}),\ \varphi_M\in\mathrm{L}^2(0,T;\mathrm{V}),\ \u_M^f\in\V_{\text{div }}\ \text{ and }\ \varphi_M^f\in\mathrm{V}.\end{align}
	From Theorem \ref{strongsol}, we know that a unique strong solution for the system (\ref{5.1}) exists only if the initial velocity belongs to $\V_{\text{div }}$.  So, in this case, we take the admissible control class, $\mathscr{U}_{\text{ad}}$ as $\V_{\text{div}}$.  The  \emph{admissible class} $\mathscr{A}_{\text{ad}}$ consists of all triples $(\u,\varphi,\U)$ such that the set of states $(\u,\varphi)$ is a unique strong solution of the system \eqref{5.1} with control $\U \in \mathscr{U}_{ad} $.  The optimal control problem can be then defined as:
	\begin{align} \label{IOCP}\tag{IOCP}
	\min_{(\u,\varphi,\U) \in \mathscr{A}_{\text{ad}} } \mathcal{J}(\u,\varphi,\U).
	\end{align}

	\subsection{Existence of an optimal control and maximum principle}
	We first show that an optimal triplet $(\u^*,\varphi^*,\U^*)$ exists for the problem \ref{IOCP}. The proof of the following theorem is similar to that of Theorem \ref{optimal}.
	\begin{theorem}[Existence of an Optimal Triplet]\label{optimal1}
		Let the Assumption \ref{prop of F and J} along with the condition (\ref{fes}) holds true and the initial data $(\U,\varphi_0)$ satisfying (\ref{initial}) be given, where $\U\in\mathscr{U}_{\text{ad}}$ is the control parameter. Then there exists at least one triplet  $(\u^*,\varphi^*,\U^*)\in\mathscr{A}_{\text{ad}}$  such that the functional $ \mathcal{J}(\u,\varphi,\U)$ attains its minimum at $(\u^*,\varphi^*,\U^*)$, where $(\u^*,\varphi^*)$ is the unique strong solution of \eqref{nonlin phi}-\eqref{initial conditions}  with the initial data control $\U^*\in\mathscr{U}_{\text{ad}}$.
	\end{theorem}
	Observe that since $\mathscr{U}_{\text{ad}}$ is a closed subset of $\V_{\text{div }}$, the condition in \eqref{initial2} yields $\U^*\in\mathscr{U}_{\text{ad}}$.

	As  in Theorem \ref{main} (see Theorem 4.7, \cite{BDM} also), the Pontryagin minimum principle follows (with some obvious modifications) for this control problem as well. That is, we have 
	\begin{align*}
	\frac{1}{2}\|\U^*\|^2 +\langle \p(0),\U^* \rangle \leq \frac{1}{2}\|\mathrm{W}\|^2 + \langle \p(0),\mathrm{W} \rangle,
	\end{align*}
	for all $\mathrm{W} \in \mathscr{U}_{\text{ad}}\subset\V_{\text{div}}$.
	The optimal control is given by $\U^*=-\p(0)$, where $\p(\cdot)$ is the solution of the following adjoint system   
	\begin{equation}\label{5.3}
	\left\{
	\begin{aligned}
	-\p_t- \nu \Delta \p +(\p \cdot \nabla^T)\u - (\u \cdot \nabla)\p - (\nabla \varphi)^{\top} \eta &= (\u-\u_M),\ \text{in } \  \Omega \times (0,T),\\
	-\eta_t   +\J \ast (\p \cdot \nabla \varphi) - (\nabla{\J} \ast  \varphi)\cdot \p+ \nabla a\cdot \p \varphi 
	- \u \cdot\nabla \eta - a \Delta \eta&\\ + \J\ast \Delta \eta - \F''(\varphi)\Delta \eta &= (\varphi - \varphi_M),\ \text{in } \  \Omega \times (0,T), \\ \text{div }\p&=0,\ \text{in } \  \Omega \times (0,T),\\ \p=0,\frac{\partial \eta}{\partial \mathbf{n}}&=\mathbf{0},\ \text{on } \  \partial\Omega \times (0,T),\\ \p(T,\cdot)&=\u(T)-\u_M^f,\ \text{in } \  \partial\Omega, \\ \eta(T,\cdot)&=\varphi(T)-\varphi_M^f\ \text{in } \  \partial\Omega.
	\end{aligned}
	\right.
	\end{equation}
	A similar calculation as in Theorem \ref{adjoint} yields the existence of a weak solution to the system (\ref{5.3}) such that  $$(\p,\eta)\in (\C([0,T];\G_{\text{div}})\cap\mathrm{L}^2(0,T;\V_{\text{div}}))\times (\C([0,T];\mathrm{H})\cap\mathrm{L}^2(0,T;\mathrm{V})).$$
	Using the continuity of $\p(\cdot)$ in time  at $t=0$ in $\G_{\text{div }}$, we know that $\p(0)\in\G_{\text{div}}$. But this is not good enough, as our optimal control is  $\U^*=-\p(0)$, and we require $\p(0)\in\mathscr{U}_{\text{ad}}\subset\V_{\text{div}}$.  Hence, we need the following regularity results for the adjoint system (\ref{5.3}).

	\begin{theorem}
		Let $(\u,\varphi)$ be a unique strong solution of the nonlinear system \eqref{nonlin phi}-\eqref{initial conditions} with $\U=0$. Then the  \emph{unique weak solution} to the system \eqref{5.3} satisfies \begin{align}
		(\p,\eta)\in(\C([0,T];\V_{\text{div}})\cap\mathrm{L}^2(0,T;\H^2))\times (\C([0,T]; \mathrm{V})\cap \mathrm{L}^2(0,T;\mathrm{H}^2)).\end{align} 
	\end{theorem}
	\begin{proof}
		To prove the higher regularity let us take inner product with $-\Delta\p$ to the first equation in (\ref{5.3}) to find 
		\begin{align}\label{p5}
		-\frac{1}{2}\frac{\d}{\d t}\|\nabla\p\|^2+\nu\|\Delta\p\|^2&=-((\p\cdot\nabla^T)\u,\Delta\p)+((\u\cdot\nabla)\p,\Delta\p)-(\eta,\nabla\varphi\cdot\Delta\p)+(\u-\u_M,\Delta\p)\nonumber\\&=:\sum_{i=1}^4I_i.
		\end{align}
		Using the H\"older, Ladyzhenskaya, Young's, Poincar\'e and Gagliardo-Nirenberg inequalities (see \eqref{gu1}), we estimate $|I_1|$ as  
		\begin{align}\label{p6}
		|I_1|&\leq \|\p\|_{\mathbb{L}^4}\|\nabla\u\|_{\mathbb{L}^4}\|\Delta\p\|\leq \frac{\nu}{8}\|\Delta\p\|^2+\frac{2}{\nu}\|\p\|_{\mathbb{L}^4}^2\|\nabla\u\|_{\mathbb{L}^4}^2\leq \frac{\nu}{8}\|\Delta\p\|^2+\frac{2\sqrt{2}}{\nu}\|\p\|\|\nabla\p\|\|\nabla\u\|_{\mathbb{L}^4}^2\nonumber\\&\leq \frac{\nu}{8}\|\Delta\p\|^2+C(\nu,\lambda_1)\|\u\|_{\mathbb{H}^2}^2\|\nabla\p\|^2.
		\end{align}
		Let us use the H\"older, Agmon's and Young's inequalities to estimate $|I_2|$ as
		\begin{align}\label{p7}
		|I_2|&\leq \|\u\|_{\mathbb{L}^{\infty}}\|\nabla\p\|\|\Delta\p\|\leq \frac{\nu}{8}\|\Delta\p\|^2+\frac{2}{\nu}\|\u\|_{\mathbb{L}^{\infty}}^2\|\nabla\p\|^2\leq \frac{\nu}{8}\|\Delta\p\|^2+C(\nu)\|\u\|_{\H^2}^2\|\nabla\p\|^2.
		\end{align}
		Using the H\"older, Gagliardo-Nirenberg and Young's inequalities, we estimate $|I_3|$ as 
		\begin{align}\label{p8}
		|I_3|&\leq \|\eta\|_{\mathrm{L}^4}\|\nabla\varphi\|_{\mathbb{L}^4}\|\Delta\p\|\leq \frac{\nu}{8}\|\Delta\p\|^2+\frac{2}{\nu}\|\eta\|_{\mathrm{L}^4}^2\|\nabla\varphi\|_{\mathbb{L}^4}^2\leq \frac{\nu}{8}\|\Delta\p\|^2+C(\nu)\left(\|\eta\|^2+\|\nabla\eta\|^2\right)\|\nabla\varphi\|_{\mathbb{L}^4}^2.
		\end{align}
		An application of the Cauchy-Schwarz and Young's inequalities yield
		\begin{align}\label{p9}
		|I_4|\leq \|\u-\u_M\|\|\Delta\p\|\leq \frac{\nu}{8}\|\Delta\p\|^2+\frac{2}{\nu}\|\u-\u_M\|^2.
		\end{align}
		Combining (\ref{p6})-(\ref{p9}) and substituting in (\ref{p5}) we get 
		\begin{align}\label{p10}
		&-\frac{1}{2}\frac{\d}{\d t}\|\nabla\p\|^2+\frac{\nu}{2}\|\Delta\p\|^2\nonumber\\&\leq \frac{2}{\nu}\|\u-\u_M\|^2+ C(\nu,\lambda_1)\|\u\|^2_{\H^2}\|\nabla\p\|^2 +C(\nu)\|\nabla\varphi\|_{\mathbb{L}^4}^2\|\eta\|^2+C(\nu)\|\nabla\varphi\|_{\mathbb{L}^4}^2\|\nabla\eta\|^2.
		\end{align}
		
		Let us now multiply the second equation in (\ref{5.3}) with $-\Delta\eta$ to get 
		\begin{align}\label{p11}
		&-\frac{1}{2}\frac{\d}{\d t}\|\nabla\eta\|^2+((a+\F''(\varphi))\Delta\eta,\Delta\eta)\nonumber\\&\quad=( \J \ast (\p \cdot \nabla \varphi),\Delta\eta) - ((\nabla{\J} \ast  \varphi)\cdot \p,\Delta\eta) + (\nabla a\cdot \p \varphi ,\Delta\eta)
		-( \u \cdot\nabla \eta,\Delta\eta) \nonumber\\&\qquad+( \J\ast \Delta \eta,\Delta\eta)+(\varphi - \varphi_M,\Delta\eta)\nonumber\\&\quad=:\sum_{i=5}^{10}I_i.
		\end{align}
		But we know that 
		\begin{align}\label{p12}
		C_0\|\Delta\eta\|^2\leq ((a+\F''(\varphi))\Delta\eta,\Delta\eta).
		\end{align}
		We estimate $|I_5 |$ using the H\"older, Ladyzhenskaya, Poincar\'e and Young's inequalities as
		\begin{align}\label{p13}
		|I_5 |&\leq \|\J*\p\|_{\mathbb{L}^4}\|\nabla\varphi\|_{\mathbb{L}^4}\|\Delta\eta\|\leq\frac{C_0}{12}\|\Delta\eta\|^2+\frac{3}{C_0} \|\J\|_{\mathrm{L}^1}^2\|\p\|_{\mathbb{L}^4}^2\|\nabla\varphi\|_{\mathbb{L}^4}^2\nonumber\\&\leq \frac{C_0}{12}\|\Delta\eta\|^2+\frac{3\sqrt{2}}{C_0} \|\J\|_{\mathrm{L}^1}^2\|\p\|\|\nabla\p\|\|\nabla\varphi\|_{\mathbb{L}^4}^2\nonumber\\&\leq \frac{C_0}{12}\|\Delta\eta\|^2+\frac{3}{C_0}\sqrt{\frac{2}{\lambda_1}}\|\J\|_{\mathrm{L}^1}^2\|\nabla\varphi\|_{\mathbb{L}^4}^2\|\nabla\p\|^2,
		\end{align}
		where we have also used the Young's inequality for convolutions.  Similarly $|I_6|$ can be estimated as
		\begin{align}\label{p14}
		|I_6|&\leq \|\nabla{\J} \ast  \varphi\|_{\mathbb{L}^4}\|\p\|_{\mathbb{L}^4}\|\Delta\eta\|
		\leq\frac{C_0}{12}\|\Delta\eta\|^2+\frac{3}{C_0} \|\nabla\J\|_{\mathbb{L}^1}^2\|\varphi\|_{\mathrm{L}^4}^2\|\p\|_{\mathbb{L}^4}^2
		\nonumber\\&\leq \frac{C_0}{12}\|\Delta\eta\|^2+\frac{3}{C_0}\sqrt{\frac{2}{\lambda_1}}\|\nabla\J\|_{\mathbb{L}^1}^2\|\varphi\|_{\mathrm{L}^4}^2\|\nabla\p\|^2.
		\end{align}
		Let us use H\"older, Ladyzhenskaya, Poincar\'e and Young's inequalities to estimate $|I_7|$ as
		\begin{align}\label{p15}
		|I_7|&\leq \|\nabla a\|_{\mathbb{L}^{\infty}}\|\p\|_{\mathbb{L}^4}\|\varphi\|_{\mathrm{L}^4}\|\Delta\eta\|\leq \frac{C_0}{12}\|\Delta\eta\|^2+\frac{3}{C_0}\|\nabla a\|_{\mathbb{L}^{\infty}}^2\|\p\|_{\mathbb{L}^4}^2\|\varphi\|^2_{\mathrm{L}^4}\nonumber\\&\leq \frac{C_0}{12}\|\Delta\eta\|^2+\frac{3}{C_0}\sqrt{\frac{2}{\lambda_1}}\|\nabla a\|_{\mathbb{L}^{\infty}}^2\|\varphi\|^2_{\mathrm{L}^4}\|\nabla\p\|^2.
		\end{align}
		Using the H\"older, Agmon (see \eqref{agm}) and Young's inequalities to estimate $|I_8|$ as 
		\begin{align}\label{p16}
		|I_8|&\leq \|\u\|_{\mathbb{L}^{\infty}}\|\nabla\eta\|\|\Delta\eta\|\leq C\|\u\|_{\H^2}\|\nabla\eta\|\|\Delta\eta\|\leq \frac{C_0}{12}\|\Delta\eta\|^2+C(C_0)\|\u\|_{\mathbb{H}^2}^2\|\nabla\eta\|^2.
		\end{align}
		We use an integration by parts, Cauchy-Schwarz inequality, convolution inequality and Young's inequality  to estimate $|I_9|$ as 
		\begin{align}\label{p17}
		|I_9|&=\left|\sum_{i=1}^2\int_{\Omega}\left(\int_{\Omega}\J(x-y)\Delta\eta(y)\d y\right)\frac{\partial^2\eta(x)}{\partial x_i^2}\d x\right|=\left|\sum_{i=1}^2\int_{\Omega}\frac{\partial}{\partial x_i}\left(\int_{\Omega}\J(x-y)\Delta\eta(y)\d y\right)\frac{\partial\eta(x)}{\partial x_i}\d x\right|\nonumber\\&=|(\nabla\J*\Delta\eta,\nabla\eta)|\leq \|\nabla\J\|_{\mathbb{L}^1}\|\Delta\eta\|\|\nabla\eta\|\leq  \frac{C_0}{12}\|\Delta\eta\|^2+\frac{3}{C_0}\|\nabla\J\|_{\mathbb{L}^1}^2\|\nabla\eta\|^2.
		\end{align}
		Finally, $|I_{10}|$ can be estimated using the Cauchy-Schwarz inequality and Young's inequality as 
		\begin{align}\label{p18}
		|I_{10}|\leq \|\varphi-\varphi_M\|\|\Delta\eta\|\leq \frac{C_0}{12}\|\Delta\eta\|^2+\frac{3}{C_0}\|\varphi-\varphi_M\|^2.
		\end{align}
		Combining (\ref{p12})-(\ref{p18}) and substituting it in (\ref{p11}) to find 
		\begin{align}\label{p19}
		-\frac{1}{2}\frac{\d}{\d t}\|\nabla\eta\|^2+\frac{C_0}{2}\|\Delta\eta\|^2&\leq C(C_0,\lambda_1)\left[\|\J\|_{\mathrm{L}^1}^2\|\nabla\varphi\|_{\mathbb{L}^4}^2+\left(\|\nabla\J\|_{\mathbb{L}^1}^2+\|\nabla a\|_{\mathbb{L}^{\infty}}^2\right)\|\varphi\|_{\mathrm{L}^4}^2\right]\|\nabla\p\|^2\nonumber\\&\quad +C(C_0)\left(\|\u\|_{\mathbb{H}^2}^2+\|\nabla\J\|_{\mathbb{L}^1}^2\right)\|\nabla\eta\|^2+\frac{3}{C_0}\|\varphi-\varphi_M\|^2.
		\end{align}
		Adding (\ref{p10}) and (\ref{p19}), we obtain  
		\begin{align}\label{p20}
		&-\frac{1}{2}\left(\frac{\d}{\d t}\|\nabla\p\|^2+\frac{\d}{\d t}\|\nabla\eta\|^2\right)+\frac{\nu}{2}\|\Delta\p\|^2+\frac{C_0}{2}\|\Delta\eta\|^2\nonumber\\&\quad \leq  C(\nu,\lambda_1,C_0)\left[\|\u\|^2_{\H^2}+\|\J\|_{\mathbb{L}^1}^2\|\nabla\varphi\|_{\mathbb{L}^4}^2+\left(\|\nabla\J\|_{\mathbb{L}^1}^2+\|\nabla a\|_{\mathbb{L}^{\infty}}^2\right)\|\varphi\|_{\mathrm{L}^4}^2\right]\|\nabla\p\|^2+C(\nu)\|\nabla\varphi\|_{\mathbb{L}^4}^2\|\eta\|^2 \nonumber\\&\qquad+C(\nu,C_0)\left[\|\nabla\varphi\|_{\mathbb{L}^4}^2+\left(\|\u\|_{\mathbb{H}^2}^2+\|\nabla\J\|_{\mathbb{L}^1}^2\right)\right]\|\nabla\eta\|^2+\frac{2}{\nu}\|\u-\u_M\|^2+\frac{3}{C_0}\|\varphi-\varphi_M\|^2.
		\end{align}
		Let us integrate the inequality from $t$ to $T$ to find 
		\begin{align}\label{p21}
		&\|\nabla\p(t)\|^2+\|\nabla\eta(t)\|^2+\nu\int_t^T\|\Delta\p(t)\|^2+C_0\int_t^T\|\Delta\eta(t)\|^2\d t\nonumber\\&\leq \|\nabla\p(T)\|^2+\|\nabla\eta(T)\|^2+\frac{2}{\nu}\int_t^T\|\u(s)-\u_M(s)\|^2\d s+\frac{3}{C_0}\int_t^T\|\varphi(s)-\varphi_M(s)\|^2\d s\nonumber\\&\quad +C(\nu,\lambda_1,C_0) \int_t^T\left[\|\u(s)\|^2_{\H^2}+\|\J\|_{\mathrm{L}^1}^2\|\nabla\varphi(s)\|_{\mathbb{L}^4}^2+\left(\|\nabla\J\|_{\mathbb{L}^1}^2+\|\nabla a\|_{\mathbb{L}^{\infty}}^2\right)\|\varphi(s)\|_{\mathrm{L}^4}^2\right]\|\nabla\p(s)\|^2\d s\nonumber\\&\quad +C(\nu,C_0)\int_t^T\left[\|\nabla\varphi(s)\|_{\mathbb{L}^4}^2+\left(\|\u(s)\|_{\mathbb{L}^4}^4+\|\nabla\J\|_{\mathbb{L}^1}^2\right)\right]\|\nabla\eta(s)\|^2\d s+C(\nu)\int_t^T\|\nabla\varphi(s)\|_{\mathbb{L}^4}^2\|\eta(s)\|^2\d s.
		\end{align}
		An application of the Gr\"onwall inequality in (\ref{p21}) yields 
		\begin{align}\label{p22}
		&\|\nabla\p(t)\|^2+\|\nabla\eta(t)\|^2\nonumber\\&\leq \left[\|\nabla(\u(T)-\u_M^f)\|^2+\|\nabla(\varphi(T)-\varphi_M^f)\|^2+\frac{2}{\nu}\left(\int_0^T\|\u(t)\|^2\d t+\int_0^T\|\u_M(t)\|^2\d t\right)\right.\nonumber\\&\left.\qquad +\frac{3}{C_0}\left(\int_0^T\|\varphi(t)\|^2\d t+\int_0^T\|\varphi_M(t)\|^2\d t\right)+C(\nu)\sup_{0\leq t\leq T}\|\eta(t)\|^2\int_0^T\|\nabla\varphi(t)\|_{\mathbb{L}^4}^2\d t\right]\nonumber\\&\quad\times \exp\left(\int_0^T\|\u(t)\|^2_{\H^2}\d t+\|\J\|_{\mathrm{L}^1}^2\int_0^T\|\nabla\varphi(t)\|_{\mathbb{L}^4}^2\d t+\left(\|\nabla\J\|_{\mathbb{L}^1}^2+\|\nabla a\|_{\mathbb{L}^{\infty}}^2\right)\int_0^T\|\varphi(t)\|_{\mathrm{L}^4}^2\d t\right)\nonumber\\&\quad \times\exp\left(\int_0^T\|\nabla\varphi(t)\|_{\mathbb{L}^4}^2\d t+\int_0^T\|\u(t)\|_{\mathbb{H}^2}^2\d t+\|\nabla\J\|_{\mathbb{L}^1}^2T\right),
		\end{align}
		for all $t\in[0,T]$. Note that the right hand side of the above inequality is  finite, since $(\u,\varphi)$ is the unique strong solution of the system (\ref{5.1}), $\u_M,\u_M^f,\varphi_M$ and $\varphi_M^f$ satisfy \eqref{um} and $\eta\in\mathrm{L}^{\infty}(0,T;\mathrm{H})$. Thus, we have $\p\in\mathrm{L}^{\infty}(0,T;\V_{\text{div}})$ and $\eta\in\mathrm{L}^{\infty}(0,T;\mathrm{V})$. Substituting (\ref{p22}) in (\ref{p21}), we also obtain 
		\begin{align}\label{p23}
		\|\nabla\p(t)\|^2+\|\nabla\eta(t)\|^2+\nu\int_t^T\|\Delta\p(t)\|^2+C_0\int_t^T\|\Delta\eta(t)\|^2\d t\leq C,
		\end{align}
		for all $t\in[0,T]$. Hence, we have $\p\in\mathrm{L}^2(0,T;\H^2)$ and $\eta\in\mathrm{L}^2(0,T;\mathrm{H}^2)$. Now it is immediate that  $\p_t\in\mathrm{L}^2(0,T;\G_{\text{div}})$ and $\eta_t\in\mathrm{L}^2(0,T;\mathrm{H})$.
		Thus $\p$ is almost everywhere equal to a continuous function from $[0,T]$ to $\V_{\text{div}}$ and $\eta$ is almost everywhere equal to a continuous function from $[0,T]$ to $\mathrm{V}$. That is, $\p\in\C([0,T];\V_{\text{div}})$ and $\eta\in\C([0,T];\mathrm{V})$. 
	\end{proof}
	Using the continuity of $\p(\cdot)$ at $0$, we have $\p(0)\in\V_{\text{div}}.$
	Hence, a similar calculation as in the proof of Theorem \ref{main} leads to attain the optimal initial velocity as 
	$$\U^*=-\p(0)\in\mathscr{U}_{\text{ad}}\subset\V_{\text{div}}.$$ Thus, we summarize the above analysis as 
	
	\begin{theorem}[Optimal Initial Control]\label{data}
		Let $(\u^*,\varphi^*,\U^*)\in\mathscr{A}_{\text{ad}}$ be an optimal triplet. Then there exists a unique weak solution $$(\p,\eta)\in(\C([0,T];\V_{\text{div}})\cap\mathrm{L}^2(0,T;\H^2))\times( \C([0,T];\mathrm{V})\cap \mathrm{L}^2(0,T;\mathrm{H}^2)),$$ of the adjoint system \eqref{5.3} such that the optimal control is obtained as $\U^*=-\p(0)\in\mathscr{U}_{\text{ad}}\subset\V_{\text{div}}.$
	\end{theorem}

\end{document}